\documentclass[11pt,reqno]{amsart}
\usepackage{mathrsfs}
\usepackage[american]{babel}
\usepackage{todonotes}
\usepackage{comment}
\usepackage{amsmath}
\usepackage{dsfont,mathtools,amssymb}
\usepackage{appendix}
\usepackage[usenames,dvipsnames]{xcolor}
\usepackage{nicematrix}
\usepackage[hidelinks]{hyperref}
\mathtoolsset{showonlyrefs,showmanualtags}
\usepackage[shortlabels]{enumitem}
\usepackage[margin=1.2in]{geometry}
\allowdisplaybreaks

\theoremstyle{theorem}
\newtheorem{theorem}{Theorem}[section]
\newtheorem{proposition}[theorem]{Proposition}
\newtheorem{lemma}[theorem]{Lemma}
\newtheorem{corollary}[theorem]{Corollary}
\theoremstyle{remark}
\newtheorem{remark}[theorem]{Remark}
\theoremstyle{definition}
\newtheorem{definition}[theorem]{Definition}

\numberwithin{equation}{section}
\numberwithin{figure}{section}

\renewcommand{\P}{\mathds{P}}

\newcommand{\R}{\mathbb{R}}

\newcommand{\Z}{\mathbb{Z}}
\newcommand{\N}{\mathbb{N}}
\newcommand{\dd}{{\rm d}}
\newcommand{\fstop}{\; \text{.}}
\newcommand{\comma}{\; \text{,}\;\;}

\newcommand{\tonde}[1]{\left(#1\right)}

\newcommand{\sqbra}[1]{\left[#1\right]}
\newcommand{\ttonde}[1]{\big(#1\big)}

\newcommand{\tttonde}[1]{(#1)}
\newcommand{\ttquadre}[1]{[#1]}
\newcommand{\abs}[1]{\left\lvert#1\right\rvert}

\newcommand{\ttabs}[1]{\lvert#1\rvert}

\newcommand{\emparg}{{\,\cdot\,}}

\newcommand{\eqdef}{\coloneqq}
\newcommand{\defeq}{\eqqcolon}
\newcommand{\car}{\mathbf{1}}

\newcommand{\norm}[1]{\left\lVert#1\right\rVert}

\newcommand{\set}[1]{\left\{#1\right\}}							

\newcommand{\ttset}[1]{\{#1\}}
 
\newcommand{\eps}{\varepsilon} 

\newcommand{\Var}{{\mathbb V}{\textsc{ar}}}
\newcommand{\pp}{\mathsf p}

\newcommand{\cB}{\ensuremath{\mathcal B}} 
\newcommand{\cC}{\ensuremath{\mathcal C}} 
\newcommand{\cD}{\ensuremath{\mathcal D}} 
 
\newcommand{\cF}{\ensuremath{\mathcal F}} 
 
\newcommand{\cH}{\ensuremath{\mathcal H}}

\newcommand{\cL}{\ensuremath{\mathcal L}} 
\newcommand{\cM}{\ensuremath{\mathcal M}}

\newcommand{\cR}{\ensuremath{\mathcal R}} 
\newcommand{\cS}{\ensuremath{\mathcal S}} 
 
\newcommand{\cU}{\ensuremath{\mathcal U}} 
\newcommand{\cV}{\ensuremath{\mathcal V}} 
\newcommand{\cW}{\ensuremath{\mathcal W}} 
\newcommand{\cX}{\ensuremath{\mathcal X}} 
\newcommand{\cY}{\ensuremath{\mathcal Y}} 
\newcommand{\cZ}{\ensuremath{\mathcal Z}}

\newcommand{\x}{\mathbf x}

\newcommand{\y}{\mathbf y}

\newcommand{\Q}{{\mathbb Q}}

\newcommand{\z}{{\mathbf z}}
\renewcommand{\u}{{\mathbf u}}

\renewcommand{\P}{\mathbb{P}}	

\newcommand{\E}{\mathbb{E}}

\newcommand{\mE}{\mathbf{E}}

\renewcommand{\hat}{\widehat}

\newcommand{\ind}{\mathbf{1}}

\title[The critical KPZ scale for the averaging process]{The critical KPZ scale for the averaging process}
\author{Hindy Drillick}
\address{New York University, Courant Institute of Mathematical Sciences, New York, NY, USA}
\email{hindy.drillick@nyu.edu}
\author{Shalin Parekh}
\address{University of Maine, Department of Mathematics, Orono, ME, USA}
\email{shalin.parekh@maine.edu}
\author{Federico Sau}
\address{University of Milan\\
	 Department of Mathematics \textquotedblleft Federigo Enriques\textquotedblright\\
	 Milan, Italy}
\email{federico.sau@unimi.it}

\begin{document}

	\maketitle
	
	\begin{abstract}
			KPZ-type extremal fluctuations have recently been proved for several models of random walks in space-time random environments (RWRE) in $1+1$ dimensions. A general moment criterion predicts the spatial scale at which this behavior should occur, but does not by itself guarantee non-trivial fluctuations at that scale.
			
			In this paper, we show that the averaging process provides an instance in which this criterion is not sharp: because of a degeneracy in the update mechanism, the actual KPZ limit appears only beyond the predicted scale. The relevant critical contribution arises from the interplay of two distinct fluctuation mechanisms, a phenomenon that appears to be rather special within the RWRE setting. 
            
            Our proof builds on Dobrushin-type local-time limit theorems for zero-sum additive functionals, refined estimates for tilted $k$-point motions, and the recent moment-based axiomatic characterization of Cole--Hopf solutions to the one-dimensional KPZ equation. We also identify the behavior on the two sides of the critical scale, thereby sharply separating the subcritical, critical, and supercritical regimes of the model.
		\end{abstract}
		\thispagestyle{empty}

\section{Introduction, model and main result}\label{sec:intro}

Let $(S_t)_{t\ge0}$ be a random walk on $\Z$ in a space-time random
environment (RWRE). We take the time variable to be continuous, and we assume only nearest-neighbor jumps. Given a realization $\omega$ of the environment, let
\begin{equation}
	P^\omega(t,x)
	\eqdef
	\mathsf P^\omega(S_t=x\mid S_0=0)\comma
	\qquad t\ge 0\comma x\in\Z\comma
\end{equation}
denote its quenched transition probability. Assume that the environment is
stationary under space-time shifts and independent over disjoint time intervals, so
that the annealed transition probabilities form a deterministic, translation-invariant
Markov semigroup. To fix ideas, we take this semigroup to be that of
 the continuous-time simple symmetric random walk (CTRW) $(R_t)_{t\ge0}$ on $\Z$ with
total jump rate one. Hence, letting $\P$ and $\E$ denote probability and expectation
with respect to the environment, we have
\begin{equation}\label{eq:quenched-annealed}
	\E[P^\omega(t,x)]=\pp_t(x)\comma
\end{equation}
where $\pp_t(x):=\mathsf P(R_t=x\mid R_0=0)$.
The identity in \eqref{eq:quenched-annealed} suggests the form
\begin{equation}\label{eq:P=pexp}
	P^\omega(t,x)
	=
	\pp_t(x)\,e^{H(t,x)}\comma
\end{equation}
where $H(t,x) \in [-\infty,\infty) $ is a model-dependent random correction, and normalized by
\begin{equation}\label{eq:normalization}
	\E[e^{H(t,x)}]=1\fstop
\end{equation}

Limit theorems for $P^\omega(t,x)$ depend both on the scaling relation between
$t$ and $x=x(t)$, and on the corresponding asymptotics of the annealed kernel
$\pp_t(x)$ and the quenched correction $e^{H(t,x)}$. The standard asymptotics of the simple CTRW
kernel $\pp_t(x)$ lead to the classical trichotomy between the \textit{bulk} or \textit{central limit theorem (CLT)}
regime, the \textit{large-deviation} regime, and the \textit{moderate-deviation} regime interpolating between them. These correspond, respectively, to
\begin{equation}\label{eq:thresholds-annealed}
	|x| \lesssim t^{1/2}\comma
	\qquad
	|x|\asymp t\comma
	\qquad
	|x|\asymp t^\lambda\comma \lambda \in (1/2,1)\fstop
\end{equation}
This transition is naturally captured by a Cramér's exponential tilting, which
removes the cost of reaching the point under consideration.

Fix $\lambda\le 1$ and set, for all $N\in \N$,
\begin{equation}\label{eq:beta-lambda-N}
	\beta_{\lambda,N}
	\eqdef
	\operatorname{arcsinh}(N^{\lambda-1})\comma
\end{equation}
so that $N\sinh(\beta_{\lambda,N})=N^\lambda$.
Letting $\lfloor a \rfloor\in \Z$ denote the floor of $a \in \R$, define 
\begin{equation}\label{eq:D-tilt}
	\hat\vartheta_{\lambda,N}(t,z)
	\eqdef
	\exp\set{
		\beta_{\lambda,N}\lfloor tN^\lambda+zN^{1/2}\rfloor	
		-
		tN(\cosh(\beta_{\lambda,N})-1)
	}\comma
	\qquad t>0\comma z\in \R\fstop
\end{equation}
This is the exponential tilt which makes the point
$\lfloor tN^\lambda+zN^{1/2}\rfloor \in \Z$ diffusive for the tilted annealed walk. Indeed, the tilted kernel
\begin{equation}
	p_{\lambda,N}(tN,y)
	\eqdef
	\exp\set{
		\beta_{\lambda,N}y
		-
		tN(\cosh(\beta_{\lambda,N})-1)
	}
	\pp_{tN}(y)\comma\qquad t\ge 0\comma y \in \Z\comma
\end{equation}
has mean $tN\sinh(\beta_{\lambda,N})=tN^\lambda$ and variance
$tN\cosh(\beta_{\lambda,N})$. Therefore,  the tilted local CLT gives, as $N\to \infty$,
\begin{equation}\label{eq:tilted-LCLT-general}
	N^{1/2}\,\hat\vartheta_{\lambda,N}(t,z)\,
	\pp_{tN}(\lfloor tN^\lambda+zN^{1/2}\rfloor)
	=
	\frac{1+o(1)}{\sqrt{2\pi t\cosh(\beta_{\lambda,N})}}
	\exp\set{
		-\frac{z^2}{2t\cosh(\beta_{\lambda,N})}
	}\fstop
\end{equation}
In particular, if $\lambda<1$, then $\beta_{\lambda,N}\sim N^{\lambda-1}$, and hence
\begin{equation}\label{eq:tilt-CLT}
	N^{1/2}\,\hat\vartheta_{\lambda,N}(t,z)\,
	\pp_{tN}(\lfloor tN^\lambda+zN^{1/2}\rfloor)
	\xlongrightarrow[N\to \infty]{}
	\frac{e^{-z^2/2t}}{\sqrt{2\pi t}}\defeq \mathsf g(t,z)\fstop
\end{equation}
Finally, observe that the cost $\log \hat\vartheta_{\lambda,N}(t,z)$ diverges as $N\to \infty$ if and only if $\lambda>1/2$.

\subsection{Fluctuation regimes}\label{sec:fluctuation-regimes}

The asymptotics of the random factor $e^{H(t,x)}$ in \eqref{eq:P=pexp} are subtler, and need not be organized by the same thresholds in \eqref{eq:thresholds-annealed} for the annealed kernel $\pp_t(x)$.

Following 	 \cite{thierry_ledoussal_exact_2016,ledoussal_thiery_diffusion_2017,barraquand_ledoussal_moderate,parekh2024hierarchy,drillick_parekh_random_2025}, let
$
	\lambda_c=\lambda_c(\P)\in(1/2,1)
$
be the conjectural unique critical exponent for the random-environment correction. More precisely, $\lambda_c$ should separate the perturbative regime
\begin{equation}
	e^{H(t,x)}\sim 1\comma
\end{equation}
from the non-perturbative one, where the full exponential correction survives. The heuristic picture of this phase transition is the following.

\smallskip \noindent
{\em Subcritical regime.} For $\lambda<\lambda_c$,  the random correction is thought of as perturbative. Hence, as $N\to \infty$, the linearization
\begin{equation}\label{eq:linearization-heuristics}
	e^{H(tN,\lfloor tN^\lambda +zN^{1/2}\rfloor)}\sim 1+H(tN,\lfloor tN^\lambda +zN^{1/2}\rfloor)\comma\qquad t>0\comma z \in \R\comma
\end{equation}
is somehow justified. In this regime, one expects additive Gaussian fluctuations around the tilted annealed profile: for a suitable centered Gaussian field $\cW_\lambda$ and some exponent $\chi_\lambda>0$, 
\begin{equation}\label{eq:conv-subcritical}
	N^{\chi_\lambda}
	\set{
		N^{1/2}\hat\vartheta_{\lambda,N}(t,z)\,
		P^\omega(tN, \lfloor tN^\lambda +zN^{1/2}\rfloor)
		-
		\mathsf g(t,z)
	}
	\xLongrightarrow[N\to \infty]{}
	\cW_\lambda(t,z)\comma
\end{equation}
where $\hat\vartheta_{\lambda,N}(t,z)$ is the annealed exponential tilt given in \eqref{eq:D-tilt}, which satisfies, for all $\lambda<1$, the tilted local CLT in \eqref{eq:tilt-CLT}.

As in \cite{drillick_parekh_random_2025}, the convergence in \eqref{eq:conv-subcritical} is meant in law as random distributions. Further, $\cW_\lambda$ is expected to solve an additive stochastic heat equation (SHE), whose nature may depend on $\lambda<\lambda_c$: for $\lambda\le1/2$,  the noise  is typically conservative; for $\lambda>1/2$, one expects a non-conservative additive-noise limit.

\smallskip \noindent
{\em Critical regime.}
At the critical scale $\lambda=\lambda_c$, the noise starts to become relevant:
\begin{equation}\label{eq:W-O(1)}
	H(tN,\lfloor tN^{\lambda_c}+zN^{1/2}\rfloor)\asymp 1\comma
	\qquad t>0\comma z\in\R\fstop
\end{equation}
Hence, the linearization in \eqref{eq:linearization-heuristics} is no longer justified, and the full exponential correction survives. As in \cite{das_drillick_parekh_multiplicative_2026,parekh2024hierarchy}, without centering, the expected limit remains stochastic, but it is 	multiplicative rather than additive:
\begin{equation}
	N^{1/2}\hat\vartheta_{\lambda_c,N}(t,z)\,
	P^\omega(tN,\lfloor tN^{\lambda_c}+zN^{1/2}\rfloor)
	\xLongrightarrow[N\to \infty]{}
	\cY(t,z)\comma
\end{equation}
where $\cY$ solves the {multiplicative} SHE
\begin{equation}\label{eq:mSHE}
	\partial_t\cY(t,z)
	=
	\tfrac12\partial_z^2\cY(t,z)
	+
	\gamma\,\cY(t,z)\,\xi(t,z)\comma
	\qquad
	\cY(0,\cdot)=\delta_0\comma
\end{equation}
for some noise coefficient $\gamma\ge0$ and space-time white noise $\xi$.
Equivalently, $\cH\eqdef \log \cY$
 is the Cole-Hopf solution to the Kardar--Parisi--Zhang (KPZ) equation \cite{kardar_parisi_zhang_dynamic_1986} 
\begin{equation}\label{eq:kpz}
	\partial_t \cH(t,z) = \tfrac12\partial_z^2 \cH(t,z) + \tfrac12\tttonde{\partial_z \cH(t,z)}^2+\gamma\,\xi(t,z)
\end{equation}
with narrow-wedge initial condition.

If the coefficient $\gamma$ in \eqref{eq:mSHE} is non-degenerate, i.e., strictly positive, then we can view $\lambda_c$ as the \textit{true} critical exponent at which the multiplicative SHE (or KPZ equation) arises, turning the heuristic relation in \eqref{eq:W-O(1)} into a robust definition of the critical regime.

\smallskip \noindent
{\em Supercritical regime.}
If $\lambda>\lambda_c$, then the noise term in \eqref{eq:P=pexp} is expected to be no longer perturbative, exhibiting large fluctuations as $N\to\infty$. However, by the normalization in \eqref{eq:normalization}, typical environments cannot carry a positive exponential growth, as the expectation is fixed at one. Rather, the expectation is sustained by large, though rare, random values of $H(tN,\lfloor tN^\lambda +zN^{1/2}\rfloor)$. This suggests the emergence of intermittent behavior.

Accordingly, in this regime it is more natural to write
\begin{equation}\label{eq:supercritical-W-decomposition}
	H(tN,\lfloor tN^\lambda +zN^{1/2}\rfloor)
	=
	-j_{\lambda,N}(t,z)
	+
	V_{\lambda,N}(t,z)\comma
	\qquad t>0\comma z\in\R\comma
\end{equation}
where $j_{\lambda,N}(t,z)$ is a positive and diverging deterministic centering term, while $V_{\lambda,N}(t,z)$ is random, expected to typically exhibit smaller fluctuations.

This regime remains mostly conjectural. Besides $V_{\lambda,N}$ being typically smaller than $j_{\lambda,N}$, we conjecture
\begin{equation}\label{eq:j}
	j_{\lambda,N}(t,z)\asymp t N^{\frac{\lambda-\lambda_c}{1-\lambda_c}}\comma\qquad \lambda \in [\lambda_c,1]\fstop
\end{equation}
Together with
 $
 	\log \pp_{tN}(tN^\lambda+zN^{1/2})\asymp -t N^{2\lambda-1}$ and the decomposition
 	\begin{align}
 		\log P^\omega(tN,\lfloor tN^\lambda +zN^{1/2}\rfloor)
 		&=
 		\log \pp_{tN}(\lfloor N^\lambda t+zN^{1/2}\rfloor)
 		-
 		j_{\lambda,N}(t,z)
 		+
 		V_{\lambda,N}(t,z)\comma
 	\end{align}
 the asymptotics in \eqref{eq:j} would imply that, a.s.,  
 \begin{equation}\label{eq:log-1st}
 \log P^\omega(tN,\lfloor tN^\lambda+zN^{1/2}\rfloor)\sim \log \pp_{tN}(\lfloor tN^\lambda+zN^{1/2}\rfloor)\comma\qquad \lambda <1\fstop	
 \end{equation}
 
 Remarkably,   \cite{yilmaz_large_2009,yilmaz_differing_2010} proved that the first-order approximation in \eqref{eq:log-1st} is \textit{not} valid for a large class of RWREs when $\lambda=1$. Rigorous fluctuation limit theorems are still scarce. At the large-deviation scale $\lambda=1$, Tracy--Widom type fluctuations have so far been proved only for certain exactly solvable RWRE models \cite{barraquand_corwin_random_2017,barraquand_large_2020}. Let us also mention the recent work \cite{groisman2026universality}, which, in the different setting of static two-dimensional RWRE, establishes GUE Tracy--Widom fluctuations for quenched large deviations in directions close to a coordinate axis, for a broad class of i.i.d.\ elliptic environments.

\subsection{Critical exponent}
For a class of one-dimensional RWREs with random quenched first moment,  \cite{thierry_ledoussal_exact_2016,ledoussal_thiery_diffusion_2017,barraquand_ledoussal_moderate,das_drillick_parekh_multiplicative_2026} identify the critical exponent
\begin{equation}
	\lambda_c=3/4\fstop
\end{equation}

However, as first considered in \cite{hass2025super} in a discrete-time setting (i.e., the environment law is stationary and independent across distinct integer time intervals), there are many other models whose quenched first moment is deterministic, and it is only the higher quenched moments that are truly random. Following the predictions made in \cite{hass2025super}, the recent works  \cite{parekh2024hierarchy,drillick_parekh_random_2025} establish a  sufficient condition for identifying candidate critical exponents in a large class of discrete-time RWRE and stochastic-flow models. Suppose that the first $p-1$  moments of the random kernel $P^\omega(t,x)$ are deterministic, while the $p$-th  moment is genuinely random, that is, for all integer times $t>0$,
\begin{equation}\label{eq:random-moments}
\Var\big(\textstyle{\sum_{x\in \Z}}\, P^\omega(t,x)\,x^q\big)=0\comma \text{for all}\  1\le q< p \comma
\qquad
\Var\big(\textstyle{\sum_{x\in \Z}}\,P^\omega(t,x)\,x^p\big)>0\comma
\end{equation}
where the variance is taken with respect to the environment law $\P$.  Then, the candidate critical exponent is
\begin{equation}
	\lambda_{c,p}:=1-(4p)^{-1}\fstop
\end{equation}
Under suitable mixing and irreducibility assumptions, these works show that
\begin{equation}
	\lambda_c\ge \lambda_{c,p}\fstop
\end{equation}
More precisely, for all $\lambda <\lambda_{c,p}$, 
one is in a subcritical Gaussian regime \cite{drillick_parekh_random_2025}, while at
$\lambda=\lambda_{c,p}$, 
the tilted density field converges---no recentering is required---to the multiplicative SHE, with an explicit noise coefficient $\gamma$ \cite{parekh2024hierarchy}. This framework applies to many models, including generalized RWREs with spatially correlated kernels, random landscape models, diffusion in random media, sticky Brownian motions, and the Kipnis--Marchioro--Presutti (KMP) model; see \cite{barraquand2025convergence} for the application of this framework to the KMP model.

However, the validity of the general theory outlined so far does not guarantee that the multiplicative SHE limit at $\lambda=\lambda_{c,p}$ is non-degenerate: one may have $\gamma=0$. This leaves open the  question:
\begin{equation}
	\text{Is it always true that}\  \lambda_c=\lambda_{c,p}\ 
	\text{for the suitable}\  p\in \N\ \text{?}
\end{equation}
We answer this question in the negative. More specifically, our work studies the first example of an RWRE for which the genuinely random moment is the first one, but the actual critical exponent is
\begin{equation}
	\lambda_c=\lambda_{c,2}=\frac78>\lambda_{c,1}=\frac34\comma
\end{equation}
and describes its critical picture.

\subsection{Averaging process}\label{sec:averaging-process}

The model we analyze is the \textit{averaging process}. This process has received growing attention in the recent literature, both as a gossip algorithm and as an interacting particle system,   for which convergence rates to equilibrium and scaling limits have been investigated; see, e.g., \cite{boyd_et_al_randomized_2006,aldous_lecture_2012,chatterjee2020phase,quattropani2021mixing,caputo_quattropani_sau_cutoff_2023,movassagh_repeated2022,sau_concentration_2024,sau_tiny_2024,gantert_vilkas_averaging_2025,elboim_peres_peretz_edge_2025,campos_hydrodynamic_2025}.

Its description is simple. Attach to every unoriented nearest-neighbor bond $\{x,x+1\}$ of $\Z$ an independent Poisson clock of rate one. The state of the process at time $t$ is a mass profile
\begin{equation}
	\eta_t=(\eta_t(x))_{x\in\mathbb Z}\in [0,\infty)^{\mathbb Z}\fstop
\end{equation}
When the clock on the bond $\{x,x+1\}$ rings, the two masses $\eta(x)$ and $\eta(x+1)$ are replaced by their average:
\begin{equation}
	\eta(x)\longmapsto \frac{\eta(x)+\eta(x+1)}{2}\comma
	\qquad
	\eta(x+1)\longmapsto \frac{\eta(x)+\eta(x+1)}{2}\fstop
\end{equation}
All other coordinates are left unchanged.  Thus, the redistribution rule itself is deterministic; the only randomness in
the dynamics is the space-time set of bonds selected by the Poisson clocks.

A key observation is that the averaging process can be viewed as a random walk in a space-time random environment and hence can be studied using the framework described above. We start the process from a unit mass at the origin,
\begin{equation}\label{eq:IC-eta}
	\eta_0(x)=\car_{\{x=0\}}\fstop
\end{equation}
Then $\eta_t(x)$ may be viewed as a quenched transition probability in a random dynamic environment: conditionally on the Poisson clocks, a particle can cross a bond only when the corresponding clock rings, and at each such mark it chooses with probability $1/2$ whether to cross. In this representation,
\begin{equation}\label{eq:eta-P}
	\eta_t(x)=P^\omega(t,x)\comma\qquad t\ge 0\comma x \in \Z\fstop
\end{equation}
The annealed law is the CTRW from \eqref{eq:quenched-annealed}; thus,
\begin{equation}\label{eq:eta-E}
	\mathbb E[\eta_t(x)]=\pp_t(x)\fstop
\end{equation}

At first glance, this model should fall under the $p=1$ theory, since the quenched first moment is random, see \eqref{eq:random-moments}. Indeed, for every $t>0$, $\Var(\sum_{x\in \Z} \eta_t(x)\,x)>0$ as, on the one hand, the annealed first moment---the mean displacement of the symmetric CTRW on $\Z$---equals
\begin{equation}
\E\big[\textstyle{\sum_{x\in \Z}}\,\eta_t(x)\,x\big]=\textstyle{\sum_{x\in \Z}}\,\pp_t(x)\,x = 0\comma
\end{equation}
whereas, on the other hand, letting $A_t$ denote the event that exactly one update involving the origin of $\Z$ occurs in the time interval $[0,t]$, 
\begin{equation}
\E\big[\big(\textstyle{\sum_{x\in \Z}}\,\eta_t(x)\,x\big)^2\big]\ge \E\big[\big(\textstyle{\sum_{x\in \Z}}\,\eta_t(x)\,x\big)^2\, \car_{A_t}\big] =   \frac14\,\P(A_t)>0\fstop
\end{equation}
This shows $\lambda_c\ge\lambda_{c,1}=3/4$.  Nevertheless, the multiplicative SHE coefficient at this candidate scale vanishes (see Section \ref{sec:remarks} for further details), thus suggesting $\lambda_c>3/4$. The first non-degenerate critical behavior occurs instead at
\begin{equation}
	\lambda_c=\frac78=\lambda_{c,2}\fstop
\end{equation}
The main goal of this work is to prove this prediction and to sharply identify the critical fluctuation theory for the averaging process.

\subsection{Main results}\label{sec:main-results}

As in Section \ref{sec:fluctuation-regimes}, we state our results at microscopic
times of order $N$, around spatial locations of order $N^\lambda$, and after
diffusive spatial averaging on a window of size $N^{1/2}$. Hence, for $\lambda\in[0,1]$ and $N\in\N$, define the normalized quenched density field associated to the averaging process, recalling \eqref{eq:eta-P}, by
\begin{equation}\label{eq:tilted-field}
	\cY_{\lambda,N}(t,\emparg)
	\eqdef
	\sum_{x\in\Z}
	\vartheta_{\lambda,N}(t,x)\,
	\eta_{tN}(x)\,
	\delta_{N^{-1/2}(x-tN^\lambda)}\comma
	\qquad t\ge 0\comma
\end{equation}
where $\vartheta_{\lambda,N}$ is the deterministic tilt (cf.\ \eqref{eq:D-tilt})
\begin{equation}\label{eq:tilt-micro}
	\vartheta_{\lambda,N}(t,x)\eqdef 
	 \exp\set{\beta_{\lambda,N}x-tN\tonde{\cosh(\beta_{\lambda,N})-1}}\comma\qquad t\ge 0\comma x \in \Z\fstop
\end{equation}
We view $\cY_{\lambda,N}(t,\emparg)$ as an $\mathscr M_f(\R)$-valued random
variable, where $\mathscr M_f(\R)$ is the space of finite Borel measures on $\R$
with the weak topology induced by $\cC_b(\R)$, the space of bounded continuous
functions. The initial condition \eqref{eq:IC-eta} gives,
deterministically,
\begin{equation}
	\cY_{\lambda,N}(0,\emparg)=\delta_0\fstop
\end{equation}

\subsubsection{Critical regime}

Our main result captures, through a full pathwise scaling limit, the averaging process' critical KPZ exponent at
$\lambda_c=\lambda_{c,2}=7/8$. In what follows, for a given Polish space $E$, $\cD([0,\infty);E)$ (resp.\ $\cC([0,\infty);E)$) denotes the space of càdlàg (resp.\ continuous) $E$-valued paths, equipped with the $J_1$-Skorokhod (resp.\ uniform) topology; $\cD([0,\infty);E)$ and $\cC([0,\infty);E)$ are Polish spaces; see, e.g., \cite{billingsley_convergence_1999}. The  space $\mathscr M_f(\R)$ is Polish. Throughout the paper, the symbol $\Longrightarrow$ stands for convergence in law.

\begin{theorem}\label{thm:main}
	Let $\lambda=7/8$. Then, we have the following convergence in law
	\begin{equation}
		\cY_{\lambda,N}
		\xLongrightarrow[N\to\infty]{}
		\cY\comma
		\qquad
		\text{in}\  \cD([0,\infty);\mathscr M_f(\R))\fstop
	\end{equation}
	Here, $\cY=\cY(t,z)$ is the unique continuous It\^o--Walsh solution to the
	multiplicative SHE
	\begin{equation}\label{eq:mSHE-theorem}
		\partial_t\cY(t,z)
		=
		\tfrac12\partial_z^2\cY(t,z)
		+
		{\tfrac{1}{\sqrt 2}}\,\cY(t,z)\,\xi(t,z)\comma
		\qquad
		\cY(0,\emparg)=\delta_0\comma
	\end{equation}
	where $\xi=\xi(t,z)$ is a standard space-time white noise on
	$\R_+\times\R$.
\end{theorem}

\begin{remark}
	The topology in Theorem \ref{thm:main} is natural for the non-interpolated
	fields. For fixed $N$ and $\lambda\le 1$, the paths of
	$t\longmapsto \cY_{\lambda,N}(t,\emparg)$ have jumps, due to the Poisson
	updates of the averaging process, and the spatial profiles are atomic
	measures on $N^{-1/2}\Z$. Thus, without additional time or space
	interpolation/regularization, one cannot expect convergence in a continuous-path or
	continuous-function topology.
\end{remark}

  In words, Theorem \ref{thm:main} states that, at $\lambda=7/8$, the tilted
fields retain non-trivial fluctuations, and converge as processes to a multiplicative SHE. In agreement with the discussion in  Section \ref{sec:fluctuation-regimes}, this  identifies the critical KPZ scale of the averaging process.
We refer to Section \ref{sec:remarks} for further comments on this convergence result, on the exponent $7/8$, and
on the origin of the simple-looking noise coefficient ${1/\sqrt 2}$ in
\eqref{eq:mSHE-theorem}.

\subsubsection{Non-critical regimes}
As a by-product of our analysis, we also locate the two sides of the critical
scale at the level of second moments. This gives a sharp separation between the
subcritical regime, where the tilted field concentrates around its annealed
profile, and the supercritical regime, where an intermittent behavior emerges.

Let us first record the corresponding first-moment behavior. By
\eqref{eq:eta-E}, for every $\lambda\le 1$, $t>0$ and
$\phi\in\cC_b(\R)$,
\begin{equation}
	\E[\cY_{\lambda,N}(t,\phi)]
	=
	\sum_{x\in \Z}
	\vartheta_{\lambda,N}(t,x)\,
	\pp_{tN}(x)\,
	\phi(N^{-1/2}(x-tN^\lambda))\comma
\end{equation}
and the tilted central limit theorem \eqref{eq:tilt-CLT} yields, for every
$\lambda<1$,
\begin{equation}
	\E[\cY_{\lambda,N}(t,\phi)]
	\xrightarrow[N\to\infty]{}
	\int_\R \mathsf g(t,z)\,\phi(z)\,\dd z\fstop
\end{equation}
At the endpoint $\lambda=1$, the tilted CLT in
\eqref{eq:tilted-LCLT-general} still gives a finite limit. Thus, the normalization
keeps the first moment of the field at order one throughout the whole range
$\lambda\le 1$.

The following result identifies, through second-moment asymptotics, the
subcritical and supercritical regimes. Here, $\cC_b^+(\R)\subset \cC_b(\R)$  stands for the subset of nonnegative functions. 

\begin{theorem}\label{th:non-critical}
	Let $\lambda< 1$ with $\lambda \neq 7/8$. Then for every $t>0$ and every 
	$0\neq\phi\in\cC_b^+(\R)$,
	\begin{equation}
		\Var(\cY_{\lambda,N}(t,\phi))
		\xrightarrow[N\to\infty]{}
		\begin{dcases}
			0 &\text{if}\ \lambda<7/8\comma\\
			\infty &\text{if}\ \lambda>7/8\comma
		\end{dcases}
	\end{equation}
	where the variance is taken with respect to the environment law $\P$.
\end{theorem}

\begin{remark}[Concentration vs.\ intermittency]
As the first moment of the tilted fields remains bounded, the two sides of the threshold correspond, respectively, to the subcritical
and supercritical regimes.
	  For $\lambda<7/8$, the collapse of the variance shows that the tilted fields
	\textit{concentrate} around their annealed profile.
	For $\lambda>7/8$, the divergence of
	the second moment instead signals an \textit{intermittent behavior}: for every $t>0$ and every 
	$0\neq\phi\in\cC_b^+(\R)$, 
	\begin{equation}
		\E[\cY_{\lambda,N}(t,\phi)^{2(1+\eps)}]
		\gg
		\E[\cY_{\lambda,N}(t,\phi)^{1+\eps}]^2\comma\qquad \eps>0\fstop
	\end{equation}
	For the proof of these asymptotics, see, e.g.,  \cite[Remark 1.14]{caravenna2025singularity}.
\end{remark}

The main goal of this paper is the full description of the critical scale. We therefore leave further refinements on these non-critical regimes to future work; see also Sections \ref{sec:fluctuation-regimes} and \ref{sec:questions} for more details and conjectures. 
Here, we only record that, while the supercritical
picture for the averaging process, and more generally for non-exactly solvable
RWRE models, remains largely open, the subcritical fluctuation regime is well understood in view of the recent work \cite{drillick_parekh_random_2025}. Indeed, by an adaptation of the analysis carried out there, one may prove additive-noise scaling limits
for the centered fields. These results are not needed in the proof of
Theorem \ref{thm:main}, and we state them here only as a complement to the
critical picture.

 In the range $\lambda\in(1/2,7/8)$, a
 non-conservative additive SHE arises in the limit: the centered field
 \begin{equation}
 	N^{7/4-2\lambda}
\set{
 	\cY_{\lambda,N}
 	-
 	\E[\cY_{\lambda,N}]
 }
 \end{equation}
 converges, in a suitable distribution-valued path space, to the additive SHE of the form \begin{equation}\label{eq:aSHE-non-cons}
 	\partial_t\cV(t,z)
 	=
 	\tfrac12\partial_z^2\cV(t,z)
 	+
 	\sigma_\lambda\,\mathsf g(t,z)\,\xi(t,z)\comma
 	\qquad
 	\cV(0,\emparg)=0\fstop
 \end{equation}
 Here, $g$ is the Gaussian kernel appearing in \eqref{eq:tilt-CLT},
 $\xi$ is a standard space-time white noise, and $\sigma_\lambda>0$ is a noise coefficient.
 
 At smaller scales, $\lambda\le 1/2$, the subcritical fluctuations are of
 bulk type. This regime should be understood as the full-line analogue of the
 main result in \cite{sau_tiny_2024}. Thus, for $\lambda<1/2$,
 \begin{equation}
 	N^{3/4}
 \set{
 	\cY_{\lambda,N}
 	-
 	\E[\cY_{\lambda,N}]
 } \xLongrightarrow[N\to \infty]{} \cU\comma
 \end{equation}
 where $\cU$ solves a 
 conservative additive SHE of the form (cf.\ the notation in \eqref{eq:aSHE-non-cons})
 \begin{equation}
 	\partial_t \cU(t,z) = \tfrac12\partial_z^2\cU(t,z) + \sigma_\lambda\,\partial_z(\partial_z \mathsf g(t,z)\,\xi(t,z))\comma\qquad \cU(0,\emparg)=0\fstop
 \end{equation}
 At $\lambda=1/2$, the same bulk limit is obtained only after a non-trivial
 shearing; see, for instance, the first item in \cite[Theorem 1.13]{drillick_parekh_random_2025}.

\subsection{Further remarks and technical challenges}\label{sec:remarks}

\subsubsection{Comparison with the ${\rm KMP}$ model}\label{sec:KMP} Let us further comment on the value of the critical exponent. A first hint that
the identity $\lambda_c=\lambda_{c,1}$ should fail for the averaging process
comes from the generalized KMP model studied in \cite{barraquand2025convergence}. Fix $\alpha>0$. In the KMP
model $\eta^{(\alpha)}$ with parameter $\alpha$, each nearest-neighbor bond rings at rate one;
when the bond $\{x,x+1\}$ rings, the total mass
$\eta^{(\alpha)}(x)+\eta^{(\alpha)}(x+1)$ is redistributed according to an independent
${\rm Beta}(\alpha,\alpha)$ random variable $Y\in [0,1]$, namely
\begin{equation}
	(\eta^{(\alpha)}(x),\eta^{(\alpha)}(x+1))
	\longmapsto
	(
	Y(\eta^{(\alpha)}(x)+\eta^{(\alpha)}(x+1)),
	(1-Y)(\eta^{(\alpha)}(x)+\eta^{(\alpha)}(x+1))
	)
	\fstop
\end{equation}
Thus, the averaging process is the degenerate redistribution limit
$\alpha\to\infty$, since
\begin{equation}
	{\rm Beta}(\alpha,\alpha)
	\Longrightarrow
	\delta_{1/2}
\comma
\end{equation}
and hence
$\eta^{(\alpha)}$ yields $\eta$ as $\alpha\to \infty$.
For all  $\alpha>0$, including the limit case $\alpha\to \infty$, the first quenched spatial moment is genuinely
random:
\begin{equation}
	\Var\big(
	\textstyle{\sum_{x\in\mathbb Z}}\, \eta_t^{(\alpha)}(x)\,x
	\big)
	>
	0
	\comma
\end{equation}
so the general moment criterion in \cite{parekh2024hierarchy} predicts the candidate exponent
$\lambda_{c,1}=3/4$. In \cite{barraquand2025convergence}, the corresponding
multiplicative SHE coefficient is computed explicitly and equals
\begin{equation}
	\gamma^{(\alpha)}
	=
	\frac{1}{2\sqrt{\alpha}}
	\fstop
\end{equation}
This coefficient vanishes as $\alpha\to\infty$. Thus, the averaging process
sits exactly at the boundary of the $p=1$ theory: the first quenched moment is
random, but the corresponding KPZ coefficient degenerates in the averaging
limit.

\subsubsection{Local smoothness of the averaging process} 	A more precise guess for the critical exponent is obtained from the martingale
decomposition of the tilted field. Fix $1/2<\lambda<1$, and recall \eqref{eq:tilted-field}. Then, for all smooth test functions $\phi\in \cC_c(\R)$ and $N\in \N$  (see Section \ref{sec:tightness} below for the full derivation), 
\begin{equation}
	\cY_{\lambda,N}(t,\phi)= \cY_{\lambda,N}(0,\phi) + (1+o(1)) \int_0^t \cY_{\lambda,N}(s,\tfrac12\partial_z^2 \phi)\,\dd s + \cM_{\lambda,N}(t,\phi) \comma
\end{equation}
where $(\cM_{\lambda,N}(t,\phi))_{t\ge 0}$ is a martingale with predictable quadratic variation proportional to
\begin{equation}\label{eq:pred-QV-1}
	\langle \cM_{\lambda,N}(\emparg,\phi)\rangle(t) \asymp N^{2\lambda-1} \int_0^t \sum_{z\in N^{-1/2}\Z}\hat\vartheta_{\lambda,N}^2(s,z)\,\phi^2(z)\,\ttquadre{\nabla \eta_{sN}(\lfloor sN^\lambda+zN^{1/2}\rfloor)}^2\, \dd s	\comma
\end{equation}
where $\nabla \eta(x)\eqdef \eta(x+1)-\eta(x)$ denotes the discrete gradient.
 If the field behaved as in a generic $p=1$
model (see, e.g., \cite{das_drillick_parekh_multiplicative_2026}), in which neighboring masses become roughly independent (and, thus, do not cancel out), one would expect 
\begin{equation}\label{eq:wrong-approx}
	|\hat\vartheta_{\lambda,N}(t,z)\nabla \eta_{tN}(\lfloor tN^\lambda+zN^{1/2}\rfloor)|
	\lesssim
	|\hat\vartheta_{\lambda,N}(t,z)\,\eta_{tN}(\lfloor tN^\lambda+zN^{1/2}\rfloor)|
	\asymp
	N^{-1/2}
	\comma
\end{equation}
where the second approximation comes from the fact that, up to the critical scale,  sizes of quenched  and annealed kernels are still comparable, so that the local CLT in \eqref{eq:tilt-CLT} applies. Therefore, \eqref{eq:wrong-approx} would predict 
\begin{equation}
	\langle \cM_{\lambda,N}(\emparg,\phi)\rangle(t)
	\lesssim
	N^{2\lambda-1}
	N^{-1}\sum_{z\in N^{-1/2}\Z}\phi^2(z)	
	\asymp
	N^{2\lambda-3/2}
	\comma
\end{equation}
and the right-hand side becomes order one precisely at $\lambda=3/4$.

The point is that, in the averaging process,  neighboring masses do approximately cancel out,  and quenched gradients fluctuate around the corresponding annealed gradients. This local smoothness of the averaging process, already observed in the bulk
in \cite{sau_concentration_2024,sau_tiny_2024}, turns out to persist up to the critical scale
and yields the sharper approximation (cf.\  \eqref{eq:wrong-approx})
\begin{align}
		|\hat\vartheta_{\lambda,N}(t,z)\nabla \eta_{tN}(\lfloor sN^\lambda+zN^{1/2}\rfloor)|&\asymp 	|\hat\vartheta_{\lambda,N}(t,z)\nabla \pp_{tN}(\lfloor tN^\lambda+zN^{1/2}\rfloor)|\\
		&\asymp N^{\lambda-1}\,\hat\vartheta_{\lambda,N}(t,z)\, \pp_{tN}(\lfloor tN^\lambda+zN^{1/2}\rfloor)\\
		&\asymp N^{\lambda-3/2}\fstop
\end{align}
Plugging this approximation into the right-hand side of \eqref{eq:pred-QV-1} yields
\begin{equation}
	\langle \cM_{\lambda,N}(\emparg,\phi)\rangle(t) \asymp tN^{4\lambda-4}\sum_{z\in N^{-1/2}\Z}\phi^2(z) \asymp N^{4\lambda-7/2}\comma
\end{equation}
predicting $\lambda_c=7/8=\lambda_{c,2}$, as claimed in Theorem \ref{thm:main}.

The previous argument uses a local smoothness property which is special to the
averaging process: each update on $\{x,x+1\}$ replaces the two endpoint values
by their arithmetic mean.  In particular, flat
configurations are fixed by every realization of the dynamics, and gradients are
damped, on average, rather than amplified.

This should be contrasted with generic $p=2$ RWRE models. 
For instance, consider a RWRE which, at each update,
samples a random variable $U\in(0,1)$ and then jumps to the left and to the right
with probability $U/2$, while staying put with probability $1-U$. Its annealed kernel coincides, up to a deterministic time rescaling depending only on $\E[U]$, with $\pp_t(x)$ from \eqref{eq:P=pexp}. Moreover, this model has
zero quenched drift, random second moment, and thus  belongs to the $p=2$ hierarchy.   
 Nevertheless, the quenched kernel is not locally smooth like the averaging process.
Even if the profile is flat, any update creates new microscopic
inhomogeneities proportional to the masses.  Therefore, quenched gradients fluctuate at the same order as the
quenched masses.

\subsubsection{A Dobrushin-type theorem}\label{sec:dobrushin}
The true critical exponent of the averaging process, together with the value of the noise coefficient $\gamma$ in \eqref{eq:mSHE-theorem}, can also be guessed from a second-moment computation of the tilted field $\cY_{\lambda,N}$ in \eqref{eq:tilted-field}. Indeed, for all $\lambda \in (1/2,1)$, this second moment reduces to the expectation of an exponential functional of an auxiliary walk\footnote{Actually, one encounters a tilted walk (see Section \ref{sec:moment-repr}), but this detail is unimportant at this level.} $(X_t)_{t\ge 0}$ on $\Z$, defined as the gap process $R^1 - R^2$ for the two-point motion $(R^1, R^2)$ of the RWRE associated to the averaging process. This process is essentially the symmetric simple random walk $(R_{2t})_{t\ge 0}$, except for a small local defect around the origin; see Section \ref{sec:moment-repr} below for the details.

In a general RWRE model whose environment law is also invariant under space reflection, the exponent of this exponential has the formal expansion
\begin{equation}\label{eq:Phi-series}
	\sum_{q=1}^\infty 2^{-q} 
	N^{-2q(1-\lambda)} 
	\int_0^{tN}\Phi_q(X_s)\,\dd s
	\fstop
\end{equation}
The framework developed in \cite{parekh2024hierarchy,drillick_parekh_random_2025} says that, if the $p$-th quenched moment is the first genuinely random one, then
\begin{equation}\label{eq:p-condition-Phi}
	\Phi_1=\ldots=\Phi_{p-1}=0\comma
	\qquad
	\Phi_p\neq 0
	\fstop
\end{equation}
In this case, the first nonzero term in \eqref{eq:Phi-series} is
\begin{equation}
	N^{-2p(1-\lambda)}
	\int_0^{tN}\Phi_p(X_s)\,\dd s
	\fstop
\end{equation}
As $(X_t)_{t\ge 0}$ diffusively scales to a Brownian motion $(W_t)_{t\ge 0}$, classical local-time theorems are expected to apply, yielding
\begin{equation}\label{eq:classical-local-time}
	N^{-1/2}
	\int_0^{tN}\Phi_p(X_s)\,\dd s
	\xLongrightarrow[N\to \infty]{}
	\gamma_p^2\,L_t^W\comma\quad \text{with}\ \gamma_p^2= \sum_{x\in \Z}\Phi_p(x)\,\pi^{\rm inv}(x)\fstop
\end{equation}
Here, $L_t^W$ denotes the local time at the origin of $(W_t)_{t\ge 0}$, while $\pi^{\rm inv}$ denotes the invariant measure of the walk $(X_t)_{t\ge 0}$. Thus, in the setting \eqref{eq:p-condition-Phi}, the natural matching condition is
\begin{equation}
	N^{-2p(1-\lambda)}
	=
	N^{-1/2}
	\comma
\end{equation}
which gives $\lambda=\lambda_{c,p}=1-(4p)^{-1}$.

Although the averaging process belongs to the $p=1$ hierarchy, the corresponding additive functional vanishes at this first scale. More precisely, $\Phi_1\neq 0$, but $\Phi_1$ has mean zero (Lemma~\ref{lem:two-point-discrepancy}):
\begin{equation}\label{eq:Phi1-zero-mass}
	\gamma_1^2=\sum_{x\in \Z}\Phi_1(x)\,\pi^{\rm inv}(x)=0\fstop
\end{equation}
Consequently, the local-time convergence in \eqref{eq:classical-local-time} gives no contribution when $\lambda=\lambda_{c,1}=3/4$.

The fact that $\Phi_1$ has mean zero can be understood intuitively as follows. By
\cite{parekh2024hierarchy,drillick_parekh_random_2025},  
\begin{equation}
   \Phi_1(x) = \mathbb C\text{ov}_{(0,x)}(R^1_1, R^2_1)\comma\qquad x \in \Z\comma
\end{equation}
where the right-hand side denotes the covariance of the two-point motion $(R^1_t, R^2_t)$ at time $t=1$ started from initial conditions $(R^1_0, R^2_0) = (0,x)$. Furthermore, for the averaging process, the invariant measure $\pi^{\rm inv}$ for the gap process is just the counting measure on $\mathbb Z$ (see Proposition \ref{pr:gapprocess-tilt}). Suppose the two random walkers $R^1$ and $R^2$ to be at the same location $x$. Their increments will be positively correlated if they both jump across the same edge, and this happens with rate $2 \times \frac14 = \frac12$, where $2$ is the rate at which one of the two Poisson clocks on the edges $\{x, x+1\}$ or $\{x-1, x\}$ rings and $\frac14$ is the probability that both particles make the jump. When the two particles are at neighboring locations $x, x+1$, their increments are negatively correlated if the two particles swap locations; this happens at rate $1 \times \frac14$, where $1$ is the rate at which the Poisson clock rings along the edge $\{x, x+1\}$ and $\frac14$ is the probability that both particles make the jump.  
Observe that the rate of the negative contribution is half the rate of the positive contribution. However, since we are summing $\Phi_1(x)$ against the counting measure, every possible gap between particles is equally weighted. In particular, a gap of magnitude $1$ (i.e., $x=\pm1$) is weighted twice as much as a gap of magnitude $0$ (i.e., $x=0$). In this way, the positive and negative contributions exactly cancel out. This argument is made rigorous in Lemma \ref{lem:two-point-discrepancy}. 

This is where one needs a Dobrushin-type local-time theorem \cite{dobrusin_two_1955}; see also \cite{csaki_random_2009} and the references therein. In the form used here (Theorem \ref{th:Dobrushin-k=2}), it says that, for a nonzero function $f:\Z\to \R$ supported on a finite set and satisfying $\sum_{x\in \Z}f(x)\,\pi^{\rm inv}(x)=0$,
\begin{equation}\label{eq:Dobrushin-type}
	N^{-1/4}
	\int_0^{tN} f(X_s)\,\dd s
	\xLongrightarrow[N\to \infty]{}
	\sigma(f)\,B_{L_t^W}\fstop
\end{equation}
Here, $(B_t)_{t\ge 0}$ is another Brownian motion, independent of $(W_t)_{t\ge 0}$, while $\sigma^2(f)\ge 0$ is the variance associated with this additive functional.

Applying \eqref{eq:Dobrushin-type} with $f=\Phi_1$ to the first-order term of the averaging process forces
\begin{equation}
	N^{-2(1-\lambda)}
	=
	N^{-1/4}
	\comma
\end{equation}
namely $\lambda=7/8$. Since $\sigma_1:=\sigma(\Phi_1)$ turns out to be strictly positive, $\Phi_1$ gives a non-trivial contribution at $\lambda=7/8$. As $7/8=\lambda_{c,2}$, the second term in \eqref{eq:Phi-series} also contributes at the same scale, as $\Phi_2$ has non-zero mean. Therefore, in the averaging process, the limiting coefficient {$1/\sqrt 2$} in \eqref{eq:mSHE-theorem} is not produced by a single local-time term. It is the effect of two contributions: the  $\sigma_1$ coming from the mean-zero term $\Phi_1$, and the ordinary  $\gamma_2$ coming from $\Phi_2$.

Let us remark that this simultaneous-contribution phenomenon seems to be special to the first level of the hierarchy. Indeed, a mean-zero contribution $\Phi_p$ and a nonzero-mass contribution $\Phi_q$, with $q>p$, are simultaneous if their critical scales coincide. The Dobrushin fluctuation of $\Phi_p$ is critical when
$N^{-2p(1-\lambda)}N^{1/4}\asymp 1$, 
whereas the ordinary local-time contribution of $\Phi_q$ is critical when
$N^{-2q(1-\lambda)}N^{1/2}\asymp 1$.
Equating the two scales gives
\begin{equation}
	q=2p\fstop
\end{equation}
Thus, this mechanism involves adjacent orders only when $p=1$. This is precisely what happens for the averaging process. We leave to the reader the easy verification that the same mechanism is also present in natural block variants of the model, in which perfect averaging updates are performed on intervals of size $k\ge 2$, rather than only on nearest neighbors.

When $p\ge 2$, 
the $p$-th and the $q=2p$-th terms being dominant would  require, aside from
\begin{equation}\label{eq:general-p-partI}
\Phi_p\neq 0\comma\qquad \Phi_{2p}\neq 0\comma\qquad \textstyle{\sum_{x\in \Z}}\,\Phi_p(x)\,\pi^{\rm inv}(x)=0\comma
\end{equation}
also the additional conditions:
\begin{equation}\label{eq:general-p-partII}
\Phi_r=0\  \text{for all}\  r<p\comma\qquad \textstyle{\sum_{x\in \Z}}\,\Phi_r(x)\,\pi^{\rm inv}(x)=0\  \text{for all}\  p<r<2p=q\fstop
\end{equation}
These should only be regarded as formal necessary conditions. A calculation at the level of quenched moments suggests that, under a space-time i.i.d.\ random environment, they may be incompatible as soon as $p\ge 2$; see Section \ref{sec:questions} below.

\subsubsection{Proof strategy and technical challenges}
\label{sec:proof-strategy}
The proof of Theorem \ref{thm:main} follows the usual tightness-plus-characterization scheme. A key point, however, is that we use convergence of finite-dimensional distributions twice: as usual, to characterize the limiting field, but also as one of the inputs in the proof of tightness. 

We start by proving finite-dimensional convergence of the tilted fields.
The main tool is the recent moment-based characterization of stochastic flows, or propagators, associated with the one-dimensional multiplicative SHE in \eqref{eq:mSHE-theorem}; see \cite{parekh_moment_2025}. This builds on the analogous characterization of  \cite{tsai2024stochastic} for the two-dimensional critical heat flow \cite{caravenna_sun_zygouras_critical_2023}. In our setting, this characterization reduces the problem to proving convergence of all joint moments of the tilted fields, started from general Dirac initial conditions $z_N\in N^{-1/2}\Z$, with $z_N\to z\in\R$.

These joint moments are expressed through tilted multi-particle motions. For the $k$-th moment, one has to analyze a tilted $k$-point motion
$
(R^{i})_{i=1,\ldots,k}
$
and its pairwise differences
\begin{equation}\label{eq:diff-processes}
	(R^{i}-R^{j})_{1\le i<j\le k}\fstop
\end{equation}
The required moment limits follow from joint functional CLTs for the $k$-point motion, for the gap processes in \eqref{eq:diff-processes}, and for additive functionals which charge the times at which some particles are at uniformly bounded lattice distance from each other. We refer to these times as collisions.

For $k=2$, this reduces to proving a convergence of the form \eqref{eq:Dobrushin-type}, with $X=R^{(\lambda,1)}-R^{(\lambda,2)}$. For $k\ge 3$, analogous additive functionals appear for the tilted $k$-particle motion. As in many other works on RWRE limits, the main point is to prove that, in the limit, only pairwise collisions contribute, while simultaneous collisions involving three or more particles give negligible error terms.

In \cite{drillick_parekh_random_2025}, this issue is handled through on- and off-diagonal heat-kernel upper bounds for a general class of non-reversible Markov chains on lattices with short-range interactions. Those estimates are well suited to the usual $N^{-1/2}$ local-time scale. In the present model, however, the relevant fluctuations occur at the finer scale $N^{-1/4}$, and the existing bounds are not sharp enough to discard the higher-order collision terms. A substantial part of the proof is therefore devoted to refining these heat-kernel and collision estimates to the scale required by \eqref{eq:Dobrushin-type}. This is one of the main technical novelties of the paper, and is carried out in Section \ref{sec:higher-moments}.

We then prove tightness of the tilted fields in
$\cD([0,\infty);\mathscr M_f(\R))$
with respect to the uniform topology. The argument is rather soft: it combines the classical stopping-time tightness criterion of Aldous \cite{aldous_stopping_1978} with a second criterion of Aldous based on finite-dimensional convergence of martingales \cite{aldous_stopping_1989}. This is where the finite-dimensional convergence proved earlier enters the tightness proof. In particular, this approach avoids estimating $(1+\varepsilon)$-th moments of the carré-du-champ of the field martingales (see Remark \ref{rem:carre-du-champ}), often the most cumbersome step in pathwise $\cD$-tightness proofs based on stopping times.

We also note that both the tightness result and its proof differ from those in
\cite{das_drillick_parekh_multiplicative_2026,parekh2024hierarchy,drillick_parekh_random_2025},
where pathwise $\mathcal C$-tightness was obtained for linearly interpolated fields through the Kolmogorov--Chentsov tightness criterion.

\subsubsection{Further questions}\label{sec:questions}  
We close with a few open problems suggested by the preceding discussions.

First, it would be interesting to understand how exceptional the averaging process is. In Section \ref{sec:dobrushin}, we explained that its limiting coefficient comes from two simultaneous contributions: a mean-zero Dobrushin fluctuation generated by $\Phi_1$, and an ordinary local-time contribution generated by $\Phi_2$. As observed above, this simultaneous-contribution mechanism involves adjacent-level contributions only for $p=1$. For $p\ge 2$, the analogous mechanism would instead require a resonance between the non-adjacent orders $p$ and $2p$, together with the cancellation conditions; see \eqref{eq:general-p-partI}--\eqref{eq:general-p-partII}. A preliminary calculation suggests that these conditions may be incompatible within the standard class of RWRE models driven by a space-time i.i.d.\ environment. This points to the problem of identifying a broader framework, possibly involving correlated environments or more general stochastic flows, in which higher-level resonances can occur, and to determine whether they lead to further KPZ-type limits or to genuinely new universality classes.

Second, as already discussed in Section \ref{sec:fluctuation-regimes}, the supercritical regime remains poorly understood. For the averaging process, Theorem \ref{th:non-critical} identifies the onset of intermittency at the level of second moments, but it does not describe the typical behavior of the tilted field when $\lambda\in(7/8,1]$. More generally, even for one-dimensional RWRE models, first-order limits for $\lambda\in(\lambda_c,1)$ and fluctuation limits beyond criticality are largely open outside a few exactly solvable cases. One would like to identify the deterministic centering in the supercritical decomposition, the order and limiting law of the residual fluctuations. At the endpoint $\lambda=1$, this connects with the large-deviation regime, where Tracy--Widom fluctuations are known only in special integrable models.

Third, one expects an analogous phase diagram in higher spatial dimensions, although the relevant scales are different from the one-dimensional ones. For directed polymers in random environments (see, e.g., \cite{zygouras_directed_2024,junk_lacoin_coincidence_2026}, and references therein), the subcritical, critical, and supercritical regimes depend strongly on dimension, and a similar picture is expected for RWRE and stochastic-flow models. In particular, the subcritical regime for a broad class of RWREs in all spatial dimensions is now covered by \cite{drillick_parekh_random_2025}, at least under the higher-dimensional analogue of the non-degeneracy condition $\lambda_c=\lambda_{c,p}$. The critical and supercritical regimes in dimensions $d\ge 2$, both for the averaging process and for more general RWRE models, remain open.

 \subsection*{Organization of the paper}
The rest of the paper is organized as follows. In Section \ref{sec:moment-repr}, we derive a moment representation for the fields $\cY_{\lambda,N}$ in terms of $k$-point motions. In Section \ref{sec:2nd-moment}, we analyze the second moment, prove the cancellation of the ordinary local-time coefficient, and identify the Dobrushin-type contribution at the scale $\lambda=7/8$. In Section \ref{sec:higher-moments}, we extend the argument to higher moments by proving the required heat-kernel, collision, and additive-functional estimates for the tilted multi-point motions. In Section \ref{sec:conv-fdd}, we use these moment asymptotics to prove convergence of finite-dimensional distributions and identify the limiting multiplicative SHE propagator. In Section \ref{sec:tightness}, we prove tightness. Finally, in Section \ref{sec:proof-final}, we combine these inputs to prove Theorems \ref{thm:main} and \ref{th:non-critical}.

\section{Moment representation: $k$-point motion and gap processes}\label{sec:moment-repr}

A key ingredient to compute moments of the fields in \eqref{eq:tilted-field} will be the so-called \emph{$k$-point motion} of the averaging process, which describes the annealed motion of $k$ particles sampled independently using the same Poisson clocks. In this section, we define the $k$-point motion, and explain how to represent the moments of $\cY_{\lambda,N}$ in terms of a tilted version of it.

\subsection{$k$-point motion and its tilted variant}
We first define a four-parameter extension of the averaging process $\eta=(\eta_t(x))_{t\ge 0,\,x\in \Z}$ (Section \ref{sec:averaging-process}) as follows. Let $\Z$ be the integer lattice, and assign each bond between neighboring integers an independent Poisson point process of unit intensity. In what follows, $\P$ and $\E$ denote the law and corresponding expectation associated to this random environment.
Given a realization of this Poisson environment, we run a random walk on $\mathbb Z$,  in which the particle traverses a bond only at the occurrence of a Poisson mark, and with probability $1/2$, independently at each step. 
    
\begin{definition}[Stochastic kernel]\label{kst}  For all $0\le s \le t$ and $x,y \in \Z$, let $K_{s,t}(x,y)$ be the probability that the particle starting from time $s$ and location $x$ ends up at position $y$ at time $t$.  We refer to $K=(K_{s,t}(x,y))_{s\le t,\,x,y\in \Z}$ as the averaging process' \textit{stochastic kernel}.
\end{definition}
When $s=0$, we simply write $K_t=K_{0,t}$.
In this way, the original Markov process introduced in Section \ref{sec:averaging-process} is embedded inside the family $K_{s,t}(x,y)$ by setting $s=x=0$, i.e., $\eta_t(y) = K_t(0,y)$.
    
\begin{definition}[$k$-point motion]
   \label{kpoint}
    For $k\in \Bbb N$, define the \textit{$k$-point motion} to be the continuous-time Markov process $\mathbf R=(R^1,\ldots,R^k)$ on $\Bbb Z^k$ with transition probabilities  given, for every $t\ge0$ and $\mathbf x = (x_1,\ldots,x_k), \mathbf y=(y_1,\ldots,y_k)\in \Z^k$, by
    \begin{equation}\pp^{(k)}_t(\mathbf x, \mathbf y)\eqdef \Bbb E \bigg[ \prod_{j=1}^k K_t (x_j,y_j)\bigg]\fstop
    \end{equation} 
     Let $\cL^{(k)} (\x,\y):= \frac{\dd}{\dd t}\big|_{t=0}\, \pp^{(k)}_t(\mathbf x,\mathbf y) $, $\x, \y \in \Z^k$, denote the corresponding generator, while
    $\mathbf P_\x^{(k)}$ and $\mathbf E_\x^{(k)}$ stand for the path law and corresponding expectation when starting from $\x\in \Z^k$.
    \end{definition}
    Note that for $k=1$, the $k$-point motion reduces to $(R_t)_{t\ge 0}$ from \eqref{eq:quenched-annealed}, namely, the CTRW on $\Z$ whose rate to jump from $x$ to $x\pm 1$ equals $1/2$. In accordance with \eqref{eq:P=pexp}, \begin{equation}\pp_t(y-x)=\pp_t(x-y)=\pp_t^{(1)}(x,y)\comma\qquad t \ge 0\comma \x, \y \in \Z^k\fstop
    \end{equation}
    Indeed, the one-point motion may be described as a particle on $\Z$ which, whenever a Poisson clock attached to an adjacent bond rings, it relocates itself with probability $1/2$ to the other endpoint of the bond. As Poisson clocks ring at unit rate, this coincides with the law of $(R_t)_{t\ge 0}$.
    
   For $k > 1$, the $k$ coordinates nontrivally interact by possibly jumping simultaneously together: when a bond rings, all particles sitting on the two endpoints of the bond independently move to the other end with probability $1/2$. This process may be viewed as a labeled version of the Binomial splitting process, see, e.g., \cite{quattropani2021mixing,bristiel_caputo_entropy_2021}. Note that, for each $k\ge 2$ and $1\le \ell < k$, the marginal law of any $\ell$-tuple of the $k$ coordinates evolves as an $\ell$-point motion. We shall further elaborate on this and other properties in the next section.

    Let us now introduce the tilted variant of these processes. Remark that such objects are parametrized by a tilt parameter $\beta\in \R$. When a scaling involving $\lambda\le 1$ and $N\in \N$ takes place, we
shall specialize $\beta$ to $\beta_{\lambda,N}=\operatorname{arcsinh} (N^{\lambda-1})$ as in \eqref{eq:beta-lambda-N}.
    \begin{definition}[Tilted $k$-point motion]\label{tilt} For $\beta\in \Bbb R$ and $k\in \N$, we define the \textit{$\beta$-tilted $k$-point motion} to be the Markov process whose generator is given by \begin{equation}\cL^{(\beta,k)}(\x,\y)=e^{\beta \sum_{j=1}^k (y_j-x_j)} \cL^{(k)}(\x,\y) - \ind_{\{\y=\x\}} \sum_{\mathbf z\in\Bbb Z^k} e^{\beta \sum_{j=1}^k (z_j-x_j)} \cL^{(k)}(\x,\mathbf z)\fstop
    \end{equation}
    In analogy with Definition \ref{kpoint}, we write $\mathbf P_\x^{(\beta,k)}$ and $\mathbf E_\x^{(\beta,k)}$ for the corresponding law and expectation, and $\mathbf p_t^{(\beta,k)}$ for its transition probabilities.
\end{definition}
\begin{remark}[Girsanov interpretation]\label{rem:girsanov-tilted-k-point}
	Let us motivate the above definition as a \textit{tilted} $k$-point motion. Let
	$\mathbf R=(R^1,\ldots,R^k)$ be the $k$-point motion on $\Z^k$ with
	generator $\cL^{(k)}$, and
    \begin{equation}\label{eq:F-beta-k-def}
		F_\beta(\mathbf x)=F_{\beta,k}(\x)
		\eqdef
		\beta\sum_{i=1}^k x_i\comma
		\qquad
		\mathbf x=(x_1,\ldots,x_k)\in\Z^k\fstop
	\end{equation} Then (see, e.g., \cite[Appendix 1, Section 7]{kipnis_scaling_1999}),
	\begin{equation}\label{eq:k-point-girsanov-martingale}
		\mathsf M_t
		\eqdef
		\exp\set{
			F_\beta(\mathbf R_t)-F_\beta(\mathbf R_0)
			-
			\int_0^t e^{-F_\beta}(\mathbf R_s)\cL^{(k)}e^{F_\beta}(\mathbf R_s)\,\dd s
		}\comma\qquad t\ge 0\comma
	\end{equation}
	is a positive martingale with mean one under $\mathbf P_{\mathbf x}^{(k)}$,
	for every $\mathbf x\in\Z^k$. Hence, we may define a
	new probability, whose Radon--Nikodym derivative with respect to $\mathbf P^{(k)}$, evaluated on paths up to time $t$, is given by $\mathsf M_t$.
	This corresponds to the $F_\beta$-tilt of the $k$-point motion. Under
	this measure, the process is again a continuous-time
	Markov chain, with generator
	\begin{equation}\label{eq:k-point-F-tilted-generator}
		e^{-F_\beta}\cL^{(k)}(e^{F_\beta} g)
		-
		(e^{-F_\beta}\cL^{(k)}e^{F_\beta})g\fstop
	\end{equation}
    By taking $g=\car_{\{y\}}$ and $\mathbf x\neq\mathbf y$, the above expression equals 
	\begin{equation}\label{eq:k-point-F-tilted-rates}
		e^{F_\beta(\mathbf y)-F_\beta(\mathbf x)}
		\cL^{(k)}(\mathbf x,\mathbf y)\fstop
	\end{equation}
    Consequently, one indeed recovers $\cL^{(\beta,k)}(\x,\y)$ from Definition \ref{tilt} and, thus, this tilted measure is precisely $\mathbf P_\x^{(\beta,k)}$.
\end{remark}
\begin{remark}[Discrete-time approximation]
Intuitively, the tilted $k$-point motion is obtained by taking the discrete-time Markov chain with transition probabilities \begin{equation} \label{eq:discrete}\frac{e^{\beta \sum (y_j-x_j)} \pp^{(k)}_\epsilon(\x,\y)}{\sum_{\mathbf z\in \Z^k} e^{\beta \sum(z_j-x_j)} \pp^{(k)}_\epsilon(\x,\mathbf z) }\comma\qquad t \ge 0\comma \x, \y \in \Z^k\comma
\end{equation}
with $\epsilon^{-1}$ equally-spaced steps in a time unit, and then letting $\epsilon\to 0$ to obtain a continuous-time process. 
\end{remark}

If $\beta=0$, one recovers the $k$-point motion (i.e., $\cL^{(\beta,k)}=\cL^{(k)}$), whereas for $\beta >0$ (resp.\ $\beta<0$), the $\beta$-tilted process induces, on each coordinate, a positive (resp.\ negative) drift. For instance, for $k=1$, $\cL^{(\beta,k)}=\cL^{(\beta,1)}$ describes the motion of the CTRW taking one step to the right (resp.\ left) with rate $e^{\beta}/2$ (resp.\ $e^{-\beta}/2$).

The role of this $\beta$-tilting is to probe the behavior of the averaging process in extremal regimes, whereas the untilted case $\beta=0$ would correspond to the bulk regime.
As already mentioned in Section \ref{sec:intro}, in this work we will  primarily be concerned with the case where $\beta\sim N^{-1/8}$ (i.e., $\beta=\beta_{\lambda,N}=\operatorname{arcsinh}(N^{\lambda-1})$ with $\lambda = 7/8$), while diffusively rescaling the process---so that time is of order $N$ and space of order $N^{1/2}$.

\subsection{Gap processes}\label{sec:gapprocesses}
Recall Definitions \ref{kpoint} and \ref{tilt}.
\begin{definition}[Gap process]
For $\beta\in \R$ and $k\ge 2$, a \textit{gap process} is defined to be the difference of two coordinates of the $\beta$-tilted $k$-point motion, i.e., $R^i-R^j$, $i,j=1,\ldots, k$, $i\neq j$. When $\beta = 0$ (resp.\ $\beta \neq 0$), we speak of the \textit{untilted} (resp.\ \textit{tilted}) gap process. 
\end{definition}

When $\beta=0$, the gap process is particularly simple to describe.

\begin{proposition}\label{gapprocess}
For $\beta=0$, the gap process of the averaging process is a Markov process on $\Bbb Z$ whose transitions go as follows. Suppose that it sits at $x\in \Z$:
    \begin{enumerate}[(i)]
        \item  if  $x\neq 0, \pm 1$, it jumps to $x\pm1$ with rate $1$;
        \item  if $x=0$,  it jumps to $1$ with rate $1/2$, and to $-1$ with rate $1/2$;
        \item  if $x=\pm 1$, it jumps to $\mp 1$ with rate $1/4$, to $0$ with rate $1/2$, and to $\pm 2$ with rate $1$.
    \end{enumerate}
\end{proposition}
\begin{proof}
As already mentioned, for $\beta=0$, any two coordinates of the $k$-point motion evolve as a two-point motion. The desired claim follows as $R^i$ and $R^j$, $i\neq j$, evolve as two independent copies of the CTRW $R$ as long as they sit at distance at least two, while when at distance one or zero, they may resample their positions simultaneously if the Poisson clock attached to an adjacent bond rings. In that case, the two particles choose independently and with probability $1/2$ to relocate in either one of the two endpoints of the bond.
\end{proof}
In other words, 
the untilted gap process is a symmetric random walk on $\Z$, with slow bonds attached to the origin. Indeed, as the one-point motion of the averaging process is a CTRW crossing bonds with rate $1/2$, the difference of two independent copies of such   walks is a CTRW crossing bonds with a doubled rate. The untilted gap process is like the latter process, but with slow bonds attached to the origin, ringing at half rate (plus a teleportation $\pm 1\to \mp 1$ with rate $1/4$).

As a direct consequence of Proposition \ref{gapprocess} and the fact the untilted gap process is a CTRW on $\Z$ with symmetric conductances, the counting measure on $\Z$ is invariant and, actually, reversible for the process. We record this result in the following corollary. 
\begin{corollary}\label{prop:inv} For $\beta=0$, the generator $\cL^{\rm gap}$ of the untilted gap process is self-adjoint on $L^2(\Z,\pi)$, where $\pi$ is the counting measure on $\Z$. 
\end{corollary}
Tilted gap processes of the $k$-point motion are more delicate than their
untilted counterparts. This is already visible at the level of coordinate
projections. Indeed, if $\beta\neq0$, $k\ge2$, and $1\le \ell<k$, then the
$\ell$-coordinate marginal of the tilted $k$-point motion is not, in general,
Markovian. In particular, its law need not coincide with that of the tilted
$\ell$-point motion.

For example, take $k=2$ and project onto the first coordinate. The jump rate of
$R^1$ from $0$ to $1$ depends on the position of the second coordinate: it is
\begin{equation}
	\tfrac12 e^\beta
	\quad\text{if}\  R^2\notin\set{0,1}\comma
	\qquad
	\tfrac14 e^\beta\tttonde{1+e^\beta}
	\quad\text{if}\  R^2=0\comma
	\qquad
	\tfrac14\tttonde{1+e^\beta}
	\quad\text{if}\ R^2=1\fstop
\end{equation}
Thus, the evolution of $R^1$ alone retains information about the discarded
coordinate.

Nevertheless, in the special case $k=2$, the tilted gap process is still
Markovian. More precisely, it evolves as a speed change of the untilted gap
process. The next result is the tilted analogue of Proposition~\ref{gapprocess};
it follows directly from Definition~\ref{tilt} with $k=2$ and arbitrary
$\beta\in\R$. We leave the elementary verification to the reader.

\begin{proposition}\label{pr:gapprocess-tilt} Fix $k=2$. Then, for every $\beta\in \R $, the gap process $X\eqdef R^1-R^2$ is a ${\rm CTRW}$ on $\Z$. Supposing that the gap process sits at $x \in \Z$, its evolution goes as follows:
\begin{enumerate}[(i)]
    \item if $x\neq 0,\pm 1$,  it  jumps to $x\pm 1$ with rate $\cosh(\beta)$;
    \item if $x=0$, it jumps to $\pm 1$ with rate $\frac12\cosh(\beta)$;
    \item if $x=\pm 1$, it jumps to $\pm 2$ with rate $\cosh(\beta)$, to $0$ with rate $\frac12\cosh(\beta)$, and to $\mp 1$ with rate $1/4$.
\end{enumerate}
As a consequence, the counting measure $\pi$ on $\Z$ is reversible for $X$.
\end{proposition}
Crucially, despite each coordinate of the two-point motion moves like a $\beta$-biased random walk under $\mathbf P_\x^{(\beta,k)}$, $X$ is unbiased and, actually, symmetric, with the counting measure $\pi$ on $\Z$ as its reversible measure. This feature---lacking as soon as $k\ge 3$ and $\beta\neq 0$---turns the analysis of second moments of the fields considerably less involved than those of higher ones.

\subsection{Moment representation}\label{sec:moment-representation}
We now express the $k$-th moment of the tilted fields in \eqref{eq:tilted-field} in terms of the tilted $k$-point motion. For later purposes, we generalize the definition of these fields to account for arbitrary initial conditions.

Fix $N \in \N$, and a tilt parameter $\beta \in \R$.
For $t\ge0$, $x\in\Z$, and $\phi\in\cC_b(\R)$, define
\begin{equation}\label{eq:H-beta-N}
	\mathfrak H_{\beta,N}(t,x,\phi)
	\eqdef 
	\sum_{y\in\Z}
	\exp\{\beta (y-x)-tN\tonde{\cosh(\beta)-1}\}\,
	K_{tN}(x,y)\,
	\phi(N^{-1/2}\tonde{y-tN\sinh(\beta)})\fstop
\end{equation}
Since  $N\sinh(\beta_{\lambda,N})=N^\lambda$ (cf.\ \eqref{eq:beta-lambda-N}), one recovers, for every $\lambda \le 1$, 
\begin{equation}\label{eq:H-specializes-Y}
	\mathfrak H_{\beta_{\lambda,N},N}(t,0,\phi)
	=
	\cY_{\lambda,N}(t,\phi)\fstop
\end{equation}

For illustrative purposes, we start with the simplest case of the first moment of $\mathfrak H_{\beta,N}$.
 Recall that the one-point
motion has generator
\begin{equation}\label{eq:one-point-generator}
	\cL^{(1)}g(x)
	=
	\frac12\tonde{g(x+1)+g(x-1)-2g(x)}\comma
	\qquad x\in\Z\comma g: \Z\to \R\fstop
\end{equation}
For $F_\beta(x)=F_{\beta,1}(x)=\beta x$ (cf.\ \eqref{eq:F-beta-k-def}), one has
\begin{equation}\label{eq:one-point-eigenvalue}
	e^{-F_\beta}(x)\,\cL^{(1)}e^{F_\beta}(x)
	=
	\cosh(\beta)-1\fstop
\end{equation}
Consequently, the tilted one-point generator from Definition \ref{tilt} is
\begin{equation}\label{eq:one-point-tilted-generator}
	\cL^{(\beta,1)}g(x)
	=
	\frac{e^\beta}{2}\tonde{g(x+1)-g(x)}
	+
	\frac{e^{-\beta}}{2}\tonde{g(x-1)-g(x)}\fstop
\end{equation}
Under $\mathbf P_x^{(\beta,1)}$, the canonical coordinate $R=(R_t)_{t\ge 0}$ has drift
$\sinh(\beta)$, in the sense that
\begin{equation}\label{eq:one-point-drift-beta}
	\frac{\dd}{\dd t}\mathbf E_x^{(\beta,1)}[R_t]
	=
	\sinh(\beta)\fstop
\end{equation}
The first moment of \eqref{eq:H-beta-N} is now a direct consequence of \eqref{eq:one-point-eigenvalue}, see
Remark \ref{rem:girsanov-tilted-k-point}: for all $x\in\Z$,
\begin{align}\label{eq:first-moment-H-beta}
	\E[\mathfrak H_{\beta,N}(t,x,\phi)]
	&= \mathbf E_x^{(1)}
	\big[
		\exp\set{
			\beta (R_{tN}-x)
			-
			tN\tonde{\cosh(\beta)-1}
		}
		\phi\tttonde{
			N^{-1/2}\tonde{R_{tN}-tN\sinh(\beta)}
		}
	\big]\\
    &=
	\mathbf E_x^{(\beta,1)}
	\big[
		\phi\tttonde{N^{-1/2}\tonde{R_{tN}-tN\sinh(\beta)}}
	\big]\fstop
\end{align}

We now pass from first moments to joint moments. For $k\in\N$, let $\cL^{\otimes k}$ denote the generator of $k$ independent one-point motions.
Since $F_\beta=F_{\beta,k}$ from \eqref{eq:F-beta-k-def} is additive in the coordinates, \eqref{eq:one-point-eigenvalue}
implies
\begin{equation}\label{eq:independent-k-eigenvalue}
	e^{-F_\beta}(\x)
	\cL^{\otimes k}e^{F_\beta}(\x)
	=
	k\tonde{\cosh(\beta)-1}\comma
	\qquad \x\in\Z^k\fstop
\end{equation}
For the genuine $k$-point motion, this identity no longer holds in general,
because the coordinates interact when they are close. The defect is encoded by
the exponential functional
\begin{equation}\label{eq:G-beta-k}
	\mathscr E_{\beta,k}(T) = \mathscr E_{\beta,k}(T,\mathbf R)
	\eqdef
	\exp\set{
		\int_0^T
		e^{-F_\beta}(\mathbf R_s)
		\ttonde{\cL^{(k)}-\cL^{\otimes k}}
		e^{F_\beta}(\mathbf R_s)
		\,\dd s
	}\comma
	\qquad T\ge0\fstop
\end{equation}
Here and below, $\mathbf R=(R^1,\ldots,R^k)$ denotes the canonical process
under the law displayed in the expectation. In particular, $\mathscr E_{\beta,k}\equiv 1$ for $k=1$.

The next proposition is the desired moment representation. It is stated with a
general initial condition $\mathbf x\in\Z^k$, since this is the form needed when
the moment computation is combined with later Markov-property arguments.

\begin{proposition}[Moment representation]\label{cor:moments}
	Fix $\beta\in\R$, $N\in\N$, $k\in\N$, and
	$\mathbf x=(x_1,\ldots,x_k)\in\Z^k$. Recall $\mathscr E_{\beta,k}$ from \eqref{eq:G-beta-k}. Then, for every $t\ge0$ and every
	$\phi_1,\ldots,\phi_k\in\cC_b(\R)$,
	\begin{equation}\label{eq:moment-representation-beta}
		\begin{aligned}
		&\E\bigg[
			\prod_{i=1}^k
			\mathfrak H_{\beta,N}(t,x_i,\phi_i)
		\bigg]
		=
		\mathbf E_{\mathbf x}^{(\beta,k)}
		\bigg[
			\mathscr E_{\beta,k}(tN)
			\prod_{i=1}^k
			\phi_i(N^{-1/2}\tttonde{R_{tN}^i-tN\sinh(\beta)})
		\bigg]\fstop
		\end{aligned}
	\end{equation}
\end{proposition}

\begin{proof} Write $\phi_{\beta,N,t}(\mathbf R_{tN})$ for the product $\prod_{i=1}^k \phi_i(\ldots)$ on the right-hand side of \eqref{eq:moment-representation-beta}.
	By Definitions \ref{kst} and \ref{kpoint}, expanding the product
	of the fields gives
	\begin{equation}\label{eq:moment-start-beta}
		\begin{aligned}
		\E\bigg[\prod_{i=1}^k
			\mathfrak H_{\beta,N}(t,x_i,\phi_i)
		\bigg]
		=
		\mathbf E_{\mathbf x}^{(k)}
		\big[
			\exp\set{
				(F_\beta(\mathbf R_{tN})-F_\beta(\x))
				-
				ktN\tonde{\cosh(\beta)-1}
			}
			\phi_{\beta,N,t}(\mathbf R_{tN})
		\big]\fstop
		\end{aligned}
	\end{equation}
	Indeed, the average of the product of stochastic kernels is precisely the
	transition probability of the $k$-point motion, while the deterministic
	exponential factors combine into
	$F_\beta(\mathbf R_{tN})-ktN(\cosh(\beta)-1)$.

	Set $T=tN$. By the change of measure in
	Remark~\ref{rem:girsanov-tilted-k-point}, we have,
	 for every bounded $\cF_T$-measurable random variable $\Psi$,
	\begin{equation}\label{eq:girsanov-identity-used}
		\mathbf E_{\x}^{(k)}
		\bigg[
			e^{F_\beta(\mathbf R_T)-F_\beta(\x)}\Psi
		\bigg]
		=
		\mathbf E_{\mathbf x}^{(\beta,k)}
		\bigg[
			\exp\set{
				\int_0^T
				e^{-F_\beta}(\mathbf R_s)
				\cL^{(k)}e^{F_\beta}(\mathbf R_s)\,\dd s
			}
			\Psi
		\bigg]\fstop
	\end{equation}
	Applying \eqref{eq:girsanov-identity-used} with
	\begin{equation}
		\Psi
		\eqdef
		\exp\set{-kT\tonde{\cosh(\beta)-1}}\,
		\phi_{\beta,N,t}(\mathbf R_T)
	\end{equation}
	gives
	\begin{equation}\label{eq:girsanov-step-moment}
		\begin{aligned}
		&\mathbf E_{\mathbf x}^{(k)}
		\big[
			\exp\set{
				(F_\beta(\mathbf R_T)-F_\beta(\x))
				-
				kT\tonde{\cosh(\beta)-1}
			}
			\phi_{\beta,N,t}(\mathbf R_T)
		]
		\\
		&\quad =
		\mathbf E_{\mathbf x}^{(\beta,k)}
		\bigg[
			\exp\set{
				\int_0^T
				\tttonde{
					e^{-F_\beta}\cL^{(k)}e^{F_\beta}(\mathbf R_s)
					-
					k\tonde{\cosh(\beta)-1}
				}\,
				\dd s
			}
			\phi_{\beta,N,t}(\mathbf R_T)
		\bigg]\fstop
		\end{aligned}
	\end{equation}
	By \eqref{eq:independent-k-eigenvalue},
	\begin{equation}
		e^{-F_\beta}\cL^{(k)}e^{F_\beta}
		-
		k\tonde{\cosh(\beta)-1}
		=
		e^{-F_\beta}
		\ttonde{\cL^{(k)}-\cL^{\otimes k}}
		e^{F_\beta}\fstop
	\end{equation}
	Therefore, the exponential functional in
	\eqref{eq:girsanov-step-moment} is exactly
	$\mathscr E_{\beta,k}(T)$ from \eqref{eq:G-beta-k}. Taking $T=tN$
	and substituting into \eqref{eq:moment-start-beta} proves
	\eqref{eq:moment-representation-beta}.
\end{proof}

Specializing Proposition~\ref{cor:moments} to
$\beta=\beta_{\lambda,N}$ and $\mathbf x=(0,\ldots,0)$ recovers the moment
representation for the original fields $\cY_{\lambda,N}$ in
\eqref{eq:tilted-field}: as
$F_{\beta_{\lambda,N}}(\mathbf 0)=0$ and
$tN\sinh(\beta_{\lambda,N})=tN^\lambda$,
\begin{equation}\label{eq:moment-representation-lambda-specialized}
	\E\sqbra{
		\prod_{i=1}^k
		\cY_{\lambda,N}(t,\phi_i)
	}
	=
	\mathbf E_{\mathbf 0}^{(\beta_{\lambda,N},k)}
	\sqbra{
		\mathscr E_{\beta_{\lambda,N},k}(tN)
		\prod_{i=1}^k
		\phi_i\tttonde{
			N^{-1/2}\tttonde{R_{tN}^i-tN^\lambda}
		}
	}\fstop
\end{equation}
\begin{remark}[Discrete-time viewpoint]\label{rem:discrete-time-moment}
	Proposition~\ref{cor:moments} can also be obtained as the continuous-time
	limit of the discrete-time tilting procedure used in
	\cite{parekh2024hierarchy,drillick_parekh_random_2025}. We record this only
	to connect the present notation with that approach.

	Fix $\epsilon>0$ and set
	\begin{equation}\label{eq:Z-eps-beta-k}
		W_\epsilon^{(\beta,k)}(\mathbf x)
		\eqdef
		\sum_{\mathbf z\in\Z^k}
		e^{F_\beta(\mathbf z)-F_\beta(\mathbf x)}
		\pp_\epsilon^{(k)}(\mathbf x,\mathbf z)
		=
		\mathbf E_{\mathbf x}^{(k)}
		\big[
			e^{F_\beta(\mathbf R_\epsilon)-F_\beta(\mathbf x)}
		\big]\fstop
	\end{equation}
	The discrete-time tilted transition probability in \eqref{eq:discrete} is
	then
	\begin{equation}\label{eq:q-eps-beta-k}
		q_\epsilon^{(\beta,k)}(\mathbf x,\mathbf y)
		\eqdef
		\frac{
			e^{F_\beta(\mathbf y)-F_\beta(\mathbf x)}
			\pp_\epsilon^{(k)}(\mathbf x,\mathbf y)
		}{
			W_\epsilon^{(\beta,k)}(\mathbf x)
		}\fstop
	\end{equation}
	If $T=n\epsilon$, the corresponding discrete-time change of measure gives,
	for every bounded function $\Psi:\Z^k\to\R$,
	\begin{equation}\label{eq:discrete-time-girsanov-moment}
		\begin{aligned}
		&\mathbf E_{\mathbf x}^{(k)}
		\big[
			e^{F_\beta(\mathbf R_T)-F_\beta(\mathbf x)}
			\Psi(\mathbf R_T)
		\big]
		=
		\mathbf E_{\mathbf x}^{(\epsilon,\beta,k)}
		\bigg[
			\exp\bigg\{
				\sum_{m=0}^{n-1}
				\log W_\epsilon^{(\beta,k)}(\mathbf R_{m\epsilon})
			\bigg\}
			\Psi(\mathbf R_T)
		\bigg]\fstop
		\end{aligned}
	\end{equation}
	Here, $\mathbf E_{\mathbf x}^{(\epsilon,\beta,k)}$ denotes expectation for
	the discrete-time chain with transition kernel
	$q_\epsilon^{(\beta,k)}$ and time mesh $\epsilon$.

	Now use the small-time expansion of the $k$-point semigroup. Since the
	$k$-point motion has bounded jump rates and $F_\beta$ has bounded
	increments along allowed jumps,
	\begin{equation}\label{eq:Z-eps-beta-expansion}
		W_\epsilon^{(\beta,k)}(\mathbf x)
		=
		1+
		\epsilon\,
		e^{-F_\beta}(\mathbf x)
		\cL^{(k)}e^{F_\beta}(\mathbf x)
		+
		O_\beta(\epsilon^2)\fstop
	\end{equation}
	uniformly in $\mathbf x\in\Z^k$. Hence
	\begin{equation}\label{eq:log-Z-eps-beta-expansion}
		\log W_\epsilon^{(\beta,k)}(\mathbf x)
		=
		\epsilon\,
		e^{-F_\beta}(\mathbf x)
		\cL^{(k)}e^{F_\beta}(\mathbf x)
		+
		O_\beta(\epsilon^2)\fstop
	\end{equation}
	Moreover, \eqref{eq:q-eps-beta-k} gives
	\begin{equation}
		q_\epsilon^{(\beta,k)}(\mathbf x,\mathbf y)
		=
		\ind_{\{\mathbf x=\mathbf y\}}
		+
		\epsilon \cL^{(\beta,k)}(\mathbf x,\mathbf y)
		+
		O_\beta(\epsilon^2)\comma
		\qquad \mathbf x,\mathbf y\in\Z^k\fstop
	\end{equation}
	Thus, as $\epsilon\downarrow0$, the discrete-time chain with kernel
	$q_\epsilon^{(\beta,k)}$ and mesh $\epsilon$ converges to the continuous-time
	$\beta$-tilted $k$-point motion.

	Taking in \eqref{eq:discrete-time-girsanov-moment}
	\begin{equation}
		\Psi(\mathbf z)
		=
		\exp\set{-kT\tonde{\cosh(\beta)-1}}\,
		\prod_{i=1}^k
		\phi_i\tttonde{
			N^{-1/2}\tttonde{z_i-tN\sinh(\beta)}
		}\comma
		\qquad T=tN\fstop
	\end{equation}
	and using
	$
		e^{-F_\beta}\cL^{\otimes k}e^{F_\beta}
		=
		k\tonde{\cosh(\beta)-1}
	$
	show that the discrete exponential converges to
	\begin{equation}
		\exp\set{
			\int_0^T
			e^{-F_\beta}(\mathbf R_s)
			\ttonde{\cL^{(k)}-\cL^{\otimes k}}
			e^{F_\beta}(\mathbf R_s)
			\,\dd s
		}
		=
		\mathscr E_{\beta,k}(T)\fstop
	\end{equation}
	This recovers exactly the continuous-time moment representation in
	Proposition~\ref{cor:moments}.
\end{remark}
With the moment representation proved in Proposition \ref{cor:moments}, we see that computing 
\begin{equation}\label{eq:k-discrepancy}
e^{-F_\beta}
(\cL^{(k)}-\cL^{\otimes k})e^{F_\beta}\comma\qquad k\ge 2\comma
\end{equation}
and then Taylor-expanding near $\beta=0$ is the next crucial step. This will be the main focus of the subsequent two sections.

\section{Second moment asymptotics}\label{sec:2nd-moment}
In this section, we exploit the moment formulas \eqref{eq:moment-representation-beta} and \eqref{eq:moment-representation-lambda-specialized} to analyze asymptotics of the second moment of the rescaled field $\cY_{\lambda,N}$, and prove that it converges to the second moment of the multiplicative SHE. The main result of the section is the functional CLT in Theorem \ref{th:Dobrushin-k=2}, from which we derive Proposition \ref{pr:second-moment}. The main idea will be to compute the discrepancy in \eqref{eq:k-discrepancy} for $k=2$. As this turns out to be a functional of the gap process from Section \ref{sec:gapprocesses}, our task reduces to prove local-time limit theorems for additive functionals of these processes. However, as anticipated in Section \ref{sec:dobrushin}, a subtle cancellation complicates the analysis.

 \subsection{Two-point discrepancy and Green function}
 We start by computing the discrepancy in \eqref{eq:k-discrepancy} for $k=2$. This is the first occurrence of a functional of the gap process. We remind the reader that $\pi(x)=1$ is a reversible invariant measure for the gap process on $\Bbb Z$.
\begin{lemma}[Two-point discrepancy]\label{lem:two-point-discrepancy}
	Fix $\beta\in\R$ and set
	\begin{equation}\label{eq:c-beta}
		\mathfrak c_\beta\eqdef \cosh(\beta)\fstop
	\end{equation}
	For $\mathbf x=(x_1,x_2)\in\Z^2$, the two-point discrepancy  $e^{-F_\beta}(\mathbf x)
		\ttonde{\cL^{(2)}-\cL^{\otimes2}}e^{F_\beta}(\mathbf x)$ in \eqref{eq:k-discrepancy} with $k=2$
	depends only on the gap
	$x\eqdef x_1-x_2\in \Z$ and is given by
	\begin{equation}\label{eq:D-beta-gap-formula}
    \begin{aligned}
		\tonde{\mathfrak c_\beta-1}
		\set{
			\mathfrak c_\beta \ind_{\{x=0\}}
			-
			\frac12\ind_{\{x=1\}}
			-
			\frac12\ind_{\{x=-1\}}
		}= \tonde{\mathfrak c_\beta-1}\Phi_1(x)+\tonde{\mathfrak c_\beta-1}^2\Phi_2(x)\comma
        \end{aligned}
	\end{equation}
	where
	\begin{equation}\label{eq:f-beta-1-2}
		\Phi_1(x)
		\eqdef
			\ind_{\{x=0\}}
			-
			\frac12\ind_{\{x=1\}}
			-
			\frac12\ind_{\{x=-1\}}\comma\qquad \Phi_2(x)=\car_{\{x=0\}}\fstop
	\end{equation}
    Moreover, the terms $\Phi_1$ and $\Phi_2$ satisfy \begin{equation}\label{eq:mean-zero-mean-one}
    \sum_{x\in \Z}\Phi_1(x)\,\pi(x)=0\qquad \text{and}\qquad \sum_{x\in \Z}\Phi_2(x)\,\pi(x)=1\fstop
\end{equation}
\end{lemma}
\begin{proof}
	Write $\mathfrak c=\mathfrak c_\beta$. Since $\cL^{(2)}$ and $\cL^{\otimes2}$ agree
	when the two coordinates are at distance at least two,
	the discrepancy vanishes if $|x_1-x_2|\ge2$.

	If $x_1=x_2$, then the two particles are affected simultaneously by the two
	adjacent clocks. The tilted exponential increments are $\pm\beta$ for the
	single-particle moves and $\pm2\beta$ for the simultaneous moves. Hence,
	\begin{equation}
		e^{-F_\beta}\cL^{(2)}e^{F_\beta}(\mathbf x)
		=
		\frac12(e^\beta-1)
		+
		\frac12(e^{-\beta}-1)
		+
		\frac14(e^{2\beta}-1)
		+
		\frac14(e^{-2\beta}-1)\fstop
	\end{equation}
	Therefore,
	\begin{equation}
		e^{-F_\beta}\cL^{(2)}e^{F_\beta}(\mathbf x)
		=
		(\mathfrak c-1)+\frac12\tonde{\cosh(2\beta)-1}
		=
		\mathfrak c^2+\mathfrak c-2\fstop
	\end{equation}
	In view of \eqref{eq:independent-k-eigenvalue}, we get
	\begin{equation}
		e^{-F_\beta}(\mathbf x)\ttonde{\cL^{(2)}-\cL^{\otimes 2}}e^{F_\beta}(\mathbf x)
		=
		\mathfrak c\tonde{\mathfrak c-1}\comma
		\qquad\text{if}\ x_1=x_2\fstop
	\end{equation}

	It remains to consider $|x_1-x_2|=1$. By symmetry it is enough to take
	$x_1=x_2+1$. The outer clocks give the same contribution as independent
	motions. The common clock can either move exactly one particle, producing
	exponential increments $\pm\beta$, or swap the two particles, producing
	increment $0$. Hence,
	\begin{equation}
		e^{-F_\beta}\cL^{(2)}e^{F_\beta}(\mathbf x)
		=
		(\mathfrak c-1)+\frac12(\mathfrak c-1)
		=
		\frac32(\mathfrak c-1)\fstop
	\end{equation}
	Subtracting the independent value $2(\mathfrak c-1)$ gives
	\begin{equation}
		e^{-F_\beta}(\x)\ttonde{\cL^{(2)}-\cL^{\otimes 2}}e^{F_\beta}(\x)
		=
		-\frac12(\mathfrak c-1)\comma
		\qquad\text{if}\  |x_1-x_2|=1\fstop
	\end{equation}
	This provides the formula in \eqref{eq:D-beta-gap-formula}. 
\end{proof}

For later purposes, we construct the Green function for the two-point gap process. 

\begin{lemma}[Green function of the tilted gap process]\label{lem:green-gap-beta}
	Fix $\beta\in\R$ and recall $\mathfrak c_\beta\eqdef \cosh(\beta)$ from \eqref{eq:c-beta}.
	Let $\cL_\beta^{\rm gap}$ be the generator of the tilted two-point gap
	process from Proposition~\ref{pr:gapprocess-tilt}. Define, for
	$x,y\in\Z$, with the convention $\operatorname{sgn}(0)\eqdef0$,
	\begin{equation}\label{eq:green-gap-beta-explicit}
		G_\beta(x,y)
		\eqdef
		\frac{|x-y|}{2\mathfrak c_\beta}
		+
		\frac{\mathfrak c_\beta-1}
		{2\mathfrak c_\beta(\mathfrak c_\beta+1)}
		\abs{\operatorname{sgn}(x)-\operatorname{sgn}(y)}
		+
		\frac{1-\ind_{\{x=y=0\}}}
		{\mathfrak c_\beta(\mathfrak c_\beta+1)}\fstop
	\end{equation}
	Then, $G_\beta$ is symmetric, nonnegative, and, for every $x, y\in\Z$,
	\begin{equation}\label{eq:green-gap-beta-equation}
		\cL_\beta^{\rm gap}G_\beta(x,\emparg)(y) =\ind_{\{x=y\}} = \cL_\beta^{\rm gap}G_\beta(\emparg,y)(x)
		\fstop
	\end{equation}
	In particular, if $f:\Z\to\R$ has finite support and satisfies $\sum_{x\in\Z} f(x)\,\pi(x)=0$,
	then
	\begin{equation}
		G_\beta f(x)
		\eqdef
		\sum_{y\in\Z}G_\beta(x,y)f(y)\comma
		\qquad x\in\Z\comma
	\end{equation}
	is constant on $\Z_{>0}\setminus {\rm supp}(f)$ and $\Z_{<0}\setminus {\rm supp}(f)$, bounded, and satisfies $\cL_\beta^{\rm gap}G_\beta f=f$.
\end{lemma}

\begin{proof} 
	Write $\mathfrak c=\mathfrak c_\beta$, and set
	\begin{equation}
		\mathfrak a=\mathfrak a_\beta\eqdef
		\tonde{\mathfrak c-1}/\tonde{2\mathfrak c(\mathfrak c+1)}\comma
		\qquad
		\mathfrak b=\mathfrak b_\beta\eqdef
		\tonde{\mathfrak c(\mathfrak c+1)}^{-1}\fstop
	\end{equation}
	Since $\cL_\beta^{\rm gap}$ annihilates constants, it is enough to prove the
	claim for
	\begin{equation}
		G'_\beta(x,y)
		\eqdef
		\frac{|x-y|}{2\mathfrak c}
		+
		\mathfrak a\abs{\operatorname{sgn}(x)-\operatorname{sgn}(y)}
		-
		\mathfrak b\ind_{\{x=y=0\}}\fstop
	\end{equation}
	Indeed, $G_\beta(x,y)=G'_\beta(x,y)+\mathfrak b$.

	Fix $x\in\Z$ and view $G'_\beta(x,\emparg)$ as a function of $y$. If
	$y\notin\set{-1,0,1}$, then the sign and indicator terms are locally
	constant in $y$. Hence, away from $y=x$, the function is affine and
	\begin{equation}
		\cL_\beta^{\rm gap}G'_\beta(x,\emparg)(y)=0\fstop
	\end{equation}
	If $y=x$ and $|x|\ge2$, then the only contribution comes from
	$|x-y|$, and
	\begin{equation}
		\cL_\beta^{\rm gap}G'_\beta(x,\emparg)(x)
		=
		\mathfrak c\tonde{\frac{1}{2\mathfrak c}+\frac{1}{2\mathfrak c}}
		=
		1\fstop
	\end{equation}

	It remains to check $y\in\set{-1,0,1}$. This is a direct substitution into
	the rates of Proposition~\ref{pr:gapprocess-tilt}. Indeed, one obtains 
	$\cL_\beta^{\rm gap}G'_\beta(x,\emparg)(y)=\car_{\{x=y\}}$. We spell out the details of a few cases, and leave the remaining ones to the reader. For instance, if $x\le-2$, then at $y=-1$ the contribution of
	$|x-y|/(2\mathfrak c)$ is $(1-\mathfrak c)/(4\mathfrak c)$, while the sign term contributes
	\begin{equation}
		\mathfrak a\tonde{\frac{\mathfrak c}{2}+\frac12}
		=
		\frac{\mathfrak c-1}{4\mathfrak c}\comma
	\end{equation}
	so the two terms cancel. The cases $y\in \{0,1\}$ are simpler. If $x=0$, then at $y=0$ the three contributions are
	\begin{equation}
		\frac12\comma
		\frac{\mathfrak c-1}{2(\mathfrak c+1)}\comma
		\frac{1}{\mathfrak c+1}\comma
	\end{equation}
	whose sum is $1$. 
    
    This proves the first identity in \eqref{eq:green-gap-beta-equation}.
	As the formula \eqref{eq:green-gap-beta-explicit} is symmetric in $(x,y)$ and
	 $\cL_\beta^{\rm gap}$ is self-adjoint with respect to counting measure, the second identity $\cL_\beta^{\rm gap}G_\beta'(\emparg,y)(x)=\car_{\{x=y\}}$ in 
	\eqref{eq:green-gap-beta-equation} follows as well.

	Finally, if $f$ has finite support and zero mean with respect to the counting
	measure, then the linear growth of $G_\beta(x,y)$ in $x$ cancels in
	$\sum_yG_\beta(x,y)f(y)$. Hence, $G_\beta f$ is eventually constant and bounded. Using
	\eqref{eq:green-gap-beta-equation},
	\begin{equation}
		\cL_\beta^{\rm gap}G_\beta f(x)
		=
		\sum_{y\in\Z}
		\cL_\beta^{\rm gap}G_\beta(\emparg,y)(x)f(y)
		=
		\sum_{y\in\Z}\ind_{\{x=y\}}f(y)
		=
		f(x)\fstop
	\end{equation}
	This concludes the proof.
\end{proof}

\begin{remark}\label{rem:green-function-untilted} 
As $\beta\to 0$, $G_\beta$ in \eqref{eq:green-gap-beta-explicit} pointwise converges to
\begin{equation}
    G(x,y)\eqdef \frac{|x-y|}{2}+\frac{1-\car_{\{x=y=0\}}}2\comma\qquad x, y\in \Z\comma
\end{equation}
the Green function of the untilted gap process.
Hence,  for any finite-support mean-zero function $f:\Z\to \R$ and $|\beta|\le 1$,  one gets
\begin{equation}
\sup_{x\in \Z}\abs{G f(x)-G_\beta f(x)}\le C_f\beta^2\comma\quad \text{with}\   Gf\eqdef\sum_{y\in \Z}G(\emparg,y)f(y)\fstop
\end{equation}

Further, let $\Gamma_\beta$ (with $\Gamma=\Gamma_0$) denote the corresponding $\beta$-tilted carré-du-champ, that is, 
\begin{equation}\label{eq:carre-du-champ-def}
\Gamma_\beta(h_1,h_2)\eqdef \cL_\beta^{\rm gap}(h_1 h_2)-h_1\cL_\beta^{\rm gap} h_2 - h_2\cL_\beta^{\rm gap}h_1\comma\qquad h_1,h_2:\Z\to \R\ \text{bounded}\fstop
\end{equation}
Then, there exists a $\varrho_f >0$ such that ${\rm supp}(f)\subset [-\varrho_f,\varrho_f]$ and
\begin{equation}\label{eq:carre-du-champ-conv}
|\Gamma(Gf,Gf)(x)-\Gamma_\beta(G_\beta f,G_\beta f)(x)|\le C_f\beta^2 \car_{|x|\le \varrho_f}\comma\qquad x \in \Z\fstop
\end{equation}
\end{remark}

\subsection{Technical lemmas from \cite{parekh2024hierarchy,drillick_parekh_random_2025}}
The next two results are continuous-time analogues of the results proved in the discrete-time setting in
\cite{parekh2024hierarchy,drillick_parekh_random_2025}. They concern additive functionals of the
tilted $k$-point motion, with $k\ge2$. Remark that here we do
not distinguish between mean-zero and nonzero-mean functionals.

The only structural input needed from the model is the short-range interaction (SRI)
condition in \cite{drillick_parekh_random_2025}. In the present setting of the averaging process, this condition is immediate: away from a fixed
neighborhood of the collision hyperplanes $\{x_i=x_j\}$, the generator of the
$\beta$-tilted $k$-point motion agrees with that of $k$ independent tilted
nearest-neighbor walks, while near the collision hyperplanes the rates are
modified only locally and remain uniformly bounded for $\beta$ in a neighborhood
of zero. Thus, the discrete-time arguments carry over to continuous time by
replacing transition kernels with semigroup estimates and sums with time
integrals. For convenience, we record
the needed statements directly in the present continuous-time setup.

We first state the occupation-time bound used below for tightness and uniform
integrability. The main estimate, namely \eqref{eq:additive-functional-moment}, is the continuous-time analogue of \cite[Eq.\ (B.4)]{drillick_parekh_random_2025}.

\begin{lemma}[\cite{drillick_parekh_random_2025}]\label{vbound}
    Fix $k \geq 2$ and a finite-support function $h:\Z\to\R$. Then, there exist $C = C(k,h),\varepsilon = \varepsilon(k,h) > 0$ such that, for all $\beta  \in (-\varepsilon, \varepsilon)$, $0\le s\le t$, $1 \leq i, j \leq k$, and $p\ge 1$, 
    \begin{equation}\label{eq:additive-functional-moment}\sup_{\x\in\Z^k} \mathbf E_\x^{(\beta,k)} \bigg[ \bigg( \int_s^t |h(R_r^i -R_r^j )|\,{\rm d}r \bigg)^p\bigg]^{1/p} \leq C\sqrt p \abs{t-s}^{1/2}\fstop
    \end{equation}
    As a consequence, Taylor expanding the exponential yields, for every $a\ge 0$,
    \begin{equation}\label{eq:additive-functional-exp}\sup_{t>0} \sup_{\x\in\Bbb Z^k} \mathbf E_\x^{(\beta,k)}\bigg[ \exp\bigg\{a t^{-1/2} \int_0^t |h(R_r^i -R_r^j )|\,{\rm d}r \bigg\}\bigg]<\infty\fstop
    \end{equation}
\end{lemma}

The second result is the corresponding invariance principle for additive
functionals of pair gaps. It is the continuous-time counterpart of
\cite[Theorem~4.6]{drillick_parekh_random_2025}. In particular, tightness of the additive functional in $\cC([0,\infty),\R)$ follows from the bounds in Lemma \ref{vbound}. The identification of subsequential limits is the Krylov--Bogoliubov
argument used in \cite[Proposition~3.10]{parekh2024hierarchy}. We remark that in the present setting, the invariant measure $\pi^{\rm inv}$ is the counting measure $\pi$ on $\Z$.
\begin{lemma}[\cite{drillick_parekh_random_2025}]\label{vprocess}
     Fix  $k \geq 2$, a finite-support function $h:\Z\to\R$, and a sequence $\beta=\beta_N\to 0$ as $N\to \infty$. Consider $\x^N\in \Z^k$ such that $N^{-1/2}\x^N\to \mathbf z \in \R^k$ as $N\to \infty$. Then, for all $1\le i,j\le k$, $i\neq j$, the pair
     \begin{equation}
      \tonde{N^{-1/2}\tttonde{R^i_{tN} - R^j_{tN}},N^{-1/2} \textstyle{\int_0^{tN}} h(R^i_s - R^j_s)\,\dd s }_{t\ge 0}\comma\quad\text{under}\ \mathbf P_{\x^N}^{(\beta_N,k)}\comma
    \end{equation} weakly converges in  $\cD([0,\infty);\R^2)$ as $N\to \infty$ to 
    \begin{equation}\label{eq:W+local-time}
    (W_t, \gamma^2(h)\, L_t^W )_{t\ge 0}\comma\quad \text{with}\ \gamma^2(h)\eqdef \textstyle{\sum_{x\in \Z}}\, h(x)\comma
    \end{equation}where $W$ is a Brownian motion with $W_0=z_i-z_j\in \R$, and $L^W$ is its local time at zero. 
\end{lemma}

\subsection{Dobrushin-type theorem for $k=2$}
We now have all the ingredients for stating and proving the main result of this section.

\begin{theorem}[Dobrushin-type invariance principle: $k=2$]\label{th:Dobrushin-k=2}
Let $f, h:\Z\to \R$ be two finite-support functions, with $f$ additionally satisfying $\sum_{x\in \Z}f(x)\,\pi(x) = 0$. 
   Then, for any sequences $\beta=\beta_N\to 0$ and $N^{-1/2}\mathbf x^N \in N^{-1/2}\Z^2\to \mathbf z\in \R^2$ as $N\to \infty$,  the following triple (all coordinates being expressed in terms of the two-point gap process $X=R^1-R^2$)
   \begin{equation}
\tonde{N^{-1/2}X_{tN},N^{-1/2}\textstyle{\int_0^{tN}}h(X_s)\,\dd s, N^{-1/4}\int_0^{tN}f(X_s)\,\dd s}_{t\ge 0}\comma\quad\text{under}\ \mathbf   P_{\x^N}^{(\beta_N,2)}\comma
   \end{equation} weakly converges in $\cD([0,\infty);\R^3)$ as $N\to \infty$ to 
   \begin{equation}
    \tonde{W_t,\gamma^2(h) L_t^W, \sigma(f) B_{L_t^W}}_{t\ge 0}\comma
   \end{equation}
   where the first two coordinates are given as in \eqref{eq:W+local-time}, 
   $B$ is another Brownian motion independent from $W$ and such that $B_0=0$, and \begin{equation}\label{eq:sigma-squared}
    \sigma^2(f)\eqdef \textstyle{\sum_{x\in \Z}}\,\Gamma(Gf,Gf)(x)\fstop
   \end{equation}
   Here, $G$ and $\Gamma$ are, respectively, the Green function and the carré-du-champ (Remark \ref{rem:green-function-untilted}) of the untilted gap process.
\end{theorem}
\begin{proof}
    For each $N\in \N$, let $g_N:=G_{\beta_N}f:\Z\to \R$. By Remark \ref{rem:green-function-untilted}, the functions $g_N$ are uniformly bounded and uniformly converge, as $N\to \infty$, to $g\eqdef Gf$. 
        
In this proof, we simply write $\mathbf P^N= \mathbf P_{\x^N}^{(\beta_N,2)}$, and $\mathbf E^N$ for the corresponding expectation.
    Since $f$ has mean zero, Lemma \ref{lem:green-gap-beta} and Dynkin's formula ensure that
    \begin{equation}M_t^N:= g_N(X_t)- \textstyle{\int_0^t}\, f(X_s)\,\dd s\comma\qquad t \ge 0\comma
    \end{equation}
    is a martingale in the natural filtration of the gap process. Since $g_N$ is uniformly bounded in $N$, $N^{-1/4} (M^N_{tN} + \int_0^{tN} f(X_s)\, \dd s )_{t\ge 0}$
    vanishes in $\cD([0,\infty);\R)$ as $N\to \infty$. Thus, our main task reduces to analyze the joint limit of $(N^{-1/2} X_{tN}, N^{-1/4} M^N_{tN})_{t\ge 0}$. The advantage here is that this pair is an $\R^2$-valued martingale in the joint filtration, a property which turns out to be useful for calculations. 

    The predictable quadratic variation of the second marginal equals 
    \begin{equation}\label{e:fil1}\langle N^{-1/4}M^N \rangle_{tN}
    = N^{-1/2}\textstyle{\int_0^{tN}}\,\Gamma_N(g_N,g_N)(X_s)\,\dd s\comma\qquad t \ge 0\comma\end{equation}
    where $\Gamma_N=\Gamma_{\beta_N}$ denotes the carré-du-champ associated to $\cL_{\beta_N}^{\rm gap}$, see \eqref{eq:carre-du-champ-def}. Since $g_N$ is eventually constant (Lemma \ref{lem:green-gap-beta}) and $\cL_{\beta_N}^{\rm gap}$ has finite-range jumps, $\Gamma_N(g_N,g_N)$ is supported on some finite set not depending on $N$. Moreover, by \eqref{eq:carre-du-champ-conv} and Lemma \ref{vprocess}, 
    \begin{equation}\label{eq:quad-var-approx}
    \sup_{t\le T}|\langle N^{-1/4}M^N\rangle_{tN}- N^{-1/2}\textstyle{\int_0^{tN}}\,\Gamma(g,g)(X_s)\,\dd s |\xLongrightarrow[N\to \infty]{}0\comma\qquad T>0\fstop
    \end{equation}

    Hence, for any $T, \delta>0$ and two-point motion stopping time $\tau=\tau^N\in [0,TN]$, 
    \begin{align}
    \mathbf E^N[(N^{-1/4}M_{\tau+\delta N}^N-N^{-1/4}M_\tau^N)^2] &= N^{-1/2}\mathbf E^N[\langle M^N\rangle_{\tau+\delta N}-\langle M^N\rangle_\tau]\\
    &=N^{-1/2}\mathbf E^N[\textstyle{\int_\tau^{\tau+\delta N} \Gamma_N(g_N,g_N)(X_s)\,\dd s}] \\
    &\le C \delta^{1/2}\comma
    \end{align}
    where for the last step we used the strong Markov property of the gap process, \eqref{eq:carre-du-champ-conv}, and the bound in \eqref{eq:additive-functional-moment} with $p=1$. By Aldous' stopping time criterion (see, e.g., \cite[Theorem 1]{aldous_stopping_1978}), taking $\delta\to 0$ proves that $((N^{-1/4}M_{tN}^N)_{t\ge 0})_N$ is a tight sequence in $\cD([0,\infty);\R)$. Further, by, e.g., \cite[Theorem 13.4]{billingsley_convergence_1999}), any limit point lies in $\cC([0,\infty);\R)$ as, almost surely,
    \begin{equation}
    \sup_{t>0}|N^{-1/4}M_t^N-N^{-1/4}M_{t^-}^N|\le 2N^{-1/4}\sup_{x\in \Z} |g_N(x)|\xrightarrow[N\to \infty]{}0\fstop
    \end{equation}

By Lemma \ref{vprocess}, the first marginal
    $(N^{-1/2}X_{tN})_{t\ge 0}$ converges to the Brownian motion $W$. To identify the joint limit with $N^{-1/4} M_{Nt}$,  we first calculate the predictable quadratic covariance:
    \begin{equation}\label{cov1}\mathbf E^N[|\langle N^{-1/4}M^N, N^{-1/2}X \rangle_{tN}|] = N^{-3/4} \mathbf E^N[ | \textstyle{\int_0^{tN}}\,\Gamma_N(g_N,{\rm id}_\Z)(X_s)\,\dd s|] = O(N^{-1/4})\comma
    \end{equation}where ${\rm id}_\Z$ is the identity function on $\Z$, while the last estimate follows from Lemma \ref{vbound} as $\Gamma_N(g_N,\mathrm{id}_\Z)$ is finitely supported and uniformly bounded in $N$.

    Now consider any joint limit $(X^\infty, M^\infty)$ of the pair $(N^{-1/2} X_{tN}, N^{-1/4} M^N_{tN})_{t\ge 0}$. By the above discussion and uniform integrability of the pre-limit, this limit process is a continuous $\R^2$-valued martingale. Let $\mathbf P$ and $\mathbf E$ indicate law and expectation. As already mentioned, $X^\infty$ is marginally distributed as $W$. We now show that, conditionally on $X^\infty$, the second marginal $M^\infty$ is an independent Gaussian process with  the correct covariance structure. 
    
    We prove this by first showing that 
    \begin{equation}\label{e:fil2}\mathbf E [ e^{a M_t^\infty} |X^\infty] = \exp\set{\tfrac12 a^2 \sigma^2(f) L_t^{X^{\infty}}}\comma\qquad a,  t \ge 0\fstop
    \end{equation}
    This means that, conditionally on $X^\infty$, the random variable $M_t^\infty$ is a Gaussian random variable of variance $\sigma^2(f)L_t^{X^\infty}$. To prove it,  we use an argument from \cite{das_drillick_parekh_kpz_2024}, based on It\^o's formula. Fix $T>0$, and take any bounded $\mathcal F_T(X^\infty)$-measurable functional $F:\cC([0,T],\R)\to \R$.  Note that $t\in [0,T]\mapsto \mathbf E[F(X^\infty)|\mathcal F_t(X^\infty)]$ is a martingale in the filtration of $X^\infty$. As $X^\infty$ is a standard Brownian motion, the martingale representation theorem yields
		\begin{equation}
		F(X^\infty)=\mathbf E [F(X^\infty)]+\int_0^T h_s\, \dd X_s^\infty\comma
		\end{equation}
		for some $\R$-valued adapted process $h$. For the stochastic exponential, we have 
\begin{equation}\label{exp_rep}\exp\left(a M_T^\infty - \frac{a^2}2 \langle M^\infty\rangle_T\right) = 1+a \int_0^T \exp\bigg(a M_s^\infty - \frac{a^2}2 \langle M^\infty\rangle_s\bigg) \dd M_s^\infty\fstop
        \end{equation}
        Using this, we find that 
		\begin{align} &\mathbf E\bigg[\bigg(\exp\set{a M_T^\infty - \frac{a^2}2 \langle M^\infty\rangle_T} -1\bigg)\big(F(X^\infty)-\mE[F(X^\infty)]\big)\bigg] \\
        &= a \mE\bigg[\bigg(\int_0^T \exp\bigg(a M_s^\infty - \frac{a^2}2 \langle M^\infty\rangle_s\bigg) \dd M_s^\infty\bigg)\bigg(\int_0^T h_s\, \dd X_s^\infty\bigg)\bigg]
        \\&=  a \mE\bigg[\int_0^T h_s\, \exp\set{a M_s^\infty - \frac{a^2}2 \langle M^\infty\rangle_s} \dd \langle M^\infty,X^\infty\rangle_s\bigg] = 0\comma
		\end{align}
		where the last equality is due to the fact that $\langle X^\infty,M^\infty\rangle \equiv 0$ almost surely by \eqref{cov1}. On the one hand, this proves  \begin{equation}\mE[\exp\ttset{M_T^\infty - \tfrac12 \langle M^\infty\rangle_T}F(X^\infty)]=\mE[F(X^\infty)]\end{equation} 
        for all bounded and measurable $F$. On the other hand, \eqref{e:fil1}--\eqref{eq:quad-var-approx} and Lemma \ref{vprocess} imply that $\langle M^\infty\rangle = \sigma^2(f) L^{X_\infty}$, again using that martingality must be preserved by limit points. This establishes \eqref{e:fil2}.

        To finish the proof, we claim more generally than \eqref{e:fil2} that, for disjoint intervals $(t_{j-1},t_j)\subset [0,\infty)$  and $a_j\in \R$, one has $\mathbf E[\exp\ttset{ \sum_{j=1}^n a_j (M_{t_j}^\infty - M_{t_{j-1}}^\infty)}| X^\infty] = \exp\ttset{\frac12 \sigma^2(f) \sum_{j=1}^n a_j^2 (L_{t_j}^{X^\infty} - L_{t_{j-1}}^{X^\infty}) }$. This gives the claim as it implies that, conditionally on $X^\infty$, the increments $M_{t_j}^\infty - M_{t_{j-1}}^\infty$ are independent Gaussians of the correct variance. The proof is similar to \eqref{e:fil2} and is omitted.
       
                Finally, the convergence of the entire triple follows from Lemma \ref{vprocess}. Indeed, the second coordinate
        $ N^{-1/2}\int_0^{tN} h(X_s)\,\dd s
        $ is measurable with respect to the path $(N^{-1/2}X_{tN})_{t\ge 0}$, and Lemma \ref{vprocess} identifies its joint limit with this first coordinate as $\gamma^2(h)L^W$. Hence, every subsequential limit of the three-dimensional process has first two coordinates $(W,\gamma^2(h)L^W),$ and joint first and third coordinates $(W,\sigma(f)B_{L^W})$, with the second coordinate determined by the first one (as a measurable function of it). This uniquely identifies the joint limit, 
        completing the proof.
\end{proof}

Theorem \ref{th:Dobrushin-k=2} and Lemma \ref{lem:two-point-discrepancy} provide all the ingredients to compute asymptotics of second moments of the fields $\cY_{\lambda,N}$ in \eqref{eq:tilted-field} (and, more generally, of $\mathfrak H_{\beta_{\lambda,N},N}$ in \eqref{eq:H-beta-N}) with $\lambda = 7/8$, and evaluated at constant test functions. Its full version is stated as Proposition \ref{lim_moments} below. In what follows, $\mathbf P_{\mathbf z}^{{\rm BM}^{\otimes k}}$ and $\mathbf E_{\mathbf z}^{{\rm BM}^{\otimes k}}$ denote the path law and corresponding expectation of a $k$-dimensional isotropic Brownian motion $\mathbf W=(W^1,\ldots,W^k)$ started at $\mathbf z\in \R^k$.

\begin{proposition}[Moments asymptotics: $k=2$]\label{pr:second-moment}
   Fix $\lambda =7/8$, and recall from \eqref{eq:beta-lambda-N} that $\beta_{\lambda,N}\sim N^{-1/8}$. Then, for every $t\ge 0$ and sequence $\x^N\in \Z^2$ such that $N^{-1/2}\x^N\to \mathbf z\in \R^2$, 
   \begin{equation}\label{eq:2nd-moment-asymp}\E\bigg[ \prod_{j=1,2}\mathfrak H_{\beta_{\lambda,N},N}(t,x_j^N,1)\bigg] \xrightarrow[N\to \infty]{} \mathbf E_{(z_1,z_2)}^{{\rm BM    }^{\otimes 2}}\bigg[ \exp\set{{\frac12}  L_t^{W^1-W^2} } \bigg]\comma 
   \end{equation}
   where $1\in \cC_b(\R)$ in $\mathfrak H_{\beta_{\lambda,N},N}(t,x_j^N,1)$ indicates the constant function equal to one.
\end{proposition}
\begin{proof}
	Set $\beta_N\eqdef \beta_{\lambda,N}$, with $\lambda=7/8$. By Proposition \ref{cor:moments} and Lemma
	\ref{lem:two-point-discrepancy}, the left-hand side of \eqref{eq:2nd-moment-asymp} equals
	\begin{equation}\label{eq:step-1}
		\mathbf E_{\mathbf x^N}^{(\beta_N,2)}
		\bigg[
			\exp\set{
				\int_0^{tN}
				\tonde{
					\tonde{\mathfrak c_{\beta_N}-1}\Phi_1(X_s)
					+
					\tonde{\mathfrak c_{\beta_N}-1}^2\Phi_2(X_s)
				}
				\dd s
			}
		\bigg]\comma
	\end{equation}
	where $X=R^1-R^2$ is the two-point gap process, while
	$\mathfrak c_\beta=\cosh(\beta)$.
	Since $\beta_N\sim N^{-1/8}$,
	\begin{equation}
		\mathfrak c_{\beta_N}-1
		=
		\frac{N^{-1/4}}2+o(N^{-1/4})\comma
		\qquad
		(\mathfrak c_{\beta_N}-1)^2
		=
		\frac{N^{-1/2}}4+o(N^{-1/2})\fstop
	\end{equation}
	Thus, the exponent is
	\begin{equation}
		\frac{N^{-1/4}}2\int_0^{tN} \Phi_1(X_s)\,\dd s
		+
		\frac{N^{-1/2}}4
		\int_0^{tN}\Phi_2(X_s)\,\dd s
		+
		o_{\P}(1)\comma
	\end{equation}
	with
	\begin{equation}
		\Phi_1\eqdef
		\ind_{\{0\}}
		-
		\frac12\ind_{\{1\}}
		-
		\frac12\ind_{\{-1\}} = - \cL^{\rm gap}\car_{\{0\}}\fstop
	\end{equation}

The constant $\gamma^2(\Phi_2)=\sum_{x\in \Z}\Phi_2(x)=1$ has already been computed in \eqref{eq:mean-zero-mean-one}.
	Since $\gamma^2(\Phi_1)=0$, let us compute $\sigma^2(\Phi_1)$ given in \eqref{eq:sigma-squared}. By Remark \ref{rem:green-function-untilted},
	$G\Phi_1(x)
		=
		-\ind_{\{x=0\}}$.
	Hence, recalling \eqref{eq:sigma-squared} and Proposition \ref{gapprocess},
	\begin{equation}\label{eq:sigma-computation}
	\sigma^2(\Phi_1) 
		=\Gamma(G\Phi_1,G \Phi_1)(0)+\Gamma(G\Phi_1,G \Phi_1)(1)+\Gamma(G\Phi_1,G \Phi_1)(-1)=1+\frac12
    +\frac12=
		2\fstop
	\end{equation}

	By Theorem \ref{th:Dobrushin-k=2}, jointly with the diffusive convergence of
	the tilted two-point motion,
	\begin{equation}
		\int_0^{tN}
		\tonde{
			(\mathfrak c_{\beta_N}-1)\Phi_1(X_s)
			+
			(\mathfrak c_{\beta_N}-1)^2\Phi_2(X_s)
		}
		\,\dd s
		\xLongrightarrow[N\to\infty]{}
		{\frac1{\sqrt 2}}B_{L_t^{W^1-W^2}}
		+
		\frac14L_t^{W^1-W^2}\comma
	\end{equation}
	where $W^1,W^2$ and $B$ are three independent standard Brownian
	motions. The exponential moments are uniformly
	integrable by Lemma \ref{vbound}. Therefore, the limit as $N\to \infty$ of the expression in \eqref{eq:step-1} reads as
	\begin{equation}
		\mathbf E_{(z_1,z_2)}^{{\rm BM}^{\otimes 2}}
		\bigg[
			\mathbf E\sqbra{
				\exp\set{{\frac1{\sqrt 2}}
					B_{L_t^{W^1-W^2}}
				}
				\,\middle|\,
				W^1,W^2
			}
			\exp\set{
				\frac14L_t^{W^1-W^2}
			}
		\bigg]\fstop
	\end{equation}
	Since, for all $a\in \R$,
	\begin{equation}
		\mathbf E\sqbra{
			\exp\set{a B_{L_t^{W^1-W^2}}}
			\,\middle|\,
			W^1,W^2
		}
		=
		\exp\set{
			\frac{a^2}2L_t^{W^1-W^2}
		}\comma
	\end{equation}
	the total local-time coefficient is ${\frac12\frac12}+\frac14={\frac12}$. This proves
	the claim.
\end{proof}

\section{Higher moments' asymptotics and heat kernel estimates}\label{sec:higher-moments}
Next, we turn to the asymptotics of $k$-th moments of the tilted fields, for $k\ge 3$. As showed in Proposition \ref{sec:moment-repr}, this step involves the tilted $k$-point motion, thus, requiring a generalization of Dobrushin-type invariance principle in Theorem \ref{th:Dobrushin-k=2}. However, compared to the case of $k=2$, the analysis for $k\ge 3$ gets more complicated because of the possible occurrence of triple collisions, and by the fact that marginals of length $\ell$ of the tilted $k$-point motion do not perfectly decouple from the remaining $k-\ell$ coordinates.

\subsection{Heat kernel estimates and triple collisions}
Our main goal of the section is to control triple collisions using the following off-diagonal heat kernel upper bounds for the tilted $k$-point motion, which, as soon as $\beta\neq 0$, is nonreversible. In particular, we prove that the one-dimensional estimates in \cite[Theorem 4.7]{drillick_parekh_random_2025}, established only for the untilted (thus, reversible) case $\beta=0$, considerably extend to a wider range of tilt parameters $|\beta|\lesssim N^{-1/8}$. 
\begin{lemma}[Heat kernel upper bound: $\beta\neq 0$]\label{lem:HKUB} Fix $A>0$, $k\ge 1$, $\delta\in (0,\frac14)$, $0<t_0<T$,  and $\hslash>0$. Then, for some $C>0$ and all $\x, \y \in \Z^k$, $N\in \N$, $|\beta|\le AN^{-1/8}$, and $t\in [t_0,TN]$, 
\begin{equation}
\mathbf P_\x^{(\beta,k)}(|R_t^i-\sinh(\beta)t-y_i|\le 1,\,\forall\,1\le i\le k)\le C t^{-k/2+\delta}\ttonde{1+t^{-1/2}|\x-\y|}^{-\hslash} \fstop
\end{equation}
\end{lemma}
\begin{proof}
Set $\mathfrak s_\beta\eqdef \sinh(\beta)$ and  $\mathbf 1\eqdef (1,\ldots,1)\in\Z^k$. Recall that 
$\mathsf p_t^{(\beta,k)}(\x,\z)
=
\mathbf P_\x^{(\beta,k)}(\mathbf R_t=\z)
$ stands for
the heat kernel of the tilted $k$-point motion.

We first record the corresponding bound for the independent tilted walk. Let
$\mathbf S=(S^1,\ldots,S^k)$ be the process whose coordinates are independent
$\beta$-tilted one-point motions, with jump rates $e^\beta/2$ to the right and
$e^{-\beta}/2$ to the left. Let
$\mathbf P_\x^{(\beta,\otimes k)}$ denote its path law. The one-dimensional bound
\begin{equation}
\mathsf p_t^{(\beta,1)}(x,\lfloor \mathfrak s_\beta t\rfloor+y)
\le
C t^{-1/2}
\exp\tonde{
-c|y-x|\log\tonde{1+\frac{|y-x|}{t+1}}
}
\end{equation}
implies, by independence and for every $H>0$,
\begin{equation}\label{eq:independent-HKUB}
\mathbf P_\x^{(\beta,\otimes k)}
\tonde{
|\mathbf S_t-\mathfrak s_\beta t\mathbf 1-\y|\le r
}
\le
C r^k t^{-k/2}
\ttonde{1+t^{-1/2}|\x-\y|}^{-H}
\end{equation}
uniformly for $1\le r\le t^{1/2}$, $t\ge t_0$, and $|\beta|\le A$. In particular, the lemma is immediate when $k=1$. We therefore assume $k\ge 2$.

We now compare $\mathbf R$ with $\mathbf S$ by coupling. Let
\begin{equation}
\mathsf C_k
\eqdef
\big\{
\z\in\Z^k:\ \min_{1\le i<j\le k}|z_i-z_j|\le 1
\big\}
\end{equation}
be the collision neighbourhood. Outside $\mathsf C_k$, the tilted $k$-point
motion has the same transition rates as the independent tilted walk. We construct
a coupling as follows. When $\mathbf R_{t-}\notin\mathsf C_k$, the process
$\mathbf S$ makes exactly the same jump as $\mathbf R$; when
$\mathbf R_{t-}\in\mathsf C_k$, the process $\mathbf S$ evolves according to an
independent copy of the product tilted walk. With this construction, the marginal
law of $\mathbf S$ is precisely $\mathbf P_\x^{(\beta,\otimes k)}$.

Let
$
\mathbf X_t\eqdef \mathbf R_t-\mathbf S_t$.
Since $\mathbf X$ can change only while $\mathbf R_t\in\mathsf C_k$, its
Doob--Meyer decomposition has the form
\begin{equation}
\mathbf X_t=\mathbf M_t+\mathbf A_t\comma
\end{equation}
where $\mathbf M$ is a martingale and $\mathbf A$ is the compensator. From the
definition of the tilted generator, and using that the two generators agree outside
$\mathsf C_k$, there exists $C=C(k,A)>0$ such that
\begin{equation}\label{eq:A-bound}
|\mathbf A_t|
\le
C|\beta|\int_0^t \car_{\{\mathbf R_s\in\mathsf C_k\}}\,\dd s
\end{equation}
and, for each coordinate $1\le q\le k$,
\begin{equation}\label{eq:N-bracket-bound}
\langle M^q\rangle_t
\le
C\int_0^t \car_{\{\mathbf R_s\in\mathsf C_k\}}\,\dd s\fstop
\end{equation}
Indeed, the drift discrepancy between the interacting tilted motion and the
independent tilted motion is supported on $\mathsf C_k$ and is of size
$O(\sinh(\beta))=O(\beta)$, while the corresponding carré-du-champ discrepancy
is supported on $\mathsf C_k$ and is uniformly bounded, since $\cosh(\beta)$ is
bounded for $|\beta|\le A$.

We use the exponential moment estimate in Lemma \ref{vbound} for collision local times: for every
$B>0$,
\begin{equation}\label{eq:collision-exp-bound}
\max_{1\le i<j\le k}
\sup_{\x\in\Z^k}
\mathbf E_\x^{(\beta,k)}
\sqbra{
\exp\tonde{
B t^{-1/2}\int_0^t
\car_{\{|R_s^i-R_s^j|\le 1\}}\,\dd s
}
}
\le C(B,k)
\end{equation}
uniformly for $t\ge 1$, $N\in\N$, and $|\beta|\le AN^{-1/8}$. Since
\begin{equation}
\car_{\{\mathbf R_s\in\mathsf C_k\}}
\le
\textstyle{\sum_{1\le i<j\le k}}\,
\car_{\{|R_s^i-R_s^j|\le 1\}},
\end{equation}
the same estimate holds with the integral of
$\car_{\{\mathbf R_s\in\mathsf C_k\}}$ in place of any single pair-collision local
time, after changing the constant and using Hölder's inequality.

We first control the martingale part. By the exponential inequality for bounded-jump
martingales, for every $\theta\in(0,1)$ and $u>0$,
\begin{equation}
\mathbf P_\x^{(\beta,k)}(|M_t^q|>u)
\le
2e^{-\theta u}
\mathbf E_\x^{(\beta,k)}
\sqbra{
\exp\tonde{C\theta^2\langle M^q\rangle_t}
}\fstop
\end{equation}
Using \eqref{eq:N-bracket-bound}, choosing $\theta=t^{-1/4}$, and then applying
\eqref{eq:collision-exp-bound}, we obtain, for every $\varsigma>0$,
\begin{equation}\label{eq:N-tail}
\mathbf P_\x^{(\beta,k)}
\tttonde{
|\mathbf M_t|>t^{1/4+\varsigma}
}
\le
C\exp\tonde{-c t^\varsigma}
\end{equation}
uniformly in $\x,N,\beta$ and $t\in[t_0,TN]$, after enlarging $C$ to cover bounded
times.

We next control the compensator. From \eqref{eq:A-bound}, for every
$\theta>0$ and $u>0$,
\begin{equation}
\mathbf P_\x^{(\beta,k)}(|\mathbf A_t|>u)
\le
e^{-\theta u}
\mathbf E_\x^{(\beta,k)}
\sqbra{
\exp\tonde{
C\theta|\beta|
\int_0^t\car_{\{\mathbf R_s\in\mathsf C_k\}}\,\dd s
}
}\fstop
\end{equation}
Since $t\le TN$ and $|\beta|\le AN^{-1/8}$, we have
\begin{equation}
|\beta|\le C t^{-1/8}\fstop
\end{equation}
Choosing $\theta=t^{-3/8}$ and $u=t^{3/8+\varepsilon}$, with
$\varepsilon>0$, gives
\begin{equation}
\theta|\beta|\le C t^{-1/2}\comma
\qquad
\theta u=t^\varepsilon\fstop
\end{equation}
Therefore, by \eqref{eq:collision-exp-bound},
\begin{equation}\label{eq:A-tail}
\mathbf P_\x^{(\beta,k)}
\tttonde{
|\mathbf A_t|>t^{3/8+\varepsilon}
}
\le
C\exp\tonde{-c t^\varepsilon}\fstop
\end{equation}

We now choose
\begin{equation}
\rho\eqdef \frac{5}{12}\fstop
\end{equation}
Since $\mathbf R_t=\mathbf S_t+\mathbf M_t+\mathbf A_t$,
the bounds \eqref{eq:N-tail} with $\varsigma=1/6$, \eqref{eq:A-tail} with $\varepsilon=1/24$, and the independent estimate
\eqref{eq:independent-HKUB} yield
\begin{equation}\label{eq:coarse-HKUB-pre}
\mathbf P_\x^{(\beta,k)}
\tonde{
|\mathbf R_t-\mathfrak s_\beta t\mathbf 1-\y|\le t^\rho
}
\le
C t^{-k(1/2-\rho)}
\ttonde{1+t^{-1/2}|\x-\y|}^{-H}
+
C e^{-c t^{1/24}}\fstop
\end{equation}
Here $1/2-\rho=1/12$.

We claim that the stretched-exponential term in \eqref{eq:coarse-HKUB-pre} can
be absorbed into the first term, after changing $C$, provided $H$ is fixed large
enough. Indeed, if $|\x-\y|\le t^B$ for some large fixed $B$, then
$e^{-ct^{1/24}}$ is smaller than any negative power of $t$, and hence is bounded
by
\begin{equation}
C t^{-k/12}
\ttonde{1+t^{-1/2}|\x-\y|}^{-H}\fstop
\end{equation}
On the other hand, if $|\x-\y|>t^B$, then, by Markov's inequality,
\begin{equation}
\begin{aligned}
\mathbf P_\x^{(\beta,k)}
\tttonde{
|\mathbf R_t-\mathfrak s_\beta t\mathbf 1-\y|\le t^\rho
}
&\le
C\mathbf E_\x^{(\beta,k)}
\sqbra{
\exp\tonde{
-t^{-\rho}|\mathbf R_t-\mathfrak s_\beta t\mathbf 1-\y|
}
}
\\
&\le
C\exp\tonde{
-t^{-\rho}|\x-\y|
}
\mathbf E_\x^{(\beta,k)}
\sqbra{
\exp\tonde{
t^{-\rho}|\mathbf R_t-\x-\mathfrak s_\beta t\mathbf 1|
}
}
\\
&\le
C\exp\tonde{
-t^{-\rho}|\x-\y|+Ct
}\comma
\end{aligned}
\end{equation}
where the last step was obtained by passing to the exponential martingale for $\mathbf R$ and applying a crude bound in the integral term.
Choosing $B>\rho+1$, the right-hand side is bounded by
$\exp\{-c t^{-\rho}|\x-\y|\}$, which is smaller than
\begin{equation}
C t^{-k/12}
\ttonde{1+t^{-1/2}|\x-\y|}^{-H}
\end{equation}
for every prescribed $H$. Thus, for every $H>0$,
\begin{equation}\label{eq:coarse-HKUB}
\mathbf P_\x^{(\beta,k)}
\tonde{
|\mathbf R_t-\mathfrak s_\beta t\mathbf 1-\y|\le t^\rho
}
\le
C t^{-k/12}
\ttonde{1+t^{-1/2}|\x-\y|}^{-H}\fstop
\end{equation}

We next turn this coarse estimate into a weighted estimate. Fix $H>k$ and set
\begin{equation}
Q_H(\u)\eqdef (1+|\u|)^{-H}\fstop
\end{equation}
Using the covering inequality
\begin{equation}
Q_H(\u)
\le
C\sum_{\mathbf b\in\Z^k}
Q_H(\mathbf b)\car_{\{|\u-\mathbf b|\le 1\}}\comma
\end{equation}
with
\begin{equation}
\u=t^{-\rho}\tonde{\mathbf R_t-\mathfrak s_\beta t\mathbf 1-\y},
\end{equation}
and applying \eqref{eq:coarse-HKUB} to the centers
$\y+t^\rho\mathbf b$, we obtain
\begin{equation}
\begin{aligned}
\mathbf E_\x^{(\beta,k)}
\sqbra{
Q_H\tonde{
t^{-\rho}(\mathbf R_t-\mathfrak s_\beta t\mathbf 1-\y)
}
}
&\le
C t^{-k/12}
\sum_{\mathbf b\in\Z^k}
Q_H(\mathbf b)\,
Q_H\tttonde{
t^{-1/2}(\x-\y)-t^{\rho-1/2}\mathbf b
}
\\
&\le
C t^{-k/12}
Q_H\tttonde{t^{-1/2}(\x-\y)}\fstop
\end{aligned}
\end{equation}
In the last step we used the elementary convolution bound
\begin{equation}
\sum_{\mathbf b\in\Z^k}
Q_H(\mathbf b)\,Q_H(\u-\alpha\mathbf b)
\le
C Q_H(\u)\comma
\qquad
0<\alpha\le 1\comma
\end{equation}
which holds because $H>k$. Thus,
\begin{equation}\label{eq:base-weighted}
\mathbf E_\x^{(\beta,k)}
\sqbra{
\ttonde{
1+t^{-\rho}|\mathbf R_t-\mathfrak s_\beta t\mathbf 1-\y|
}^{-H}
}
\le
C t^{-k/12}
\ttonde{
1+t^{-1/2}|\x-\y|
}^{-H}\fstop
\end{equation}

We now bootstrap \eqref{eq:base-weighted} down to arbitrarily small mesoscopic
scales. We claim that, for every integer $m\ge 0$, there exists $C_m>0$ such that
\begin{equation}\label{eq:bootstrap-HKUB}
\mathbf E_\x^{(\beta,k)}
\sqbra{
\ttonde{
1+t^{-\rho(2\rho)^m}
|\mathbf R_t-\mathfrak s_\beta t\mathbf 1-\y|
}^{-H}
}
\le
C_m
t^{-\frac{k}{12}\sum_{j=0}^m(2\rho)^j}
\ttonde{
1+t^{-1/2}|\x-\y|
}^{-H}\fstop
\end{equation}
The case $m=0$ is exactly \eqref{eq:base-weighted}.

Assume \eqref{eq:bootstrap-HKUB} holds for $m-1$. Let
$
u=t^{2\rho}$.
By the Markov property at time $t-u$, and using the induction hypothesis for the
inner evolution of length $u$, we get
\begin{equation}
\begin{aligned}
&\mathbf E_\x^{(\beta,k)}
\sqbra{
\ttonde{
1+t^{-\rho(2\rho)^m}
|\mathbf R_t-\mathfrak s_\beta t\mathbf 1-\y|
}^{-H}
}
\\
&\quad =
\mathbf E_\x^{(\beta,k)}
\sqbra{
\mathbf E_{\mathbf R_{t-u}}^{(\beta,k)}
\sqbra{
\ttonde{
1+u^{-\rho(2\rho)^{m-1}}
|\mathbf R_u-\mathfrak s_\beta u\mathbf 1-(\y+\mathfrak s_\beta(t-u)\mathbf 1)|
}^{-H}
}
}
\\
&\quad \le
C_{m-1}
u^{-\frac{k}{12}\sum_{j=0}^{m-1}(2\rho)^j}
\mathbf E_\x^{(\beta,k)}
\sqbra{
\ttonde{
1+u^{-1/2}
|\mathbf R_{t-u}-\mathfrak s_\beta(t-u)\mathbf 1-\y|
}^{-H}
}\fstop
\end{aligned}
\end{equation}
Since $u^{-1/2}=t^{-\rho}$ and $t-u\sim t$, the last expectation is bounded by
\eqref{eq:base-weighted}, applied at time $t-u$. Therefore,
\begin{equation}
\begin{aligned}
&\mathbf E_\x^{(\beta,k)}
\sqbra{
\ttonde{
1+t^{-\rho(2\rho)^m}
|\mathbf R_t-\mathfrak s_\beta t\mathbf 1-\y|
}^{-H}
}
\\
&\qquad\le
C_m
t^{-\frac{k}{12}\sum_{j=1}^m(2\rho)^j}
t^{-k/12}
\ttonde{
1+t^{-1/2}|\x-\y|
}^{-H}
\\
&\qquad=
C_m
t^{-\frac{k}{12}\sum_{j=0}^m(2\rho)^j}
\ttonde{
1+t^{-1/2}|\x-\y|
}^{-H}\fstop
\end{aligned}
\end{equation}
This proves \eqref{eq:bootstrap-HKUB}.

Finally, since
$
2\rho=\frac56
$ and
$\sum_{j=0}^\infty\tonde{\frac56}^j=6$,
we may choose $m$ large enough so that
\begin{equation}
\frac{k}{12}\sum_{j=0}^m(2\rho)^j
\ge
\frac{k}{2}-\delta\fstop
\end{equation}
Taking $H\ge \hslash$ in \eqref{eq:bootstrap-HKUB}, and using
\begin{equation}
\car_{\{|\mathbf R_t-\mathfrak s_\beta t\mathbf 1-\y|\le \sqrt k\}}
\le
C
\tttonde{
1+t^{-\rho(2\rho)^m}
|\mathbf R_t-\mathfrak s_\beta t\mathbf 1-\y|
}^{-H}\comma
\end{equation}
we obtain
\begin{equation}
\mathbf P_\x^{(\beta,k)}
\tttonde{
|\mathbf R_t-\mathfrak s_\beta t\mathbf 1-\y|\le \sqrt k
}
\le
C t^{-k/2+\delta}
\ttonde{
1+t^{-1/2}|\x-\y|
}^{-\hslash}\fstop
\end{equation}
The event
\begin{equation}
\{|R_t^i-\mathfrak s_\beta t-y_i|\le 1,\ \forall\,1\le i\le k\}
\end{equation}
is contained in the event on the left-hand side above, after changing the constant.
This proves the lemma.
\end{proof}

By a union bound and a time-integration, we derive the following two immediate consequences of Lemma \ref{lem:HKUB}.
\begin{corollary}[Triple collisions]\label{cor:HKUB-collisions}
Fix $A>0$, $k\ge 3$, $\delta\in (0,\frac14)$, $0<t_0<T$, and $r\ge 1$. Then, for some $C>0$, and all $\x\in\Z^k$, $N\in \N$, $|\beta|\le AN^{-1/8}$ and $t\in[t_0,TN]$,
\begin{equation}\label{eq:HKUB-two-collisions}
\mathbf P_\x^{(\beta,k)}
(|R_t^i-R_t^j|\le r,\,
|R_t^\ell-R_t^m|\le r)
\le C t^{-1+\delta}\comma
\end{equation}
for all distinct unordered pairs $\{i,j\}\neq\{\ell,m\}$ in $\{1,\ldots,k\}$.
Consequently,
\begin{equation}\label{eq:HKUB-two-collisions-time}
\int_0^{TN}
\mathbf P_\x^{(\beta,k)}(
|R_t^i-R_t^j|\le r,\,
|R_t^\ell-R_t^m|\le r
)   
\,\dd t
\le C N^\delta\fstop
\end{equation}
\end{corollary}

\subsection{Dobrushin-type theorem for $k\ge 2$}

The next result extends Theorem \ref{th:Dobrushin-k=2} from a single gap to all pair gaps of the tilted $k$-point motion.
In what follows, $\mathbf 1\eqdef (1,\ldots,1)\in \R^k$.
\begin{theorem}[Dobrushin-type invariance principle: $k\ge 2$]\label{th:Dobrushin-k>2}
Fix $k\ge 2$.
Let $(\beta_N)_{N\ge1}$ be any sequence such that $\beta_N\to0$ and
$\sup_{N\ge1}\beta_N N^{1/8}<\infty$.
Set $\mathfrak s_N\eqdef \sinh(\beta_N)$.
Let $f,h:\Z\to\R$ be finite-support functions, with $f$ satisfying
$\sum_{x\in\Z}f(x)=0$.
Consider $N^{-1/2}\x^N\to\mathbf z\in\R^k$.
Then, under $\mathbf P_{\x^N}^{(\beta_N,k)}$, the processes
\begin{equation}
	\big(
	N^{-1/2}\tttonde{\mathbf R_{Nt}-tN\mathfrak s_N\mathbf 1}
	\big)_{t\ge0}\comma
	\qquad
	\big(
	\tttonde{
	N^{-1/2}\textstyle{\int_0^{Nt}}h(R_s^i-R_s^j)\,\dd s
	}_{1\le i<j\le k}
	\big)_{t\ge0}\comma
\end{equation}
\begin{equation}
	\big(
	\tttonde{
	N^{-1/4}\textstyle{\int_0^{Nt}}f(R_s^i-R_s^j)\,\dd s
	}_{1\le i<j\le k}
	\big)_{t\ge0}\comma
\end{equation}
jointly converge in law in $\cD([0,\infty);\R^{d_k})$, with $d_k\eqdef k+k(k-1)$, respectively to
\begin{equation}
	\begin{aligned}
		\tonde{
		\mathbf W_t}_{t\ge 0}\comma\qquad
		\big(\tttonde{\gamma^2(h)L_t^{W^i-W^j}}_{1\le i<j\le k}\big)_{t\ge 0}\comma\qquad
		\big(\tttonde{\sigma(f)B^{ij}_{L_t^{W^i-W^j}}}_{1\le i<j\le k}
		\big)_{t\ge0}\fstop
	\end{aligned}
\end{equation}
Here, $\mathbf W=(W^1,\ldots,W^k)$ is a standard $k$-dimensional Brownian motion started from $\z\in\R^k$, the processes $(B^{ij})_{1\le i<j\le k}$ are independent standard Brownian motions independent of $\mathbf W$, and the noise coefficients $\gamma^2(h)$ and $\sigma(f)$ are given, respectively, in \eqref{eq:W+local-time} and \eqref{eq:sigma-squared}.
\end{theorem}

\begin{proof}
Write $\mathbf P^N\eqdef \mathbf P_{\x^N}^{(\beta_N,k)}$, and let $\mathbf P$ denote the joint law of the limiting processes displayed in the statement. 
By the general results of \cite{drillick_parekh_random_2025}, the first marginals satisfy 
\begin{equation}
	\big(
		N^{-1/2}\tonde{\mathbf R_{tN}-tN\mathfrak s_N\mathbf 1}
	\big)_{t\ge 0}
	\xLongrightarrow[N\to \infty]{}
	\tonde{\mathbf W_t}_{t\ge 0}\comma\quad\text{in}\ \cD([0,\infty);\R^k)\fstop
\end{equation}
 Indeed, as shown in
\cite[Appendix A, Step 3 in the proof of Theorem 4.6]{drillick_parekh_random_2025},  $tN\mathfrak s_N\mathbf 1$ may be replaced, up to an error vanishing in
$\cC([0,\infty);\R^k)$, by
\begin{equation}
	\textstyle{\int_0^t}\,\cL^{(\beta_N,k)}\operatorname{id}_{\Z^k}(\mathbf R_s)\,\dd s\fstop
\end{equation}
(The error there is $\cB_N(t)$). Hence, it is enough to prove the
claim with the centered process replaced by the Dynkin martingale (under $\mathbf P^N$)
\begin{equation}
	\mathbf Q_t
	\eqdef
	N^{-1/2}\tonde{
		\mathbf R_t
		-
		\int_0^t \cL^{(\beta_N,k)}
		\operatorname{id}_{\Z^k}(\mathbf R_s)\,\dd s
	}\comma\qquad t \ge 0\fstop
\end{equation}

Fix $1\le i<j\le k$. Let
\begin{equation}
	g_N^{ij}(x_1,\ldots,x_k)
	\eqdef
	(G^N f)(x_i-x_j)\comma
\end{equation}
where $G^N$ is the Green function for the $\beta_N$-tilted two-point gap
process. By Dynkin's formula,
\begin{equation}
	H_t^{ij}=H_t^{ij,N}
	\coloneqq
	g_N^{ij}(\mathbf R_t)
	-
	\int_0^t
	\cL^{(\beta_N,k)}g_N^{ij}(\mathbf R_s)\,\dd s\comma\qquad t \ge 0\comma
\end{equation}
is a $\mathbf P^N$-martingale (with respect to the $\mathbf R$-filtration).
We compare $\cL^{(\beta_N,k)}g_N^{ij}$ with the two-point gap generator. If no
third coordinate is close to either $x_i$ or $x_j$, then the infinitesimal dynamics
of the difference $R^i-R^j$ agrees with that of the tilted
two-point gap process. Consequently, since $G^N f$ solves the corresponding
Poisson equation,
\begin{equation}
	\cL^{(\beta_N,k)}g_N^{ij}(x_1,\ldots,x_k)
	=
	f(x_i-x_j)
\end{equation}
away from triple interactions. Equivalently, there exists $\varrho>0$ such that the
difference
\begin{equation}
	\cL^{(\beta_N,k)}g_N^{ij}(\mathbf x)-f(x_i-x_j)
\end{equation}
can be nonzero only when $x_i$ or $x_j$ lies within distance $\varrho$ of a third
coordinate. By Corollary \ref{cor:HKUB-collisions}, uniformly for $t\ge 0$ in a compact interval,
\begin{equation}
	N^{-1/4}
	\int_0^{tN}
	\tttonde{
		\cL^{(\beta_N,k)}g_N^{ij}(\mathbf R_s)
		-
		f(R_s^i-R_s^j)
	}\,\dd s
	\xLongrightarrow[N\to \infty]{} 0\fstop
\end{equation}

As in the proof of Theorem \ref{th:Dobrushin-k=2}, the functions $g_N^{ij}$ are
uniformly bounded in $N$; moreover, they are constant on the two regions
$\{x_i-x_j>\varrho\}$ and $\{x_i-x_j<-\varrho\}$, after increasing $\varrho$ if
necessary. Therefore, uniformly over compact time intervals,
\begin{equation}
	N^{-1/4}g_N^{ij}(\mathbf R_{tN})
	-
	N^{-1/4}g_N^{ij}(\mathbf R_0)
	\xLongrightarrow[N\to \infty]{} 0\fstop
\end{equation}
 From the martingale identity defining
$H^{ij,N}$, we get, uniformly over compact time intervals,
\begin{equation}
	N^{-1/4}
	\tonde{
		H_{tN}^{ij}
		+
		\int_0^{tN}
		\cL^{(\beta_N,k)}g_N^{ij}(\mathbf R_s^N)\,\dd s
	}
	\xLongrightarrow[N\to \infty]{} 0\fstop
\end{equation}
 Combining the last two
estimates yields, uniformly over compact time intervals,
\begin{equation}
	N^{-1/4}
	\tonde{
		H_{tN}^{ij}
		+
		\int_0^{tN}
		f(R_s^{i,\beta_N}-R_s^{j,\beta_N})\,\dd s
	}
	\xLongrightarrow[N\to \infty]{} 0\fstop
\end{equation} Thus, the desired Dobrushin additive
functional may be replaced by the martingale $-H^{ij}$.

It remains to prove the joint convergence
\begin{equation}
	\tonde{
		\mathbf Q_{tN},
		\tttonde{N^{-1/4}H_{tN}^{ij}}_{1\le i<j\le k}
	}_{t\ge 0}
	\xLongrightarrow[N\to \infty]{}
	\tonde{
		\mathbf W_t,
		\tttonde{\sigma(f)B_{L_t^{W^i-W^j}}^{ij}}_{1\le i<j\le k}
	}_{t\ge 0}\fstop
\end{equation}
 Since the first marginal already
converges to $\mathbf W$, and all components on the left-hand side are martingales,
the convergence of each pair
$(\mathbf Q_{tN},
		N^{-1/4}H_{tN}^{ij}
	)_{t\ge 0}$
follows by the same martingale argument used in the proof of
Theorem \ref{th:Dobrushin-k=2}. In particular, the quadratic variation of
$(N^{-1/4}H^{ij}_{tN})_{t\ge 0}$ converges to
$
	(\sigma(f)L_t^{W^i-W^j})_{t\ge 0}
$,
and its covariation with $\mathbf (Q_{tN})_{t\ge0}$ vanishes.

It remains only to identify the joint law of the limiting martingales for different
pairs. Let $\{i,j\}\ne\{h,\ell\}$. The predictable covariation
\begin{equation}
	\langle
		N^{-1/4}H^{ij},
		N^{-1/4}H^{h\ell}
	\rangle_{tN}
\end{equation}
is supported on configurations where the two pair interactions occur
simultaneously. Such configurations force a triple interaction: at least three
distinct coordinates must lie within distance $\varrho$ of one another. Hence this
covariation is bounded by a constant times
\begin{equation}
	N^{-1/2}
	\int_0^{tN}
	\mathbf 1_{\{\text{three coordinates of }\mathbf R_s\ 
	\text{are within distance }\varrho\}}\,\dd s\fstop
\end{equation}
By Corollary \ref{cor:HKUB-collisions}, this bound converges to zero in probability, uniformly over compact time intervals. Therefore the limiting martingales corresponding to distinct pairs are
orthogonal. The same moment-generating-function argument used in the proof of
Theorem \ref{th:Dobrushin-k=2} then identifies them as conditionally independent
Brownian motions time-changed by the corresponding pair local times. This proves
the claimed joint convergence.
\end{proof}

The next result is the full $k$-moment version of Proposition \ref{pr:second-moment}.

\begin{proposition}[Moment asymptotics: $k\ge 2$]\label{lim_moments}
	Fix $k\ge2$, $t\ge0$, and $\phi_1,\ldots,\phi_k\in\cC_b(\R)$. Set $\lambda=7/8$ and $\beta_N\eqdef\beta_{\lambda,N}$. Let $\x^N=(x_1^N,\ldots,x_k^N)\in\Z^k$ be such that $N^{-1/2}\x^N\to\mathbf z=(z_1,\ldots,z_k)\in\R^k$. Then,
	\begin{equation}\label{eq:k-moment-asymptotics}
	\E\bigg[
		\prod_{j=1}^k
		\mathfrak H_{\beta_N,N}(t,x_j^N,\phi_j)
	\bigg]
	\xrightarrow[N\to \infty]{}
	\mathbf E_{\mathbf z}^{{\rm BM}^{\otimes k}}
	\bigg[
		\exp\bigg\{\frac12\sum_{1\le i<j\le k}L_t^{W^i-W^j}\bigg\}
		\prod_{j=1}^k\phi_j(W_t^j)
	\bigg]\fstop
	\end{equation}
	In particular, when $\x^N=\mathbf 0\in \Z^k$ and $\phi_1=\ldots=\phi_k=\phi\in \cC_b(\R)$, this gives the $k$-th moment asymptotics of $\cY_{7/8,N}(t,\phi)$.
\end{proposition}

\begin{proof}
	By the moment representation in Proposition \ref{cor:moments}, the left-hand side of \eqref{eq:k-moment-asymptotics} equals
	\begin{equation}\label{eq:k-moment-representation-start}
		\mathbf E_{\x^N}^{(\beta_N,k)}
		\bigg[
			\exp\set{\int_0^{tN}V_{\beta_N,k}(\mathbf R_s)\,\dd s}
			\prod_{j=1}^k
			\phi_j\tttonde{N^{-1/2}\tttonde{R_{tN}^j-tN\sinh(\beta_N)}}
		\bigg]\comma
	\end{equation}
	where
	\begin{equation}
		V_{\beta,k}(\x)
		\eqdef
		e^{-F_\beta}(\x)
		\ttonde{\cL^{(k)}-\cL^{\otimes k}}e^{F_\beta}(\x)\fstop
	\end{equation}
	We first compute $V_{\beta,k}$. For $\x\in\Z^k$, write
	\begin{equation}
		n_z(\x)\eqdef \#\{i:\,x_i=z\}\comma\qquad
		z\in\Z\fstop
	\end{equation}
	Then,
	\begin{equation}\label{eq:discrepancy}
		V_{\beta,k}(\x)
		=
		\sum_{z\in\Z}\Psi_\beta(n_z(\x),n_{z+1}(\x))\comma
	\end{equation}
	where, for $n,m\in\N_0$,
	\begin{equation}
		\Psi_\beta(n,m)
		\eqdef
		\tonde{\frac{1+e^\beta}{2}}^n
		\tonde{\frac{1+e^{-\beta}}{2}}^m
		-1
		-
		n\tonde{\frac{e^\beta-1}{2}}
		-
		m\tonde{\frac{e^{-\beta}-1}{2}}\fstop
	\end{equation}
	Indeed, decompose $\cL^{(k)}=\sum_{z\in \Z} \cL_z^{(k)}$ and $\cL^{\otimes k}=\sum_{z\in \Z}\cL_z^{\otimes k}$ according to the bond $(z,z+1)$. If
	\begin{equation}
		S_z(\x)\eqdef\{i:\,x_i\in\{z,z+1\}\}\comma
	\end{equation}
	and $\x^{z,A}$ is obtained from $\x$ by swapping the coordinates with labels in $A\subset S_z(\x)$ across the bond $(z,z+1)$, then
	\begin{equation}
		e^{-F_\beta}(\x)\cL_z^{(k)}e^{F_\beta}(\x)
		=
		2^{-n_z(\x)-n_{z+1}(\x)}
		\sum_{A\subset S_z(\x)}e^{F_\beta(\x^{z,A})-F_\beta(\x)}-1\fstop
	\end{equation}
	If $a$ particles are swapped from $z$ to $z+1$ and $b$ from $z+1$ to $z$, then
	\begin{equation}
		F_\beta(\x^{z,A})-F_\beta(\x)=\beta(a-b)\fstop
	\end{equation}
	Therefore,
	\begin{equation}
		\sum_{A\subset S_z(\x)}e^{F_\beta(\x^{z,A})-F_\beta(\x)}
		=
		(1+e^\beta)^{n_z(\x)}(1+e^{-\beta})^{n_{z+1}(\x)}\fstop
	\end{equation}
	On the other hand,
	\begin{equation}
		e^{-F_\beta}(\x)\cL_z^{\otimes k}e^{F_\beta}(\x)
		=
		\frac{n_z(\x)}2(e^\beta-1)+\frac{n_{z+1}(\x)}2(e^{-\beta}-1)\comma
	\end{equation}
	which proves \eqref{eq:discrepancy}.

	We now expand at $\beta=0$. Since $\Psi_\beta(n,m)=0$ for $n+m\le1$, only occupied bonds with at least two particles contribute. For $n+m\ge2$,
	\begin{equation}\label{eq:Psi-expansion}
		\Psi_\beta(n,m)
		=
		\frac{c_2(n,m)}{2^2 2!}\beta^2
		+
		\frac{c_3(n,m)}{2^3 3!}\beta^3
		+
		\frac{c_4(n,m)}{2^4 4!}\beta^4
		+O_k(\beta^5)\comma
	\end{equation}
	where
	\begin{equation}
	\begin{aligned}
		c_2(n,m)&\eqdef (n-m)^2-(n+m)\comma\\
		c_3(n,m)&\eqdef (n-m)\tonde{(n-m)^2+3(n+m)-4}\comma\\
		c_4(n,m)&\eqdef (n-m)^4+6(n-m)^2(n+m)+3(n+m)^2-10(n+m)\fstop
	\end{aligned}
	\end{equation}
	On configurations without triple collisions, every nontrivial contribution is of the form $(n,m)\in\{(1,1),(2,0),(0,2)\}$. For these values,
	\begin{equation}
		c_2(1,1)=-2\comma\qquad c_3(1,1)=0\comma\qquad c_4(1,1)=-8\comma
	\end{equation}
	and
	\begin{equation}
		c_2(2,0)=c_2(0,2)=2\comma\qquad
		c_3(2,0)=-c_3(0,2)=24\comma\qquad
		c_4(2,0)=c_4(0,2)=56\fstop
	\end{equation}
	The cubic terms cancel pairwise between the bonds adjacent to a double occupation. Hence, after integrating in time and using Corollary \ref{cor:HKUB-collisions} to remove the triple-collision contribution, \eqref{eq:Psi-expansion} gives
	\begin{equation}\label{eq:V-beta-effective}
	\begin{aligned}
		\int_0^{tN}V_{\beta_N,k}(\mathbf R_s)\,\dd s
		&=
		\frac12 N^{-1/4}
		\sum_{1\le i<j\le k}
		\int_0^{tN}\Phi_1(R_s^i-R_s^j)\,\dd s
		\\
		&\quad+
		\frac14 N^{-1/2}
		\sum_{1\le i<j\le k}
		\int_0^{tN}\Phi_2(R_s^i-R_s^j)\,\dd s
		+o_{\mathbf P}(1)\fstop
	\end{aligned}
	\end{equation}
	Here $\Phi_1$ and $\Phi_2$ are the functions in \eqref{eq:f-beta-1-2}. The $O_k((\beta_N)^5)$ remainder is negligible because $(\beta_N)^5\int_0^{tN}\car_{\{\mathbf R_s\in\mathsf C_k\}}\,\dd s=o_{\mathbf P}(1)$ by Lemma \ref{vbound}; the remaining exceptional configurations are negligible by Corollary \ref{cor:HKUB-collisions}.

	We can now apply Theorem \ref{th:Dobrushin-k>2}. Since $\sum_x\Phi_1(x)=0$, $\sum_x\Phi_2(x)=1$, and, as computed in the proof of Proposition \ref{pr:second-moment} (see \eqref{eq:sigma-computation}), $\sigma^2(\Phi_1)={2}$, the exponent in \eqref{eq:V-beta-effective} converges (as a process and jointly with the terminal positions) to
	\begin{equation}\label{f23}
		\sum_{1\le i<j\le k}
		\tonde{
		{\frac1{\sqrt2}}B^{ij}_{L_t^{W^i-W^j}}
		+
		\frac14L_t^{W^i-W^j}
		}\fstop
	\end{equation}
	Moreover,
	\begin{equation}
		\tonde{N^{-1/2}\tttonde{R_{tN}^{1,\beta_N}-tN\sinh(\beta_N)},\ldots,
		N^{-1/2}\tttonde{R_{tN}^{k,\beta_N}-tN\sinh(\beta_N)}}_{t\ge 0}
		\xLongrightarrow[N\to\infty]{}
		\mathbf W\fstop
	\end{equation}
	The exponential moments are uniformly integrable by Lemma \ref{vbound} and the finite-support structure of the functions above. Therefore, the limit of \eqref{eq:k-moment-representation-start} is
	\begin{equation}
	\begin{aligned}
		\mathbf E_{\mathbf z}^{{\rm BM}^{\otimes k}}
		\bigg[
			\mathbf E\bigg[
			\exp\bigg\{
			{\frac1{\sqrt2}}\sum_{1\le i<j\le k}B^{ij}_{L_t^{W^i-W^j}}
			\bigg\}
			\,\bigg|\,
			\mathbf W
			\bigg]
			\exp\bigg\{{\frac14}\sum_{1\le i<j\le k}L_t^{W^i-W^j}\bigg\}
			\prod_{j=1}^k\phi(W_t^j)
		\bigg]\fstop
	\end{aligned}
	\end{equation}
	Conditionally on $\mathbf W$, the Brownian motions $B^{ij}$ are independent. Hence,
	\begin{equation}
		\mathbf E\bigg[
		\exp\bigg\{
		{\frac1{\sqrt2}}\sum_{1\le i<j\le k}B^{ij}_{L_t^{W^i-W^j}}
		\bigg\}
		\,\bigg|\,
		\mathbf W
		\bigg]
		=
		\exp\bigg\{{\frac14}\sum_{1\le i<j\le k}L_t^{W^i-W^j}\bigg\}\fstop
	\end{equation}
	Combining the two local-time contributions gives the coefficient {$1/2$} in \eqref{eq:k-moment-asymptotics}, and the proof is complete.
\end{proof}

\section{Convergence of finite-dimensional distributions}\label{sec:conv-fdd}
This section is devoted to the proof of Proposition \ref{pr:conv-fdd} on convergence of finite-dimensional distributions of the tilted fields. The first key ingredient is the asymptotics of all mixed moments from Proposition \ref{lim_moments}.
The second key input is a recent moment-based characterization \cite[Proposition 1.1]{parekh_moment_2025} of mild/It\^o--Walsh solutions to the multiplicative SHE with general noise coefficient $\gamma>0$, see \eqref{eq:mSHE}. This criterion was strongly inspired by flow-based ideas developed earlier by \cite{tsai2024stochastic} for the $(2+1)$-dimensional SHE. In words, it says that the finite-dimensional laws of the It\^o--Walsh propagator of \eqref{eq:mSHE-theorem} are uniquely determined by independence on disjoint time intervals, the propagator composition rule, and the joint moment formulas.

For later convenience, we report this result below. Here and throughout, if $E\subset \R^d$, $d\ge 1$, is open, we write
$\cC(E)$ for the space of continuous functions on $E$, $\cC_c(E)$ for its
subspace of compactly supported functions, and $\mathscr M(E)$ for the space
of Radon measures on $E$, endowed with the vague topology. We also write
$\cC_c^\infty(E)$ for the space of smooth compactly supported functions on
$E$. Finally,
 $\R_\uparrow^4\eqdef \{(s,t,x,y)\in \R^4: s<t\}$.

\begin{proposition}[{\cite[Proposition 1.1]{parekh_moment_2025}}]\label{pr:mSHE-characterization}
Let $(\cZ_{s,t})_{s\le t}$ be a family of $\mathscr M(\R^2)$-valued random variables indexed by $-\infty <s\le t<\infty$ with $s,t \in \mathbb Q$, all defined on some complete probability space $(\Omega,\cF,\P)$. Fix $\gamma >0$,  and suppose that the following conditions hold true:
\begin{enumerate}[(1)]
    \item\label{it:SHE1} $\cZ_{s_j,t_j}$ are independent under $\P$ whenever $(s_j,t_j)$ are disjoint intervals.
    \item\label{it:SHE2} For all functions $f\in \cC_c^\infty(\R^2)$, $\psi \in \cC_c^\infty(\R)$ with $\int_\R\psi=1$, and indices $s<t<u$, 
    \begin{equation}
    \begin{aligned}
    &\int_{\R^4} \eps^{-1}\psi(\eps^{-1}(y_1-y_2))\, f(x,z)\, \cZ_{t,u}(\dd y_1,\dd z)\cZ_{s,t}(\dd x,\dd y_2)\\
    &\qquad\xrightarrow[\eps\to 0]{\P} \int_{\R^2} f(x,z)\,\cZ_{s,u}(\dd x,\dd z)\comma
    \end{aligned}
    \end{equation}
    where the above convergence holds in $\P$-probability.
    \item\label{it:SHE3} For all integers $k\ge 1$, functions $\phi,\psi \in \cC_c^\infty(\R)$, and indices $s<t$,
    \begin{equation}
    \begin{aligned}
&\E\bigg[\bigg(\int_{\R^2}\phi(x)\psi(y)\,\cZ_{s,t}(\dd x,\dd y)\bigg)^k\bigg]\\
&= \int_{\R^k} \prod_{j=1}^k \phi(x_j)\mathbf E_{(x_1,\ldots,x_k)}^{{\rm BM}^{\otimes k}}\bigg[\exp\bigg\{\gamma^2 \sum_{1\le i<j\le k} L_{t-s}^{W^i-W^j}\bigg\}\prod_{j=1}^k \psi(W_{t-s}^j)\bigg]\dd x_1\cdots \dd x_k\comma
\end{aligned}
    \end{equation}
    where, as before (see, e.g., the text above Proposition \ref{pr:second-moment}), the expectation is with respect to a $k$-dimensional standard Brownian motion $\mathbf W=(W^1,\ldots,W^k)$ started at $\x\in \R^k$, while $L_t$ stands for the local time at the origin in the time interval $[0,t]$ for the process in the superscript.
\end{enumerate}

Then, there exists a $\cC(\R_\uparrow^4)$-valued random element $\tilde \cZ=(\tilde \cZ_{s,t}(x,z))_{(s,t,x,z)\in \R_\uparrow^4}$, defined on the same probability space and satisfying, for every $\psi\in \cC_c(\R)$, $\phi\in \cC_b(\R)$, and $s<t$ in $\Q$, 
\begin{equation}
 \cZ_{s,t}(\psi\otimes \phi)= \int_{\R^2}\psi(x)\phi(z)\,\tilde \cZ_{s,t}(x,z)\,\dd x\dd z\comma\qquad \P\text{-a.s.}\fstop
\end{equation}
Furthermore, on the same probability space, there exists a space-time white noise $\xi$ on $\R^2$, with the property that $(\tilde \cZ_{s,t})_{s\le t}$ is the \textit{propagator} of the multiplicative ${\rm SHE}$ in \eqref{eq:mSHE}, driven by $\xi$ and with noise coefficient $\gamma>0$. More precisely, for all $s,x \in \R$, 
\begin{equation}
(t,y)\in (s,\infty)\times \R\longmapsto \tilde \cZ_{s,t}(x,y)
\end{equation}
is, $\P$-a.s., the It\^o--Walsh solution to the equation in \eqref{eq:mSHE} started from $s$ with initial datum $\delta_x$. Further,
$t\longmapsto \tilde \cZ_{s,t}(x,\emparg)$ may be realized as a process in $\cC([s,\infty),\cM_f(\R))$.

Consequently, for any finite collection of indices $s_1\le t_1, \ldots, s_m\le t_k$ in $\Q$, the finite-dimensional laws $(\cZ_{s_1,t_1},\ldots, \cZ_{s_m,t_m})$ are uniquely determined by the conditions in \ref{it:SHE1}--\ref{it:SHE3}.
\end{proposition}

In the next proposition, $\cY_{\lambda,N}$ are the tilted fields defined in \eqref{eq:tilted-field}, while $\cY$ stands for the It\^o--Walsh solution to the multiplicative SHE in \eqref{eq:mSHE-theorem}, that is, with noise coefficient $\gamma=1/\sqrt 2$ and $\cY_0=\delta_0$. By Proposition \ref{pr:mSHE-characterization}, $\cY$ may be realized in $\cC([0,\infty);\cM_f(\R))$. 
\begin{proposition}\label{pr:conv-fdd}  Fix $\lambda = 7/8$. 	Then,   $\cY_{\lambda,N}\Longrightarrow \cY$ as $N\to \infty$, 
	in the sense of finite-dimensional distributions. More specifically, we have, for all integers $k\ge 1$, times $0\le t_1<\ldots<t_k$, and test functions $\phi_1,\ldots,\phi_k \in \cC_b(\R)$,
	\begin{equation}
		(\cY_{\lambda,N}(t_1,\phi_1),\ldots,\cY_{\lambda,N}(t_k,\phi_k))\xLongrightarrow[N\to \infty]{} (\cY(t_1,\phi_1),\ldots,\cY(t_k,\phi_k))\fstop
	\end{equation}
\end{proposition}
\begin{proof}
Fix $m\ge1$, $0\le t_1<\cdots<t_m$, and
$\phi_1,\ldots,\phi_m\in\cC_b(\R)$. We prove the convergence of this
finite-dimensional marginal. Set
$
	\beta_N\eqdef \beta_{7/8,N}$, and recall that
$
	N\sinh(\beta_N)=N^{7/8}$. 
Finally, set
$
	D\eqdef \Q_{\ge 0}\cup\{t_1,\ldots,t_m\}$.
We divide the proof into steps.

\smallskip\noindent
\textit{Step 1. Tightness in $\mathscr M(\R^2)^{\{ (s,t)\in \Q^2: s\le t\}}$.}
For $0\le s\le t$, and
\begin{equation}
	x\in \Z-sN^{7/8}\comma\qquad
	y\in \Z-tN^{7/8}\comma
\end{equation}
define
\begin{equation}\label{eq:def-Z-st-N}
	Z_{s,t}^N(x,y)
	\eqdef
	K_{sN,tN}(x+sN^{7/8},y+tN^{7/8})\,
	e^{\beta_N\tttonde{y-x+(t-s)N^{7/8}}
		-
		N(t-s)\tttonde{\cosh(\beta_N)-1}
	}\fstop
\end{equation}
This is the centered two-parameter version of the tilted kernel. The exponential
weights are multiplicative, and the kernels $K_{s,t}$ compose. Therefore, for all
$0\le s<t<u$,
\begin{equation}\label{propa}
	\sum_{y\in \Z-tN^{7/8}}
	Z_{s,t}^N(x,y)Z_{t,u}^N(y,z)
	=
	Z_{s,u}^N(x,z)\fstop
\end{equation}

Define $\cZ_{s,t}^N\in \mathscr M(\R^2)$ by
\begin{equation}\label{eq:def-cZ-st-N}
	\cZ_{s,t}^N
	\eqdef
	N^{-1/2}
	\sum_{\substack{x\in \Z-sN^{7/8}\\ y\in \Z-tN^{7/8}}}
	Z_{s,t}^N(x,y)\,
	\delta_{N^{-1/2}x}\otimes \delta_{N^{-1/2}y}\fstop
\end{equation}

For each fixed $0\le s\le t$, the family $(\cZ_{s,t}^N)_{N\ge1}$ is tight in
$\mathscr M(\R^2)$ with the vague topology. Indeed, for every
$f\in\cC_c(\R^2)$ and every integer $k\ge1$, Proposition \ref{lim_moments},
applied to the time interval $[s,t]$, gives
\begin{equation}
	\sup_{N\ge1}
	\E[|\cZ_{s,t}^N(f)|^k]
	<\infty\fstop
\end{equation}
This implies tightness of the real-valued random variables $\cZ_{s,t}^N(f)$ for
every $f\in\cC_c(\R^2)$, and hence tightness of the random measures.

Consequently, for every finite collection of intervals
\begin{equation}
	(s_1,t_1'),\ldots,(s_\ell,t_\ell')\comma\qquad
	0\le s_i\le t_i'\comma\quad s_i,t_i'\in D\comma
\end{equation}
the laws of
$(\cZ_{s_1,t_1'}^N,\ldots,\cZ_{s_\ell,t_\ell'}^N)$
are tight in $\mathscr M(\R^2)^\ell$. Since $D$ is countable, a diagonal
extraction shows that every subsequence has a further subsequence along which
all finite-dimensional marginals of
$(\cZ_{s,t}^N)_{s,t\in D,\, s\le t}$
converge. Let
$	\tttonde{\cZ_{s,t}}_{s,t\in D,\, s\le t}
$ be such a subsequential limit. We prove in Steps 2--4 that every such limit
satisfies the three conditions of Proposition \ref{pr:mSHE-characterization}, with
$\gamma=1/\sqrt2$, restricted to times in $D$. This identifies the subsequential
limit uniquely on the enlarged time set $D$, and in particular at the originally
fixed times $t_1,\ldots,t_m$.

\smallskip\noindent
\textit{Step 2. Verification of Item (1) in Proposition \ref{pr:mSHE-characterization}.}
Let $(s_1,t_1'),\ldots,(s_\ell,t_\ell')$ be disjoint intervals with endpoints in $D$.
Then the prelimit random measures
\begin{equation}
	\cZ_{s_1,t_1'}^N,\ldots,\cZ_{s_\ell,t_\ell'}^N
\end{equation}
depend on disjoint increments of the Poisson environment. Hence, they are
independent. Passing to the subsequential limit preserves independence, and Item
\ref{it:SHE1} follows.

\smallskip\noindent
\textit{Step 3. Verification of Item (2) in Proposition \ref{pr:mSHE-characterization}.}
Fix $s<t<u$ with $s,t,u\in D$, $f\in\cC_c^\infty(\R^2)$, and
$\psi\in\cC_c^\infty(\R)$ with $\int_\R\psi=1$. Write
\begin{equation}
	\psi_\varepsilon(r)
	\eqdef
	\varepsilon^{-1}\psi(\varepsilon^{-1}r)\fstop
\end{equation}
We prove that
\begin{equation}\label{eq:verification-item-2}
\int_{\R^4}
	\psi_\varepsilon(y-w)
	f(x,z)\,
	\cZ_{t,u}(\dd y,\dd z)\cZ_{s,t}(\dd x,\dd w)
	\xrightarrow[\varepsilon\to0]{\P}
	\int_{\R^2}f(x,z)\,\cZ_{s,u}(\dd x,\dd z)\fstop
\end{equation}
It is enough to prove the corresponding convergence along the prelimit in $L^2$,
after taking $N\to\infty$ and then $\varepsilon\to0$.

By \eqref{propa} and \eqref{eq:def-cZ-st-N},
\begin{equation}\label{eq:discrete-propagator-measure}
\begin{aligned}
	\int_{\R^2}f(x,z)\,\cZ_{s,u}^N(\dd x,\dd z)
	&=
	\int_{\R^4}
	N^{1/2}\car_{\{y=w\}}
	f(x,z)\,
	\cZ_{t,u}^N(\dd y,\dd z)\cZ_{s,t}^N(\dd x,\dd w)\fstop
\end{aligned}
\end{equation}
Therefore, the $L^2$ error is
\begin{equation}\label{eq:L2-error-propagator}
	\E\bigg[
	\bigg(
	\int_{\R^4}
	\sqbra{
		\psi_\varepsilon(y-w)
		-
		N^{1/2}\car_{\{y=w\}}
	}
	f(x,z)\,
	\cZ_{t,u}^N(\dd y,\dd z)\cZ_{s,t}^N(\dd x,\dd w)
	\bigg)^2
	\bigg]\fstop
\end{equation}
Expanding the square, \eqref{eq:L2-error-propagator} equals
\begin{equation}\label{eq:L2-error-expanded}
	\E\bigg[
	\int_{\R^8}
	\prod_{j=1}^2 f(x_j,z_j)
	\prod_{j=1}^2
	\sqbra{
		\psi_\varepsilon(y_j-w_j)
		-
		N^{1/2}\car_{\{y_j=w_j\}}
	}
	\prod_{j=1}^2
	\cZ_{t,u}^N(\dd y_j,\dd z_j)
	\cZ_{s,t}^N(\dd x_j,\dd w_j)
	\bigg]\fstop
\end{equation}
We expand the product in square brackets into four terms. For each of these four
terms, Proposition \ref{lim_moments}, applied separately on $[s,t]$ and $[t,u]$,
together with independence of the environment on the two intervals, gives the
limit of the expectation as $N\to\infty$. The factors
\begin{equation}
	N^{1/2}\car_{\{y_j=w_j\}}
\end{equation}
converge, in the Riemann-sum sense, to the diagonal distribution
$\delta_0(y_j-w_j)$. Hence, the limit of \eqref{eq:L2-error-expanded} is
\begin{equation}\label{eq:L2-error-continuum}
	\E\bigg[
	\bigg(
	\int_{\R^4}
	\psi_\varepsilon(y-w)
	f(x,z)\,
	\cU_{t,u}(y,z)\cU_{s,t}(x,w)\,
	\dd y\dd z\dd x\dd w
	-
	\int_{\R^2}
	f(x,z)\,\cU_{s,u}(x,z)\,
	\dd x\dd z
	\bigg)^2
	\bigg]\comma
\end{equation}
where $\cU_{a,b}$ denotes the multiplicative SHE propagator with noise coefficient
$\gamma=1/\sqrt2$.

The continuum propagator satisfies
\begin{equation}
	\int_\R \cU_{s,t}(x,y)\cU_{t,u}(y,z)\,\dd y
	=
	\cU_{s,u}(x,z)\comma
\end{equation}
and the mollifier in \eqref{eq:L2-error-continuum} approximates the diagonal
$y=w$. Therefore, \eqref{eq:L2-error-continuum} vanishes as
$\varepsilon\to0$. This proves \eqref{eq:verification-item-2}, and hence Item
\ref{it:SHE2}.

\smallskip\noindent
\textit{Step 4. Verification of Item (3) in Proposition \ref{pr:mSHE-characterization}.}
Let $k\ge1$, $\phi,\psi\in\cC_c^\infty(\R)$, and $s<t$ with $s,t\in D$. For
$k=1$, the claim follows from the tilted local central limit theorem. For $k\ge2$,
Proposition \ref{lim_moments}, applied to the interval $[s,t]$, gives
\begin{equation}\label{eq:item-3-prelimit}
\begin{aligned}
	&\E\bigg[
	\bigg(
	\int_{\R^2}\phi(x)\psi(y)\,
	\cZ_{s,t}^N(\dd x,\dd y)
	\bigg)^k
	\bigg]
	\\
	&\quad
	\xrightarrow[N\to\infty]{}
	\int_{\R^k}
	\prod_{j=1}^k\phi(x_j)\,
	\mathbf E_{\x}^{{\rm BM}^{\otimes k}}
	\bigg[
		\exp\set{
			\frac12\sum_{1\le i<j\le k}L_{t-s}^{W^i-W^j}
		}
		\prod_{j=1}^k\psi(W_{t-s}^j)
	\bigg]\dd \x\fstop
\end{aligned}
\end{equation}
Since $\gamma=1/\sqrt2$, the exponent in \eqref{eq:item-3-prelimit} is precisely
\begin{equation}
	\gamma^2\sum_{1\le i<j\le k}L_{t-s}^{W^i-W^j}\fstop
\end{equation}
This is Item \ref{it:SHE3}. The required uniform integrability follows from the
same moment estimate, with $k+1$ in place of $k$.

Thus every subsequential limit of
$
	\tttonde{\cZ_{s,t}^N}_{s,t\in D,\, s\le t}
$
satisfies Items \ref{it:SHE1}--\ref{it:SHE3}. By Proposition
\ref{pr:mSHE-characterization}, its law is uniquely identified as the law of the
multiplicative SHE propagator with noise coefficient $\gamma=1/\sqrt2$, restricted
to endpoints in $D$. Hence, since every subsequential limit is uniquely identified,
the full family
$
	\tttonde{\cZ_{s,t}^N}_{s,t\in D,\ s\le t}
$
converges in finite-dimensional distributions to this propagator.

\smallskip\noindent
\textit{Step 5. The point-to-line result.}
It remains to recover the field started from a point mass at the origin. For $t>0$
and $\phi\in\cC_b(\R)$, the definition \eqref{eq:def-Z-st-N} gives
\begin{equation}\label{eq:Y-as-point-start}
	\cY_{7/8,N}(t,\phi)
	=
	\sum_{y\in\Z-tN^{7/8}}
	Z_{0,t}^N(0,y)\,
	\phi(N^{-1/2}y)\fstop
\end{equation}
Equivalently,
\begin{equation}\label{eq:Y-as-discrete-delta}
	\cY_{7/8,N}(t,\phi)
	=
	\int_{\R^2}
	N^{1/2}\car_{\{0\}}(x)\,\phi(y)\,
	\cZ_{0,t}^N(\dd x,\dd y)\comma
\end{equation}
where $N^{1/2}\car_{\{0\}}$ is understood on the support $N^{-1/2}\Z$ of the
first coordinate.

Let $\varrho\in\cC_c^\infty(\R)$ satisfy $\int_\R\varrho=1$, and set
\begin{equation}
	\varrho_\delta^0(x)
	\eqdef
	\delta^{-1}\varrho(\delta^{-1}x)\fstop
\end{equation}
For fixed $\delta>0$, the convergence of the propagators proved in Steps 1--4
applies at every time in $D$, hence in particular at $t_1,\ldots,t_m$. Therefore,
for each $t\in D$,
\begin{equation}\label{eq:smooth-start-conv}
	\int_{\R^2}\varrho_\delta^0(x)\phi(y)\,
	\cZ_{0,t}^N(\dd x,\dd y)
	\xLongrightarrow[N\to\infty]{}
	\int_{\R^2}\varrho_\delta^0(x)\phi(y)\,
	\cZ_{0,t}(\dd x,\dd y)\fstop
\end{equation}
Moreover, by the two-point moment asymptotics in Proposition \ref{lim_moments},
applied with the two initial points in the first coordinate and then letting the
mollifier collapse to the origin,
\begin{equation}\label{eq:delta-approx-prelimit}
	\lim_{\delta\to0}
	\limsup_{N\to\infty}
	\E\bigg[
	\bigg(
		\int_{\R^2}\varrho_\delta^0(x)\phi(y)\,
		\cZ_{0,t}^N(\dd x,\dd y)
		-
		\cY_{7/8,N}(t,\phi)
	\bigg)^2
	\bigg]
	=
	0\fstop
\end{equation}
On the continuum side, Proposition \ref{pr:mSHE-characterization} realizes the
limit propagator through a continuous density $\tilde\cZ_{0,t}$. Thus,
\begin{equation}\label{eq:delta-approx-continuum}
	\int_{\R^2}\varrho_\delta^0(x)\phi(y)\,
	\cZ_{0,t}(\dd x,\dd y)
	=
	\int_{\R^2}\varrho_\delta^0(x)\phi(y)\,
	\tilde\cZ_{0,t}(x,y)\,\dd x\dd y
	\xrightarrow[\delta\to0]{L^2}
	\int_\R \phi(y)\tilde\cZ_{0,t}(0,y)\,\dd y\fstop
\end{equation}
The last term is exactly $\cY(t,\phi)$.

Combining \eqref{eq:smooth-start-conv}, \eqref{eq:delta-approx-prelimit}, and
\eqref{eq:delta-approx-continuum}, first for fixed $\delta$ and then letting
$\delta\to0$, gives the joint convergence
\begin{equation}
	\tonde{
		\cY_{7/8,N}(t_1,\phi_1),\ldots,
		\cY_{7/8,N}(t_m,\phi_m)
	}
	\xLongrightarrow[N\to\infty]{}
	\tonde{
		\cY(t_1,\phi_1),\ldots,\cY(t_m,\phi_m)
	}
	\fstop
\end{equation}
If one of the times is $0$, then both the prelimit and the limit equal the
deterministic value $\phi_i(0)$ in the corresponding component. Since the times
$t_1,\ldots,t_m$ and test functions $\phi_1,\ldots,\phi_m$ were arbitrary, the
finite-dimensional convergence follows.
\end{proof}

\section{Tightness}\label{sec:tightness}

This section is devoted to the proof of the following result. 

\begin{proposition}\label{pr:tightness}The family $(\cY_{7/8,N})_{N\in \N}$ is tight in $\cD([0,\infty);\mathscr M_f(\R))$.
\end{proposition}

We divide the proof of Proposition \ref{pr:tightness} into steps. 	In what follows, $\cC_b^\infty(\R)$ denotes the space of smooth bounded functions on $\R$ whose derivatives of all orders are bounded.

\subsection{Discrete SPDE and martingale decomposition} 
We first record the stochastic equations used throughout this section.
Let $(\omega_t(x))_{t\ge 0,\,x\in\Z}$ be the independent unit-rate Poisson
clocks driving the averaging process, where the clock $\omega(x)$ is attached to
the bond $\{x,x+1\}$. When $\omega(x)$ rings, the two values
$\eta(x)$ and $\eta(x+1)$ are both replaced by their average. For the process
observed on the time scale $tN$, set
\begin{equation}
	\omega_N(t,x)\eqdef \omega_{tN}(x)\comma
	\qquad
	\overline\omega_N(t,x)\eqdef \omega_N(t,x)-Nt\fstop
\end{equation}

Then, the averaging process satisfies the discrete SPDE \cite{sau_tiny_2024}
\begin{equation}\label{eq:avg-discrete-spde}
	\dd \eta_{tN}(x)
	=
	\frac N2\Delta\eta_{tN}(x)\,\dd t
	+
	\frac12\nabla^*
	\tttonde{
		\dd\overline\omega_N(t,\emparg)\,\nabla\eta_{tN^-}
	}(x)\comma
	\qquad x\in \Z\comma
\end{equation}
where
\begin{equation}
	\nabla f(x)\eqdef f(x+1)-f(x)\comma
	\qquad
	\nabla^* f(x)\eqdef f(x)-f(x-1)\comma
\end{equation}
\begin{equation}
	\Delta f(x)\eqdef\nabla^*\nabla f(x) = \nabla \nabla^* f(x)=f(x+1)+f(x-1)-2f(x)\fstop
\end{equation}
Applying $\nabla$ to \eqref{eq:avg-discrete-spde}, we obtain the gradient equation
\begin{equation}\label{eq:gradient-discrete-spde}
	\dd\nabla\eta_{tN}(x)
	=
	\frac N2\Delta\nabla\eta_{tN}(x)\,\dd t
	+
	\frac12\Delta
	\tttonde{
		\dd\overline\omega_N(t,\emparg)\,\nabla\eta_{tN^-}
	}(x)\fstop
\end{equation}
In mild form, this reads as
\begin{equation}\label{eq:gradient-mild}
	\nabla\eta_{tN}(x)
	=
	\nabla \pp_{tN}(x)
	+
	\int_0^t
	\sum_{y\in\Z}
	\frac12\Delta \pp_{(t-s)N}(x-y)\,
	\nabla\eta_{sN^-}(y)\,
	\dd\overline\omega_N(s,y)\fstop
\end{equation}

We now pass to the tilted fields. For $\phi\in \cC_b^\infty(\R)$, write
\begin{equation}
	\phi_{\lambda,N}(t,x)
	\eqdef
	\phi(N^{-1/2}(x-tN^\lambda))\comma
	\qquad t\ge 0\comma x\in\Z\fstop
\end{equation}
Thus,
\begin{equation}
	\cY_{\lambda,N}(t,\phi)
	=
	\sum_{x\in\Z}
	\vartheta_{\lambda,N}(t,x)\,
	\phi_{\lambda,N}(t,x)\,
	\eta_{tN}(x)\fstop
\end{equation}

We now collect a useful martingale decomposition.
\begin{lemma}\label{lem:martingale-decomposition}
	Fix $1/2<\lambda<1$ and $\phi\in\cC_b^\infty(\R)$. Then,
	for every $N\in\N$,
	\begin{equation}\label{eq:martingale-decomposition}
		\cY_{\lambda,N}(t,\phi)
		=
		\phi(0)
		+
		\int_0^t
		\cY_{\lambda,N}
		(s,\tfrac12\partial_z^2\phi)
		\dd s
		+
		\cM_{\lambda,N}(t,\phi)+ \cR_{\lambda,N}(t,\phi)\fstop
	\end{equation}
	Here, 
	$(\cM_{\lambda,N}(t,\phi))_{t\ge0}$ is a martingale (with respect to the natural filtration of $(\eta_{tN})_{t\ge 0}$) with
	predictable quadratic variation given by
	\begin{equation}\label{eq:pred-QV-2}
		\langle\cM_{\lambda,N}(\emparg,\phi)\rangle
		(t)
		=\frac{N^{2\lambda-1}}{4}\int_0^t \varGamma_{\lambda,N}(s,\phi)\,\dd s + \cS_{\lambda,N}(t,\phi)\comma
	\end{equation}
	where	
	\begin{equation}\label{eq:carre-du-champ}
		\varGamma_{\lambda,N}(t,\phi)\eqdef \sum_{x\in \Z} \vartheta_{\lambda,N}(t,x)^2\,[\nabla \eta_{sN}(x)]^2\, \phi_{\lambda,N}(t,x)^2\fstop
	\end{equation}
	The error terms $\cR_{\lambda,N}(t,\phi)$ and $\cS_{\lambda,N}(t,\phi)$ satisfy, for some $C=C(\phi)>0$, 
	\begin{equation}\label{eq:bound-R}
		|\cR_{\lambda,N}(t,\phi)|\le 
		 CN^{\lambda-1}\int_0^t \cY_{\lambda,N}(s,1)\,\dd s	\comma
	\end{equation}
	\begin{equation}\label{eq:bound-S}
		|\cS_{\lambda,N}(t,\phi)|\le C N^{\lambda-1/2}
		\int_0^t
	\varGamma_{\lambda,N}(s,1)
		\,\dd s\fstop
	\end{equation}
\end{lemma}
\begin{proof}
	By \eqref{eq:avg-discrete-spde} and summation by parts ($\nabla$ and $\Delta$ always act on the $x$ variables),
	\begin{align}\label{eq:dY}
		\begin{aligned}
		\dd \cY_{\lambda,N}(t,\phi)
		&=
		\sum_{x\in\Z}
		\tonde{
			\partial_t+\frac N2\Delta
		}
		\tonde{
			\vartheta_{\lambda,N}(t,x)\,\phi_{\lambda,N}(t,x)
		}
		\eta_{tN}(x)\,\dd t
		\\
		&\quad
		+
		\frac12
		\sum_{x\in\Z}
		\nabla\tttonde{
			\vartheta_{\lambda,N}(t,x)\,\phi_{\lambda,N}(t,x)
		}\,
		\nabla\eta_{tN^-}(x)\,
		\dd\overline\omega_N(t,x)\fstop
		\end{aligned}
	\end{align}
A direct computation ensures
	\begin{equation}\label{eq:drift-computation}
		\tonde{
			\partial_t+\frac N2\Delta
		}
		\tonde{
			\vartheta_{\lambda,N}(t,x)\,\phi_{\lambda,N}(t,x)
		}
		=
		\vartheta_{\lambda,N}(t,x)
		\set{
			\frac12(\partial_z^2\phi)_{\lambda,N}(t,x)
			+
			\varepsilon_{\lambda,N}(t,x)
		}\comma
	\end{equation}
	with
	\begin{equation}
		\sup_{t\ge 0}\sup_{x\in\Z}
		\abs{\varepsilon_{\lambda,N}(t,x)}
		\le
		C_1(\phi)\,\tttonde{\beta_{\lambda,N}^2+N^{-1}}\fstop
	\end{equation}
	This gives the drift and error terms in  \eqref{eq:martingale-decomposition}, as well as the bound in
	\eqref{eq:bound-R}.
	
	As for the martingale term, \eqref{eq:dY} yields
	\begin{equation}
		\langle \cM_{\lambda,N}(\emparg,\phi)\rangle(t)=\int_0^t \frac{N}4 \sum_{x\in \Z} [\nabla(\vartheta_{\lambda,N}(s,x)\,\phi_{\lambda,N}(s,x))]^2\,[\nabla \eta_{sN}(x)]^2\, \dd s\comma
	\end{equation}
	where we used the Poisson clocks' independence and  
	$\langle\overline\omega_N(\emparg,x)\rangle(t)=Nt$. Since 
	\begin{align}\label{eq:nabla-theta-phi}
		\nabla(\vartheta_{\lambda,N}(t,x)\,\phi_{\lambda,N}(t,x))&= \vartheta_{\lambda,N}(t,x)\set{\tttonde{e^{\beta_{\lambda,N}}-1}\,\phi_{\lambda,N}(t,x)+ e^\beta\, \nabla \phi_{\lambda,N}(t,x)}\comma
	\end{align}
we  obtain, for some $C_2(\phi)>0$ and all $N$ large enough, uniformly over $t\ge 0$ and $x\in \Z$,	 
	\begin{align}
	\left|\left[\frac{\nabla(\vartheta_{\lambda,N}(t,x)\,\phi_{\lambda,N}(t,x))}{\vartheta_{\lambda,N}(t,x)}\right]^2-[\tttonde{e^{\beta_{\lambda,N}}-1}\,\phi_{\lambda,N}(t,x)]^2\right|
	 \le  C_2(\phi)\,\tttonde{e^{\beta_{\lambda,N}}-1}^2\, N^{1/2-\lambda}\fstop
	\end{align}
Since $\tttonde{e^{\beta_{\lambda,N}}-1}^2\sim N^{2\lambda-2}$, this recovers \eqref{eq:pred-QV-2} and \eqref{eq:bound-S}.
\end{proof}
\begin{remark}\label{rem:martingale} When $\phi \equiv 1$, the quantity in \eqref{eq:drift-computation} vanishes and, thus, 
	\begin{equation}\label{eq:rem-martingales}
		\cY_{\lambda,N}(t,1)= 1 + \cM_{\lambda,N}(t,1)\comma\qquad t \ge 0\fstop
	\end{equation}
	
\end{remark}

\subsection{Technical estimates}
We now estimate the expectation of the  quadratic variation in \eqref{eq:pred-QV-2} when $\phi =1$. Our main result of this part is Corollary \ref{cor:QV-estimate}.

By the gradient equations \eqref{eq:gradient-discrete-spde}--\eqref{eq:gradient-mild} and the identity (cf.\ \eqref{eq:tilt-micro})
\begin{equation}
	\vartheta_{\lambda,N}(t,x)=\vartheta_{\lambda,N}(t-s,x-y)\,\vartheta_{\lambda,N}(s,y)\comma\qquad 0\le s \le t\comma x,y \in \Z\comma
\end{equation} 
we get
\begin{equation}\label{eq:gradient-squared}
	\begin{aligned}
 &\vartheta_{\lambda,N}(t,x)^2\,\E[(\nabla \eta_{tN}(x))^2] = \vartheta_{\lambda,N}(t,x)^2\, [\nabla \pp_{tN}(x)]^2 \\
 &+ \int_0^t \sum_{y\in \Z}N(\tfrac12\Delta \pp_{(t-s)N}(x-y))^2\,\vartheta_{\lambda,N}(t-s,x-y)^2\,\vartheta_{\lambda,N}(s,y)^2\,\E[(\nabla \eta_{sN}(y))^2]\,\dd s\fstop
 \end{aligned}
\end{equation}
Define, for all $t\ge 0$,
\begin{align}
	f_{\lambda,N}(t)&\eqdef  \sum_{x\in \Z} \vartheta_{\lambda,N}(t,x)^2\,\E[(\nabla \eta_{tN}(x))^2]\comma
\\
	g_{\lambda,N}(t)&\eqdef  \sum_{x\in \Z} \vartheta_{\lambda,N}(t,x)^2\,[\nabla \pp_{tN}(x)]^2\comma
\\
	h_{\lambda,N}(t)&\eqdef N\sum_{x\in \Z} (\tfrac12 \Delta \pp_{tN}(x))^2\,\vartheta_{\lambda,N}(t,x)^2\fstop
\end{align}
Then, summing over $x\in \Z$ on both sides of \eqref{eq:gradient-squared}
yields the renewal equation
\begin{equation}\label{eq:renewal-equation}
	f_{\lambda,N}(t)= g_{\lambda,N}(t) + \int_0^t  h_{\lambda,N}(t-s)\, f_{\lambda,N}(s)\, \dd s\fstop
\end{equation}
 Remark that (cf.\ \eqref{eq:carre-du-champ})
\begin{equation}\label{eq:F-carre-du-champ}
	f_{\lambda,N}(t)=\E[\varGamma_{\lambda,N}(t,1)]\fstop
\end{equation}

In the remainder of this section, we estimate time integrals of $f_{\lambda,N}$ (Lemma \ref{lem:F}) through \eqref{eq:renewal-equation} and an analysis of  $h_{\lambda,N}$ and $g_{\lambda,N}$ (Lemmas \ref{lem:q} and \ref{lem:G}, respectively).
\begin{lemma}\label{lem:q}
	For every $\eps\in(0,1)$,
	there exists $C_\eps>0$ satisfying
	\begin{equation}\label{eq:q-bound}
		\int_0^\infty e^{-\rho t}h_{\lambda,N}(t)\,\dd t
		\le
	\frac{1+\eps}2
		+
		C_\eps\frac{N^{4(\lambda-7/8)}}{\sqrt \rho}\comma
	\end{equation}
 for all $\lambda\le 7/8$, $\rho>0$, and $N$ large enough.	
\end{lemma}
\begin{proof}
	Set $\beta=\beta_{\lambda,N}$ and
	\begin{equation}\label{eq:p-beta}
		\mathsf q_\beta(r,x)
		\eqdef
		e^{\beta x-r(\cosh\beta-1)}\pp_r(x)\comma
		\qquad r\ge0\comma x\in\Z\fstop
	\end{equation}
	Then, $\mathsf q_\beta$ is the transition kernel of the continuous-time walk with
	jump rates $e^\beta/2$ to the right and $e^{-\beta}/2$ to the left. Changing variables $r=tN$,
	\begin{equation}
		\int_0^\infty e^{-\rho t}h_{\lambda,N}(t)\,\dd t
		=
		\frac14\int_0^\infty e^{-\rho r/N}
		\sum_{x\in\Z}
		\tonde{e^{\beta x-r(\cosh\beta-1)}\Delta \pp_r(x)}^2\dd r\fstop
	\end{equation}
	Moreover,
	\begin{align}
		e^{\beta x-r(\cosh\beta-1)}\Delta \pp_r(x)
		&=
		e^{-\beta}\mathsf q_\beta(r,x+1)+e^\beta \mathsf q_\beta(r,x-1)-2\mathsf q_\beta(r,x)
		\nonumber\\
		&=
		\Delta \mathsf q_\beta(r,x)
		+
		(e^{-\beta}-1)\nabla \mathsf q_\beta(r,x)
		-
		(e^\beta-1)\nabla \mathsf q_\beta(r,x-1)
		\nonumber\\
		&\quad+
		(e^\beta+e^{-\beta}-2)\mathsf q_\beta(r,x)\fstop
	\end{align}
	Hence, by the elementary inequality $(a+b+c)^2\le (1+\eps)a^2+2(1+\eps^{-1})(b^2+c^2)$, 
	\begin{align}
		\int_0^\infty e^{-\rho t}h_{\lambda,N}(t)\,\dd t
		&\le
		\frac{1+\varepsilon}{4}
		\int_0^\infty e^{-\rho r/N}\|\Delta \mathsf q_\beta(r,\emparg)\|_{\ell^2(\Z)}^2\,\dd r
		\nonumber\\
		&\quad+
		2(1+\eps^{-1})\beta^2
		\int_0^\infty e^{-\rho r/N}\|\nabla \mathsf q_\beta(r,\emparg)\|_{\ell^2(\Z)}^2\,\dd r
		\nonumber\\
		&\quad+
		2(1+\eps^{-1})\beta^4
		\int_0^\infty e^{-\rho r/N}\|\mathsf q_\beta(r,\emparg)\|_{\ell^2(\Z)}^2\,\dd r\fstop
	\end{align}
By Plancherel,
\begin{equation}\label{eq:plancherel}
	\abs{\widehat {\mathsf q}_\beta(r,\xi)}^2
	=
	\exp\set{-2r\cosh\beta\,(1-\cos\xi)}\comma
	\qquad \xi\in[-\pi,\pi]\fstop
\end{equation}
Therefore, using $\widehat{\nabla f}=(e^{i\xi}-1)\widehat f$ and
$\widehat{\Delta f}=-2(1-\cos\xi)\widehat f$,
\begin{align}
	\frac14\int_0^\infty e^{-\rho r/N}
	\|\Delta \mathsf q_\beta(r,\emparg)\|_{\ell^2(\Z)}^2\,\dd r
	&=
	\frac1{2\pi}\int_{-\pi}^{\pi}
	\frac{(1-\cos\xi)^2}
	{\rho/N+2\cosh\beta\,(1-\cos\xi)}
	\,\dd\xi
	\nonumber\\
	&\le
	\frac1{2\cosh\beta}
	\frac1{2\pi}\int_{-\pi}^{\pi}(1-\cos\xi)\,\dd\xi
	\nonumber\\
	&=
	\frac1{2\cosh\beta}\le\frac12\comma
\end{align}
while
\begin{align}
	\int_0^\infty e^{-\rho r/N}
	\|\nabla \mathsf q_\beta(r,\emparg)\|_{\ell^2(\Z)}^2\,\dd r
	&=
	\frac1{2\pi}\int_{-\pi}^{\pi}
	\frac{2(1-\cos\xi)}
	{\rho/N+2\cosh\beta\,(1-\cos\xi)}
	\,\dd\xi
\le 1\comma
\end{align}
and
\begin{align}
	\int_0^\infty e^{-\rho r/N}
	\|\mathsf q_\beta(r,\emparg)\|_{\ell^2(\Z)}^2\,\dd r
	&=
	\frac1{2\pi}\int_{-\pi}^{\pi}
	\frac{\dd\xi}
	{\rho/N+2\cosh\beta\,(1-\cos\xi)}
	\nonumber\\
	&\le
	C_1\int_{-\pi}^{\pi}
	\frac{\dd\xi}{\rho/N+\xi^2}
	\nonumber\\
	&\le
	C_2\sqrt{\frac N\rho}\fstop
\end{align}
Observing that $\beta=\beta_{\lambda,N}\sim N^{\lambda-1}$ as $N\to \infty$ concludes the proof.
\end{proof}

\begin{lemma}\label{lem:G}
 There exists $C>0$ satisfying, for all $1/2<\lambda<1$,
	$T\ge0$ and $N\in\N$,
	\begin{equation}\label{eq:G-bound}
		\int_0^T g_{\lambda,N}(t)\,\dd t
		\le
		C\set{
			N^{-1}
			+
			N^{2\lambda-5/2}\,T^{1/2}
		}\fstop
	\end{equation}
\end{lemma}

\begin{proof}
	Set $\beta=\beta_{\lambda,N}\ge 0$ and $\mathsf q_\beta$ as in \eqref{eq:p-beta}.
	Then,
	\begin{align}
		\vartheta_{\lambda,N}(t,x)\nabla \pp_{tN}(x)
		&=
		e^{-\beta}\mathsf q_\beta(tN,x+1)-\mathsf q_\beta(tN,x)
		\nonumber\\
		&=
		e^{-\beta}\nabla \mathsf q_\beta(tN,x)
		+
		\tttonde{e^{-\beta}-1}\,\mathsf q_\beta(tN,x)\fstop
	\end{align}
	Since $\ttabs{e^{-\beta}-1}\le \beta$, we get
	\begin{equation}
		g_{\lambda,N}(t)
		\le
		2\|\nabla \mathsf q_\beta(tN,\emparg)\|_{\ell^2(\Z)}^2
		+
		2\beta^2\|\mathsf q_\beta(tN,\emparg)\|_{\ell^2(\Z)}^2\fstop
	\end{equation}
	Changing variables $r=tN$ gives
	\begin{equation}
		\int_0^T g_{\lambda,N}(t)\,\dd t
		\le
		2N^{-1}\int_0^{TN}\|\nabla \mathsf q_\beta(r,\emparg)\|_{\ell^2(\Z)}^2\,\dd r
		+
		2\beta^2N^{-1}\int_0^{TN}\|\mathsf q_\beta(r,\emparg)\|_{\ell^2(\Z)}^2\,\dd r\fstop
	\end{equation}
	Recall \eqref{eq:plancherel}.
	Hence, using $1-\cos\xi\asymp \xi^2$ on $[-\pi,\pi]$ and
	$\cosh\beta\ge1$,
	\begin{equation}
		\|\nabla \mathsf q_\beta(r,\emparg)\|_{\ell^2(\Z)}^2
		\le C(1+r)^{-3/2}\comma
		\qquad
		\|\mathsf q_\beta(r,\emparg)\|_{\ell^2(\Z)}^2
		\le C(1+r)^{-1/2}\fstop
	\end{equation}
	Therefore, for some absolute constant $C>0$ (whose value may change from line to line)
	\begin{align}
		\int_0^T g_{\lambda,N}(t)\,\dd t
		&\le
		CN^{-1}\int_0^{TN}(1+r)^{-3/2}\,\dd r
		+
		C\beta^2N^{-1}\int_0^{TN}(1+r)^{-1/2}\,\dd r
		\nonumber\\
		&\le
		CN^{-1}\set{1-(1+TN)^{-1/2}}
		+
		C\beta^2N^{-1}\set{(1+TN)^{1/2}-1}\\
		&\le CN^{-1} + C\beta^2 N^{-1}(TN)^{1/2}\fstop
	\end{align}
	The final conclusion follows because $\beta=\beta_{\lambda,N}\sim N^{\lambda-1}$ as $N\to \infty$.
\end{proof}

\begin{lemma}\label{lem:F} There exists $C>0$ such that, for all $1/2<\lambda \le 7/8$, $T>0$ and $N$ large enough,
	\begin{equation}\label{eq:F-bound}
		\int_0^T f_{\lambda,N}(t)\,\dd t \le C e^{CT}\set{N^{-1}+N^{2\lambda-5/2}\,T^{1/2}}\fstop
	\end{equation}
\end{lemma}
\begin{proof}
	By \eqref{eq:renewal-equation}, we have, for all $\rho\ge 0$, 
	\begin{align}
	 \int_0^T e^{-\rho t}f_{\lambda,N}(t)\, \dd t
		& = \int_0^T e^{-\rho t}g_{\lambda,N}(t)\,\dd t + \int_0^T e^{-\rho t}\int_0^t f_{\lambda,N}(s)\,h_{\lambda,N}(t-s)\,\dd s\,\dd t\\
		&\le \int_0^T e^{-\rho t}g_{\lambda,N}(t)\,\dd t + \int_0^\infty e^{-\rho t} h_{\lambda,N}(t)\,\dd t \int_0^T e^{-\rho t}f_{\lambda,N}(t)\,\dd t\fstop
		\end{align}
	By \eqref{eq:q-bound} and the assumption $\lambda \le 7/8$, there exists $\rho_\ast>0$ satisfying
	\begin{equation}
		\int_0^\infty e^{-\rho_\ast t}h_{\lambda,N}(t)\,\dd t\le 3/4
	\end{equation}
	for all $N$ large enough,
	and, thus,
	\begin{equation}
		\int_0^T e^{-\rho_\ast t}f_{\lambda,N}(t)\,\dd t \le 4 \int_0^T e^{-\rho_\ast t}g_{\lambda,N}(t)\,\dd t\le 4\int_0^T g_{\lambda,N}(t)\,\dd t\fstop
	\end{equation}
	We get the desired estimate by \eqref{eq:G-bound} and
	\begin{equation}
		\int_0^T f_{\lambda,N}(t)\,\dd t \le e^{\rho_\ast T}\int_0^T e^{-\rho_\ast t}f_{\lambda,N}(t)\,\dd t\fstop 	\end{equation}
		This concludes the proof of the lemma.
\end{proof}

We readily obtain the following bound.
\begin{corollary}\label{cor:QV-estimate} Fix $1/2<\lambda \le 7/8$ and $\phi \in \cC_b^\infty(\R)$. Then, there exist $C=C(\phi)>0$ satisfying, for all  $T>0$ and $N$ large enough, 
	\begin{equation}\label{eq:estimate-mart-second-moment}
	\E[(\cM_{\lambda,N}(T,\phi))^2]=	\E[\langle \cM_{\lambda,N}(\emparg,\phi)\rangle(T)]\le C e^{CT} \set{N^{2\lambda-2}+ N^{4(\lambda-7/8)}\,T^{1/2}}\fstop
	\end{equation}
	As a consequence, we have
		\begin{align}\label{eq:time-integral}
		\limsup_{N\to \infty}\int_0^T \E[(\cY_{\lambda,N}(t,\phi))^2]\,\dd t<\infty\fstop
	\end{align}
\end{corollary}
\begin{proof}
	The first claim is a direct consequence of \eqref{eq:pred-QV-2}, \eqref{eq:carre-du-champ}, \eqref{eq:F-carre-du-champ}, and \eqref{eq:F-bound}.

	By  Remark \ref{rem:martingale}, we have, for all  $t\ge 0$,
	\begin{align}
		\E[(\cY_{\lambda,N}(t,\phi))^2]&\le \norm{\phi}_\infty^2\E[(\cY_{\lambda,N}(t,1))^2]= \norm{\phi}_\infty^2\set{1+\E[(\cM_{\lambda,N}(t,1))^2]}\fstop
	\end{align} 
	Hence, 
\eqref{eq:estimate-mart-second-moment} ensures the validity of \eqref{eq:time-integral}.
\end{proof}

\subsection{Tightness of projections}
This section is devoted to  proving tightness for the projections of the tilted fields against smooth test functions.
\begin{lemma}\label{lem:tightness-projections} Fix $\lambda = 7/8$. Then, for every $\phi \in \cC_b^\infty(\R)$, the family $(\cY_{\lambda,N}(\emparg,\phi))_{N\in \N}$ is tight in $\cD([0,\infty);\R)$.	
\end{lemma}
\begin{proof} We omit $\lambda=7/8$ from the notation. As Proposition \ref{pr:conv-fdd} ensures that the family of random variables $(\cY_{N}(t,\phi))_{N\in \N}$ is tight, the desired claim follows from Aldous' tightness criterion via stopping times \cite[Theorem 1]{aldous_stopping_1978}; see also \cite[Theorem 16.10 \& Eq.\ (16.32)]{billingsley_convergence_1999}:
	\begin{enumerate}[(i)]
		\item \label{it:tightness-a} for every $T>0$;
		\item \label{it:tightness-b} for every sequence $(\tau_N)_{N\in \N}$ of stopping times (adapted with respect to the averaging process' natural filtration), a.s.\ bounded by $T$ and taking finitely many values;
		\item \label{it:tightness-c}for every sequence $(\delta_N)_{N\in \N}$ of positive constants in $(0,1)$ satisfying $\lim_{N\to \infty}\delta_N= 0$;
	\end{enumerate}
	one has
	\begin{equation}\label{eq:tightness-proj}
		\cY_N(\tau_N+\delta_N,\phi)-\cY_N(\tau_N,\phi)\xLongrightarrow[N\to \infty]{}0\fstop
	\end{equation}

In view of the decomposition in \eqref{eq:martingale-decomposition}, the bounds in \eqref{eq:bound-R}--\eqref{eq:bound-S}, and applications of Markov and Cauchy-Schwarz inequalities,  \eqref{eq:tightness-proj} follows if we prove, for the quantities in \ref{it:tightness-a}--\ref{it:tightness-c},
\begin{equation}\label{eq:drift-tightness}
\E\bigg[\bigg(\int_{\tau_N}^{\tau_N+\delta_N}\cY_{N}(t,\partial_z^2\phi)\,\dd t\bigg)^2\bigg]\le \delta_N\int_0^{T+1}\E[\cY_{N}(t,\partial_z^2\phi)^2]\,\dd t\xrightarrow[N\to \infty]{}0\comma
\end{equation}
and
\begin{equation}\label{eq:mart-tightness}
\E[(\cM_{N}(\tau_N+\delta_N,\phi)-\cM_{N}(\tau_N,\phi))^2]\xrightarrow[N\to \infty]{}0\fstop
\end{equation}
The claim in \eqref{eq:drift-tightness} follows from \eqref{eq:time-integral}. For the proof of \eqref{eq:mart-tightness}, we follow a soft argument.

By the optional stopping theorem applied to the martingales $(\cM_N(t,\phi))_{t\ge 0}$ and $(\cM_N(t,1))_{t\ge0}$ (adapted with respect to the averaging process' natural filtration) with the stopping times $\tau_N$ as in \ref{it:tightness-b}, we obtain (cf.\ \eqref{eq:bound-S})
\begin{align}\label{eq:mart-phi-1}
	\begin{aligned}
& \E[(\cM_{N}(\tau_N+\delta_N,\phi)-\cM_{N}(\tau_N,\phi))^2]\\
 &\quad=\E\bigg[\int_{\tau_N}^{\tau_N+\delta_N}\frac{N^{2\lambda-1}}{4} \varGamma_{N}(t,\phi^2)\,\dd t\bigg] + \E[\cS_N(\tau_N+\delta_N,\phi)-\cS_N(\tau_N,\phi)]\\
 &\quad\le C_1(\phi)\,{ \E\bigg[\int_{\tau_N}^{\tau_N+\delta_N}\frac{N^{2\lambda-1}}{4} \varGamma_{N}(t,1)\,\dd t\bigg]}\\
&\quad\le C_2(\phi)\, \E[(\cM_N(\tau_N+\delta_N,1)-\cM_N(\tau_N,1))^2]\fstop
\end{aligned}
 \end{align}
Let us take advantage of the identity in Remark \ref{rem:martingale}. In combination with Proposition \ref{pr:conv-fdd} applied with  $\phi_1=\ldots=\phi_k = 1$, we find that there exists an $\R$-valued process $(\cM(t,1))_{t\ge 0}= (\cY(t,1)-1)_{t\ge 0}$ for which:
\begin{enumerate}[(a)]
	\item All finite-dimensional distributions of $(\cM_N(\emparg,1))_{N\in \N}$ converge to those of $\cM(\emparg,1)$;
	\item $\cM(\emparg,1)$ is a.s.\ continuous;
	\item\label{it:c} For each $t\ge 0$,  $\lim_{N\to \infty}\E[(\cM_N(t,1))^4]=\E[(\cM(t,1))^4]<\infty$.
\end{enumerate}
As $\cM_N(\emparg,1)$ is a martingale, by \cite[Proposition 1.2]{aldous_stopping_1989}, these three properties ensure that $(\cM_N(\emparg,1))_{N\in \N}$ converges to $\cM(\emparg,1)$ in $\cD([0,\infty);\R)$. By the a.s.\ continuity of $\cM(\emparg,1)$,  $(\cM_N(\emparg,1))_{N\in \N}$ is actually $\cC$-tight, i.e., tight with respect to the uniform metric; see, e.g., \cite[Chapter 3, Theorem 10.2(a)]{ethier_kurtz_1986_Markov}. Hence, by, e.g., \cite[Chapter VI, Proposition 3.26(ii)]{jacod_shiryaev_limit_2003},
\begin{equation}\label{eq:mart-1}
	\cM_N(\tau_N+\delta_N,1)-\cM_N(\tau_N,1)\xLongrightarrow[N\to \infty]{}0
\end{equation}
holds true
for the quantities in \ref{it:tightness-a}--\ref{it:tightness-c}; more generally, \eqref{eq:mart-1} is valid for any sequence of (uniformly bounded) random times $(\tau_N)_{N\in \N}$.

By Doob's martingale inequality and item \ref{it:c}, we get
\begin{align}
\sup_{N\in \N}\E\bigg[\sup_{t\in [0,T+1]}|\cM_N(t,1)|^4\bigg]\le (4/3)^4\sup_{N\in \N}\E[|\cM_N(T+1,1)|^4]<\infty\fstop
\end{align}
Hence, the random variables $(\cM_N(\tau_N+\delta_N,1)-\cM_N(\tau_N,1))_{N\in \N}$ are uniformly $L^2$-integrable, and, thus, \eqref{eq:mart-1} may be upgraded, for the quantities in \ref{it:tightness-a}--\ref{it:tightness-c}, to
\begin{equation}
	\E[(\cM_N(\tau_N+\delta_N,1)-\cM_N(\tau_N,1))^2]\xrightarrow[N\to \infty]{}0\fstop
\end{equation}
By \eqref{eq:mart-phi-1}, this proves \eqref{eq:mart-tightness}.
\end{proof}
\begin{remark}\label{rem:carre-du-champ}
Proving \eqref{eq:mart-tightness} often goes through the  bound, for some $\eps>0$,
\begin{equation}
\E[(\cM_{N}(\tau_N+\delta_N,\phi)-\cM_{N}(\tau_N,\phi))^2]\le \delta_N^{\eps/(1+\eps)}\, \E\bigg[\int_0^{T+1}\bigg(\frac{N^{2\lambda-1}}{4} \varGamma_{N}(t,1)\bigg)^{1+\eps}\dd t\bigg]^{1/(1+\eps)}\comma
\end{equation}
which is cruder than \eqref{eq:mart-phi-1}. Our argument avoids to estimate this right-hand side.
\end{remark}

\subsection{Proof of Proposition \ref{pr:tightness}} We have now all the ingredients to prove Proposition \ref{pr:tightness}.
\begin{proof}[Proof of Proposition \ref{pr:tightness}]  By, e.g., \cite[Chapter 3, Theorem 9.1]{ethier_kurtz_1986_Markov}, tightness of $(\cY_{7/8,N})_{N\in \N}$ in $\cD([0,\infty);\mathscr M_f(\R))$ follows from the following two inputs:
	\begin{enumerate}[(a)]
		\item \textit{Tightness of projections.} For all $\phi \in \cC_b^\infty(\R)$,  $(\cY_{7/8,N}(\emparg,\phi))_{N\in \N}$ is tight in $\cD([0,\infty);\R)$;
		\item \textit{Compact containment condition.} For all $T>0$ and $\eps >0$, there exists a compact set $\mathscr K_{T,\eps}\subset \mathscr M_f(\R)$ satisfying 
		\begin{equation}
			\liminf_{N\to \infty}  \P(\cY_{7/8,N}(t,\emparg)\in \mathscr K_{T,\eps}\ \text{for all}\ t \in [0,T])>1-\eps\fstop
		\end{equation}
	\end{enumerate}
	The first property was precisely proven  in Lemma \ref{lem:tightness-projections}. 
	It remains to prove the compact containment condition.  The following argument is valid for every $\lambda\le1$.
	
	Fix $\lambda\le1$,  $T>0$ and
	$\eps>0$. For $\alpha\in\R$, set
	\begin{equation}
		\cX_N^\alpha(t)
		\eqdef
		\sum_{x\in\Z}
		\exp\set{\alpha x-tN(\cosh(\alpha)-1)}
		\eta_{tN}(x)\comma
		\qquad t\ge 0\fstop
	\end{equation}
As argued in Remark \ref{rem:martingale}, 
	$(\cX_N^\alpha(t))_{t\ge 0}$ is a non-negative martingale with mean one.

	Let $\beta_N=\beta_{\lambda,N}$ and fix $a>0$. Then, for all $t\le T$,
	\begin{align}
		\cY_{\lambda,N}(t,e^{a\emparg})
		&=
		\exp\set{
			tN\tttonde{\cosh(\beta_N+aN^{-1/2})-\cosh(\beta_N)}
			-atN^{\lambda-1/2}
		}
		\cX_N^{\beta_N+aN^{-1/2}}(t)\comma
	\end{align}
	and the exponent in front is bounded from above by $C(a,T)$, for all
	$N$ large enough. Indeed, writing $\delta_N\eqdef aN^{-1/2}$, Taylor's formula gives
	\begin{align}
		& tN\tttonde{\cosh(\beta_N+\delta_N)-\cosh(\beta_N)}
		-atN^{\lambda-1/2}  \\
		&\quad =
		tN\delta_N\sinh(\beta_N)
		-atN^{\lambda-1/2}
		+
		tN\int_0^{\delta_N}
		(\delta_N-r)\cosh(\beta_N+r)\,\dd r\fstop
	\end{align}
	As $N\sinh(\beta_N)=N^\lambda$, the first two terms cancel exactly. 
 Since $\sup_{N\in \N}\beta_N<\infty$ and
	$\delta_N\to0$, the Taylor remainder satisfies, for all $N$ large enough and some $C=C(a)>0$,
	\begin{equation}
		tN\int_0^{\delta_N}(\delta_N-r)\cosh(\beta_N+r)\,\dd r
		\le
		CTN \delta_N^2 =
		CTa^2\fstop
	\end{equation}
	Hence,
	\begin{equation}
	\sup_{t\in[0,T]}
	\set{
		tN\tttonde{\cosh(\beta_N+aN^{-1/2})-\cosh(\beta_N)}
		-atN^{\lambda-1/2}
	}
	\le C(a,T)\fstop
	\end{equation}
	
	This proves
	\begin{equation}\label{eq:+a}
		\cY_{\lambda,N}(t,e^{a\emparg})\le e^{C(a,T)}\,\cX_N^{\beta_N+aN^{-1/2}}(t)\comma\qquad t\in [0,T]\comma
	\end{equation}
	and the same argument gives
	\begin{equation}\label{eq:-a}
		\cY_{\lambda,N}(t,e^{-a\emparg})
		\le
		e^{C(a,T)}\cX_N^{\beta_N-aN^{-1/2}}(t)\comma
		\qquad t\in[0,T]\fstop
	\end{equation}
	Therefore, by \eqref{eq:+a}--\eqref{eq:-a} and  Doob's inequality for non-negative mean-one martingales,
	\begin{align}
		\P\bigg(
		\sup_{t\in[0,T]}\cY_{\lambda,N}(t,e^{a|\emparg|})>b
		\bigg)
		&\le
		\P\bigg(
		\sup_{t\in[0,T]}\cY_{\lambda,N}(t,e^{a\emparg})>\frac b2
		\bigg)+
		\P\bigg(
		\sup_{t\in[0,T]}\cY_{\lambda,N}(t,e^{-a\emparg})>\frac b2
		\bigg) \\
		&\le
		\frac{4e^{C(a,T)}}{b}\fstop
	\end{align}
	
	Choose $b=b(T,\eps)>0$ so large that the last quantity is smaller than
	$\eps$. Define
	\begin{equation}
		\mathscr K_{T,\eps}
		\eqdef
		\set{
			\nu\in\mathscr M_f(\R):
			\int_\R e^{a|z|}\,\nu(\dd z)\le b
		}\fstop
	\end{equation}
	This set is compact in $\mathscr M_f(\R)$. Indeed, it is closed by the lower
	semi-continuity of $\nu\mapsto \int e^{a|z|}\nu(\dd z)$, and it is tight since 
	$
	\sup_{r>0}	e^{ar}\nu([-r,r]^c)\le b $ holds true for all $\nu\in\mathscr K_{T,\eps}$.
	Thus, by Prokhorov's theorem,
	$\mathscr K_{T,\eps}
	\subset \mathscr M_f(\R)
	$ is compact. Hence,
	\begin{equation}
		\liminf_{N\to\infty}
		\P\big(
		\cY_{\lambda,N}(t,\emparg)\in\mathscr K_{T,\eps}
		\ \text{for all}\ t\in[0,T]
		\big)
		\ge 1-\eps\comma
	\end{equation}
	as desired.
\end{proof}

\section{Proof of Theorems \ref{thm:main} and \ref{th:non-critical}}\label{sec:proof-final}
In this short section, we collect all the ingredients to finish the proof of our two main results from Section \ref{sec:intro}.

\begin{proof}[Proof of Theorem \ref{thm:main}]
By Proposition \ref{pr:conv-fdd}, the finite-dimensional distributions of
$\cY_{7/8,N}$ converge to those of $\cY$. By Proposition \ref{pr:tightness},
the family $(\cY_{7/8,N})_{N\in\N}$ is tight in
$\cD([0,\infty);\mathscr M_f(\R))$. By Proposition \ref{pr:mSHE-characterization}, this identifies the unique possible limit,
thus, proving the theorem.
\end{proof}

\begin{proof}[Proof of Theorem \ref{th:non-critical}]
Fix $t>0$ and $0\neq\phi\in\cC_b^+(\R)$. By \eqref{eq:eta-E} and
\eqref{eq:tilted-LCLT-general},
$
	\E[\cY_{\lambda,N}(t,\phi)]
$
is bounded away from zero and infinity, uniformly in $N$ large. Hence, it is
enough to prove the corresponding second-moment bounds.

For $\lambda<7/8$, the estimate \eqref{eq:estimate-mart-second-moment}, together
with the martingale decomposition in Lemma \ref{lem:martingale-decomposition}, gives
\begin{equation}
	\Var(\cY_{\lambda,N}(t,\phi))\xrightarrow[N\to\infty]{}0\fstop
\end{equation}
For $\lambda>7/8$, set $\lambda=7/8+\eps$, $\eps \in (0,1/8)$. For simplicity of exposition, let us only consider $\phi\equiv 1$; the case with a general test function $0\neq\phi\in \cC_b^+(\R)$ may be dealt with analogously via a $\phi$-dependent tilting, whose details are left to the reader.

Recall \eqref{eq:G-beta-k} and Lemma \ref{lem:two-point-discrepancy}, and observe that, since $\Phi_2\ge 0$, one has, almost surely,
\begin{equation}\label{eq:bound-exponential}
\mathscr E_{\beta_{\lambda,N},2}(T)\ge (\mathscr E_{\beta_{7/8,N},2}(T))^{N^\eps}\comma\qquad T>0\fstop
\end{equation}By the moment representation
\eqref{eq:moment-representation-lambda-specialized} and the bound in \eqref{eq:bound-exponential},  Jensen inequality,  Theorem \ref{th:Dobrushin-k=2} and Proposition \ref{pr:second-moment}, and
\begin{equation}
\mathbf E_{\mathbf 0}^{{\rm BM    }^{\otimes 2}}\bigg[ \exp\set{{\frac12}  L_t^{W^1-W^2} } \bigg]>1\comma
\end{equation}
give
\begin{equation}
	\E[\cY_{\lambda,N}(t,\phi)^2]\ge \mathbf E_{\mathbf 0}^{(\beta_{\lambda,N},2)}\big[\mathscr E_{\beta_{7/8,N},2}(tN)\big]^{N^\eps}\xrightarrow[N\to\infty]{}\infty\fstop
\end{equation}
Since the first moment is constantly equal to one (as we assumed $\phi\equiv 1$), this implies
\begin{equation}
	\Var(\cY_{\lambda,N}(t,\phi))\xrightarrow[N\to\infty]{}\infty\fstop
\end{equation}
The proof is complete.
\end{proof}

\subsection*{Acknowledgments}
We thank Davar Khoshnevisan for pointing out that the original convergence result for mean-zero additive functionals is due to Dobrushin. This material is based upon work supported by the National Science Foundation under Grant No.\ DMS-2424139, while HD was in residence at the Simons Laufer Mathematical Sciences Institute in Berkeley, California, during the Fall 2025 semester. This work was supported in part by the Italian Ministry of Foreign Affairs and International Cooperation, grant number BR26GR05. While
this work was written, FS was associated to INdAM (Istituto Nazionale di
Alta Matematica “Francesco Severi”) and the group GNAMPA.


\end{document}